\documentclass{amsart}          

\usepackage{amsmath,amsthm}     
\usepackage{amssymb,bbm}            
\usepackage{euscript}           
\usepackage{array,enumerate,calc,graphicx}
\usepackage[matrix,arrow,curve,frame]{xy}    
\usepackage[colorlinks]{hyperref}

\xymatrixcolsep{1.9pc}                          
\xymatrixrowsep{1.9pc}
\newdir{ >}{{}*!/-5pt/\dir{>}}                  

\newtheorem{thm}[subsection]{Theorem}
\newtheorem{defn}[subsection]{Definition}
\newtheorem{prop}[subsection]{Proposition}

\newtheorem{cor}[subsection]{Corollary}
\newtheorem{lemma}[subsection]{Lemma}

\theoremstyle{definition}  
\newtheorem{example}[subsection]{Example}

\newtheorem{remark}[subsection]{Remark}

  {\end{list}}

\newcommand{\dfn}{\textbf} 

\newcommand{\mdfn}[1]{\dfn{\mathversion{bold}#1}} 

\newcommand{\iso}               {\cong}  

\newcommand{\cat}{\EuScript}    
\newcommand{\cA}{{\cat A}}      
\newcommand{\cC}{{\cat C}}

\newcommand{\cI}{{\cat I}}

\newcommand{\cM}{{\cat M}}

\newcommand{\cO}{{\cat O}}

\newcommand{\cS}{{\cat S}}

\renewcommand{\cA}{{\mathcal A}}
\renewcommand{\cC}{{\mathcal C}}
\renewcommand{\cI}{{\mathcal I}}

\newcommand{\field}[1]  {\mathbb #1} 

\newcommand{\F}         {\field F}
\newcommand{\R}         {\field R}
\newcommand{\N}         {\field N}
\newcommand{\Z}         {\field Z}

\newcommand{\Q}         {\field Q}

\newcommand{\OO}        {\field O}

\DeclareMathOperator*{\GL}{GL}
\DeclareMathOperator*{\Homeo}{Homeo}
\DeclareMathOperator*{\Invol}{Invol}
\DeclareMathOperator*{\intr}{int}

\DeclareMathOperator*{\im}{Im}

\DeclareMathOperator{\Hom}{Hom}

\DeclareMathOperator{\Sp}{Sp}

\DeclareMathOperator{\Doub}{Doub}

\newcommand{\ra}{\rightarrow}                   
\newcommand{\lra}{\longrightarrow}              

\newcommand{\inc}{\hookrightarrow}              

\newcommand{\blank}{-}                          
\newcommand{\Id}{Id}                            
\newcommand{\id}{id}                            

\newcommand{\bd}{\partial}

\newcommand{\he}{\simeq}

\numberwithin{equation}{section}

\newcommand{\RP}{{\R}P}

\newcommand{\ssim}{\!\!\sim}

\DeclareMathOperator{\esum}{\#_{\mathnormal{2}}}

\newenvironment{myequation}
  {\addtocounter{subsection}{1}\begin{eqnarray}}
  {\end{eqnarray}$\!\!$}

\DeclareMathOperator{\rot}{rot}
\DeclareMathOperator{\anti}{anti}
\newcommand{\antitube}{\text{\rm antitube}}
\newcommand{\bfantitube}{\text{\bf antitube}}
\newcommand{\at}{\text{\rm AT}}
\DeclareMathOperator{\spit}{spit}
\DeclareMathOperator{\triv}{triv}
\DeclareMathOperator{\refl}{refl}

\DeclareMathOperator{\Aut}{Aut}
\DeclareMathOperator{\Iso}{Isom}

\DeclareMathOperator{\tr}{tr}

\DeclareMathOperator{\rank}{rank}

\newcommand{\tD}{\tilde{D}}
\newcommand{\talpha}{\tilde{\alpha}}

\newcommand{\btwo}{[2]}

\newenvironment{bsmallmatrix}
  {\left [\begin{smallmatrix}}
  {\end{smallmatrix}\right ] }

\begin{document}

\title{Involutions on surfaces}

\author{Daniel Dugger}
\address{Department of Mathematics\\ University of Oregon\\ Eugene, OR
97403} 

\email{ddugger@math.uoregon.edu}

\begin{abstract}
We use equivariant surgery to classify all involutions on
closed surfaces, up to isomorphism.  Work on this problem is
classical, dating back to the nineteenth century, but some questions
seem to have been left unanswered.  We give a modern
treatment that leads to a complete classification.
\end{abstract}

\maketitle

\tableofcontents

\section{Introduction}
\label{se:intro}
Let $C_2$ denote the group of order $2$.  
The goal of this paper is to classify all $C_2$-actions on closed
$2$-manifolds, up to equivariant isomorphism.   If $X$ is a closed
$2$-manifold this involves 
\begin{itemize}
\item[(P1)]  counting all of the (isomorphism classes
of) $C_2$-structures on $X$; 
\item[(P2)] developing a nomenclature for
explicitly identifying each $C_2$-structure, and an algorithm for
listing all the  structures on $X$; 
\item[(P3)] providing an algorithm for taking a given $C_2$-action on
$X$ and deciding which element of the list from (P2) represents the
same isomorphism class.  (For practical purposes, this amounts to developing a
set of invariants that is ``complete'' in the sense that it distinguishes
isomorphism classes).  
\end{itemize}
In addition---and this is important---we
 want the nomenclature from (P2) to be
reasonably geometric and to 
lend  itself to the calculation of cohomology groups and other
homotopical invariants.

Our motivation for wanting to solve these problems comes from ongoing
work on trying to understand $RO(G)$-graded Bredon cohomology in the
case $G=C_2$.  Computations in this area are scarce, and we wanted a
supply of basic spaces to use as a testing ground.  It was natural to
start by looking at 2-manifolds, and we originally hoped  there
was a very simple answer to (P1)---(P3) that one could just look up.
The present paper exists because we were unable to find such a
reference.  Several papers in the literature treat significant aspects
of this problem, and it is worth mentioning upfront \cite{Sc},
\cite{S}, \cite{A}, \cite{N1}, \cite{N2}, \cite{BCNS}.  The case of orientable
surfaces is certainly well understood and completely classical, but
the non-orientable case is not.  We give a bit
more history in Section~\ref{se:history} below.

\medskip

We now describe the results in more detail.
An involution on a space $X$ is a map
$\sigma\colon X\ra X$ such that $\sigma^2=\id$.  This is the same as
an action of the group $C_2$ on $X$.  If $(X,\sigma_X)$ and
$(Y,\sigma_Y)$ are spaces with involutions, an isomorphism between them
is a homeomorphism $f\colon X\ra Y$ such that 
$f\circ \sigma_X=\sigma_Y\circ f$.  

Let $T_g$ denote the genus $g$ torus and let $N_r$ denote the
connected sum of $r$ copies of $\RP^2$.  The 
solution to problem (P1) is the following:

\begin{thm}
\label{th:count}
The number of isomorphism classes of $C_2$-actions on $T_g$ is equal to $4+2g$.  
The number of isomorphism classes of $C_2$-actions on $N_r$ is given by the
following formulas:

\[\begin{cases}
1+\frac{(r+3)^3}{64}=\frac{1}{64}\bigl ( r^3+9r^2 +27r+91 \bigr ) &
\text{if $r\equiv 1$\ mod $4$,}\\[0.2in]
1+\frac{(r+1)(r+3)(r+5)}{64}=\frac{1}{64}\bigl ( r^3+9r^2+23r+79\bigr )
& \text{if $r\equiv 3$\ mod $4$,}\\[0.2in]
\frac{1}{64}\bigl ( r^3+18r^2+152r \bigr )
&
\text{if $r\equiv 0$\ mod $4$,}\\[0.2in]
\frac{1}{64}\bigl ( r^3+18r^2+156r -8 \bigr )
&
\text{if $r\equiv 2$\ mod $4$.}\\[0.2in]
\end{cases}
\]
\end{thm}

Of course, merely counting the actions is not our  main goal.  But
Theorem~\ref{th:count}  gives an immediate sense of the
qualitative difference between the orientable and non-orientable
cases.  It also raises some questions.  Why is the count a linear
function of $g$ in the orientable case, but a cubic function of $r$
in the non-orientable case?  Why do the formulas admit a nice
factorization in the case $r$ is odd, but not when $r$ is even?  
This paper contains answers to both, although the {\it ultimate\/} source of
the factorizations is number-theoretic and somewhat of a mystery.  See
Proposition~\ref{pr:A,B} and Remark~\ref{re:factorizations} below.

\begin{remark}[Connection with the mapping class group]
There are relations between (P1)--(P3) and the problem of identifying
conjugacy classes of involutions in the mapping class group of $X$.
But in the end, these are somewhat different problems.  For example,
the 180-degree fixed-point-free
rotation of the torus about the central axis of its doughnut
hole is the identity in the mapping class group, but is a nontrivial
$C_2$-action.  The mapping class group of $S^2$ is $\Z/2$, but there
are four isomorphism classes of $C_2$-actions. For the Klein bottle
the mapping class group is $\Z/2\times \Z/2$, while there are six
$C_2$-actions.  See Section~\ref{se:MCG} for further discussion of
this issue.
\end{remark}

\subsection{Invariants}
Before describing the rest of our results we need a few tools.
There are four easily-obtained invariants for $C_2$-actions on a $2$-manifold
$X$.  The fixed set $X^{C_2}$ will be a union of isolated points and copies of
$S^1$, and the number of such will be denoted by $F$ and $C$,
respectively.  The copies of $S^1$ in $X^{C_2}$ are historically referred
to as {\it ovals\/}.  
A simple closed curve in a $2$-manifold has a normal
bundle, which will be either trivial or nontrivial: in the former case
we call the curve \mdfn{$2$-sided}, and in the latter case
\mdfn{$1$-sided}.
Let $C_+$ and $C_-$ be the number of $2$-sided and $1$-sided ovals
in $X^{C_2}$, so that $C=C_+ + C_-$.  The three numbers  $(F,C_+,C_{-} )$ are the
first three interesting invariants of the action.  

We will let $\beta$ always 
denote $\dim_{\Z/2} H_1(X;\Z/2)$, sometimes instead writing $\beta(X)$ if
necessary.  
When $X$
is non-orientable this is typically called the genus of $X$, but when
$X$ is orientable the word ``genus'' refers to
$\frac{\beta}{2}$.  To correct for this ambiguous terminology we will
refer to $\beta(X)$  as the ``$\beta$-genus'' of $X$.

The following result shows that the invariants $F$, $C_+$, and $C_-$
are constrained by the $\beta$-genus.  This is an old result due to
Scherrer \cite{Sc}, 
but we will give a modern treatment in Section~\ref{se:general} below.
The inequality portion also follows from Smith theory.

\begin{prop}[Scherrer]
\label{pr:Scherrer}
If a closed $2$-manifold $X$ has a nontrivial $C_2$-action then
$F+2C\leq \beta+2$ and $F\equiv C_-\equiv \beta$ (mod $2$).  
\end{prop}

Scherrer's result can be used to help give a heuristic explanation for one of
our questions about Theorem~\ref{th:count}.  When $X$ is orientable
then all the curves in $X$ are $2$-sided, and so of course $C_-=0$.
Even more, it follows by an easy argument (see Lemma~\ref{le:isolated-or-ovals}) 
that either $F=0$ or $C_+=0$: in the
orientable case one cannot have both isolated fixed points and ovals.
Thus, of the three invariants $(F,C_+,C_-)$ only one is ever nonzero
at a time, and Proposition~\ref{pr:Scherrer} constrains this nonzero
invariant by a linear function in $\beta$.  Coarsely speaking, this is
responsible for the linear formula for the number of actions that
appears in
Theorem~\ref{th:count}.

When $X$ is non-orientable the three invariants $(F,C_+,C_-)$ turn out
to be essentially independent, with again
Proposition~\ref{pr:Scherrer} constraining each of them by a linear
function in $\beta$.  Coarsely speaking again, this is responsible for the
number of actions in Theorem~\ref{th:count} being cubic in the $\beta$-genus.

\medskip

In many cases a triple $(F,C_+,C_-)$ satisfying the conditions of
Proposition~\ref{pr:Scherrer} uniquely determines a $C_2$-action, but
in many cases it does not.  We will need more invariants.  The
quotient $X/C_2$ is either orientable or not, and we will call this
the \mdfn{$Q$-sign} of the $C_2$-space $X$ (``Q'' for quotient).  More
precisely, the
$Q$-sign is said to be $+$ (positive) if $X/C_2$ is orientable, and $-$
(negative) otherwise.  

A quadruple $(F,C_+,C_-,Q)$ will be called a \dfn{taxonomy}.  
It turns out to be useful to include $C$ in the notation even
though it is redundant, and sometimes we will want to ignore $Q$.  So
we will write $[F,C:(C_+,C_-)]$ for an ``unsigned taxonomy'' and
$[F,C:(C_+,C_-),Q]$ for a ``signed taxonomy''.  When we merely say
``taxonomy'' we let it be determined from context whether it
is signed or unsigned.  

Unfortunately, it is still not the case that a $C_2$-space is uniquely
determined by its signed taxonomy.  But this does happen in most cases,
and even when it fails it doesn't fail too badly:

\begin{prop}
\label{pr:taxonomy}
Fix a $2$-manifold $X$ and a taxonomy $[F,C:(C_+,C_-),Q]$.  
Up to isomorphism, there are
at most three $C_2$-actions on $X$ having the given taxonomy.  
The only cases where there exist more than one action with a given
taxonomy are where all of the following are true:
\[ \text{ $X\iso N_r$, $r$ is even, $r\geq 4$, $F=C_-=0$, and the
$Q$-sign is negative.}
\]
In these cases things break down as follows:
when $r\geq 4$ there are exactly two actions when $C=0$, three when
$1\leq C\leq \frac{r}{2}-2$, two when $C=\frac{r}{2}-1$, and one when
$C=\frac{r}{2}$.  
\end{prop}

Proposition~\ref{pr:taxonomy} is stated in a relatively weak form, as
it does not indicate which taxonomies correspond to a unique action
as oppposed to no action at all.
The complete answers can be extracted from  Theorem~\ref{th:Tg} and
\ref{th:P2} 
below.  However, an exhaustive statement along these lines requires a
large number of cases and is a bit off-putting.

Proposition~\ref{pr:taxonomy} explains a bit more about the
counting results from Theorem~\ref{th:count} for actions on $N_r$.
Notice that when $r$ is odd there is never more than one action with a
given taxonomy.  The problem of counting possible taxonomies
(which end up being basically---but not quite---all the 
taxonomies satisfying Scherrer's conditions)  
ends up having a
nice solution, described by the factorizations that appear in
Theorem~\ref{th:count}.  This is the answer we are looking for when
$r$ is odd, but not when $r$ is even.  In the latter case it needs to
be modified to take into account the few cases where multiple actions can
have the same taxonomy.

\subsection{Describing  \mdfn{$C_2$}-equivariant 2-manifolds}
We will build up equivariant $2$-manifolds using specific types of
surgery, but explaining these surgeries requires a good deal of
notation.  We let $S^{1,0}$ denote a circle with trivial involution,
$S^{1,1}$ for a circle whose involution is reflection across a
diameter, and $S^1_a$ for a circle with the antipodal involution.
Similarly,
we write $S^2_a$ for a $2$-sphere with antipodal action, $S^{2,1}$ for
a $2$-sphere where the action is reflection across the equatorial
plane, and 
$S^{2,2}$ for a $2$-sphere with involution given by $180$-degree
rotation about an axis.  The reasons for our nomenclature will be given
in Section~\ref{se:background}, but for now let us just accept it.  

There are three basic $C_2$-spaces we will need, two cylinders and one
M\"obius band, as shown in the following diagrams:

\begin{picture}(300,110)(0,-10)
\put(0,0){\includegraphics[scale=0.5]{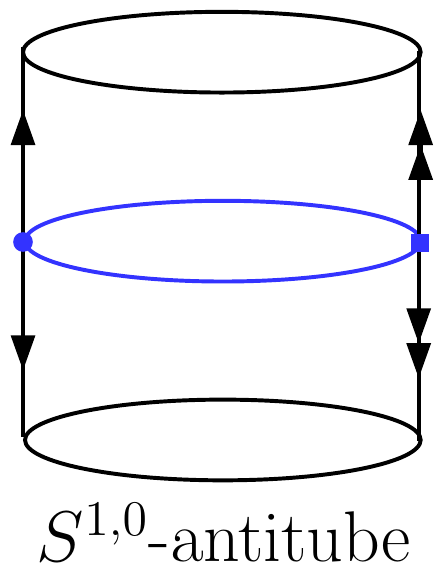}}
\put(120,0){\includegraphics[scale=0.5]{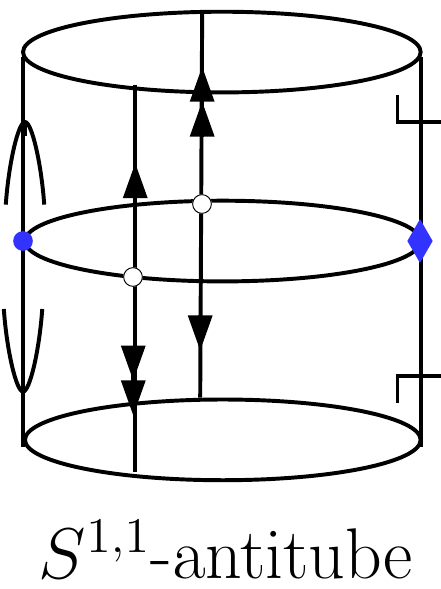}}
\put(232,-5){\includegraphics[scale=0.5]{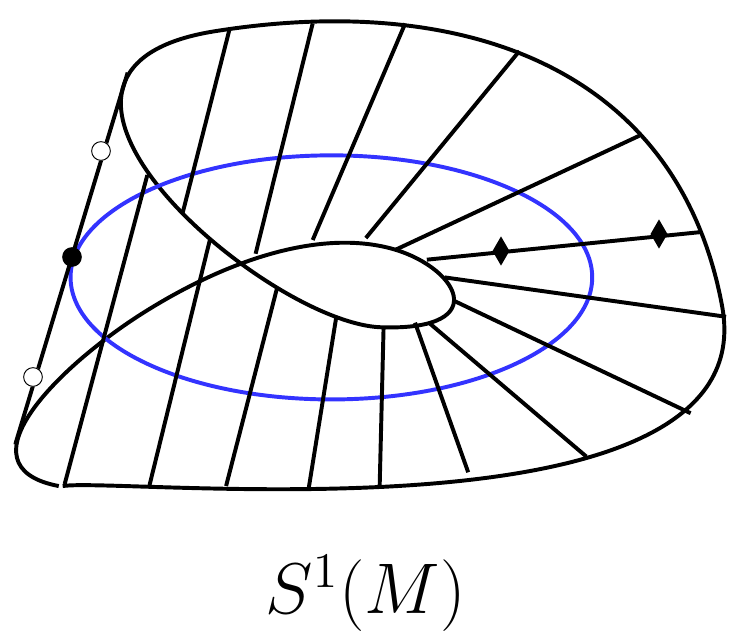}}
\end{picture}

\noindent
In each of these diagrams we depict conjugate points by marking them with the
same symbols.  The fixed set is always shown in blue.  
The ``antitubes'' have the feature that the two ends of
the cylinder are swapped by the involution, whereas in ``tubes'' (not
depicted here) the ends are not swapped.  There are other
$C_2$-actions on cylinders and M\"obius bands that are not shown here;
for a complete list, see Section~\ref{se:cylcap}.

There are five kinds of surgery operations  we will need to perform on 
$C_2$-equivariant $2$-manifolds $X$, three of them using the above spaces:
\begin{enumerate}[]
\item \mbox{} {\bf [DCC]}: 
cut out two disjoint disks, conjugate under the involution,
and sew in conjugate copies of a M\"obius band (that is, form the
connected sum with two conjugate copies of $\RP^2$).  
\item \mbox{} {\bf [DT]}: 
cut out two disjoint disks, conjugate under the involution,
and sew in conjugate copies of a torus (that is, form the
connected sum with two conjugate copies of $T_1$).  
\item \mbox{} {\bf [FM]}: cut out a small neighborhood of an isolated fixed point, leaving an
$S^1_a$ on the boundary, and sew in  a copy of the M\"obius band
$S^1(M)$ (which has $S^1_a$ for its boundary).

\item \mbox{} \mdfn{[$S^{1,0}-\bfantitube$]}: cut out two disjoint disks, 
conjugate under the involution, and sew in an $S^{1,0}$-antitube.
\item \mbox{} \mdfn{[$S^{1,1}-\bfantitube$]}: 
cut out two disjoint disks, conjugate under the involution,
and sew in an $S^{1,1}$-antitube.
\end{enumerate}
We will write $X+[DCC]$, $X+[FM]$, and so on, for the results of performing
these operations on $X$, and will often include multiplicities.  For
example,
\begin{myequation}
\label{eq:surg-pro}
 S^2_a+2[DCC]+3[S^{1,0}-\antitube]+[S^{1,1}-\antitube]+2[FM]
\end{myequation}
is obtained from $S^2_a$ by eight surgery operations.  Note that the
underlying space is $N^{14}$: each [DCC] operation and each antitube
increases the $\beta$-genus by 2, whereas each [FM] operation increases
the $\beta$-genus by 1.  The ``DT'', ``DCC'', and ``FM'' acronyms are silly but
convenient; they stand for ``Dual Tori'', ``Dual CrossCaps'',
and for ``Fixed point $\lra$ M\"obius band''.  The equivariant isomorphism
type of the resulting space does not depend on the order in which the
surgeries are performed, or the choices made in where to
perform them. 
See Sections~\ref{se:background} and \ref{se:general} for more
information.  

\begin{remark}
Be warned that surgery decompositions such as (\ref{eq:surg-pro}) are
not unique.  For example, 
$S^{2,1} + [S^{1,1}-\antitube]\iso S^{2,2}+[S^{1,0}-\antitube]$.  This
leads to a certain amount of hardship when trying to enumerate
all decompositions,
as we have not been able to identify a ``canonical form''
for such things.
\end{remark}

The following result (which is largely self-evident) 
summarizes how we will use the above surgeries to
inductively decompose a $C_2$-equivariant $2$-manifold into pieces of
smaller $\beta$-genus:

\begin{thm}
Let $X$ be a $2$-manifold with $C_2$-action.  
\begin{enumerate}[(a)]
\item If $X$ has two isolated fixed points $a$ and $b$, then there is
a (possibly disconnected)
equivariant $2$-manifold $Y$ and an
isomorphism $X\iso Y+[S^{1,1}-\antitube]$ that carries $a$
and $b$ to the fixed points of the $S^{1,1}$-antitube.   Note that
if $Y$ is connected then $\beta(Y)=\beta(X)-2$.  
\item If $X$ has a two-sided oval $\cO$ then there is a (possibly
disconnected) equivariant $2$-manifold $Y$ and an isomorphism
$X\iso Y+[S^{1,0}-\antitube]$ that carries $\cO$ to the  oval
inside the $S^{1,0}$-antitube.  Again, if $Y$ is connected then
$\beta(Y)=\beta(X)-2$.  
\item If $X$ has a one-sided oval $\cO$ then there is an equivariant
$2$-manifold $Y$ and an isomorphism $X\iso Y+[FM]$ that carries $\cO$
to the oval inside the attached copy of $S^1(M)$.  The space $Y$ has
one more isolated fixed point than $X$, but $\beta(Y)=\beta(X)-1$.  
\end{enumerate}
\end{thm}

\subsection{The orientable case}
We next describe certain special actions on the genus $g$ torus
$T_g$.  
If $T_g$ is embedded in $\R^3$ in the standard way, with the
origin at its center of mass, then the antipodal map $x\mapsto -x$
gives an involution of $T_g$: we call this space $T_g^{\anti}$.  When
$g$ is odd the origin is inside the central doughnut hold of $T_g$,
and $180$-degree rotation about an appropriate axis gives another free
action: we call this one $T_g^{\rot}$.  See Section~\ref{se:cons-free}
 for pictures.

The following diagrams depict two classes of non-free actions on
$T_g$, called the {\it spit\/} action and the {\it reflection\/}
action, respectively:

\begin{picture}(300,135)
\put(0,29){\includegraphics[scale=0.5]{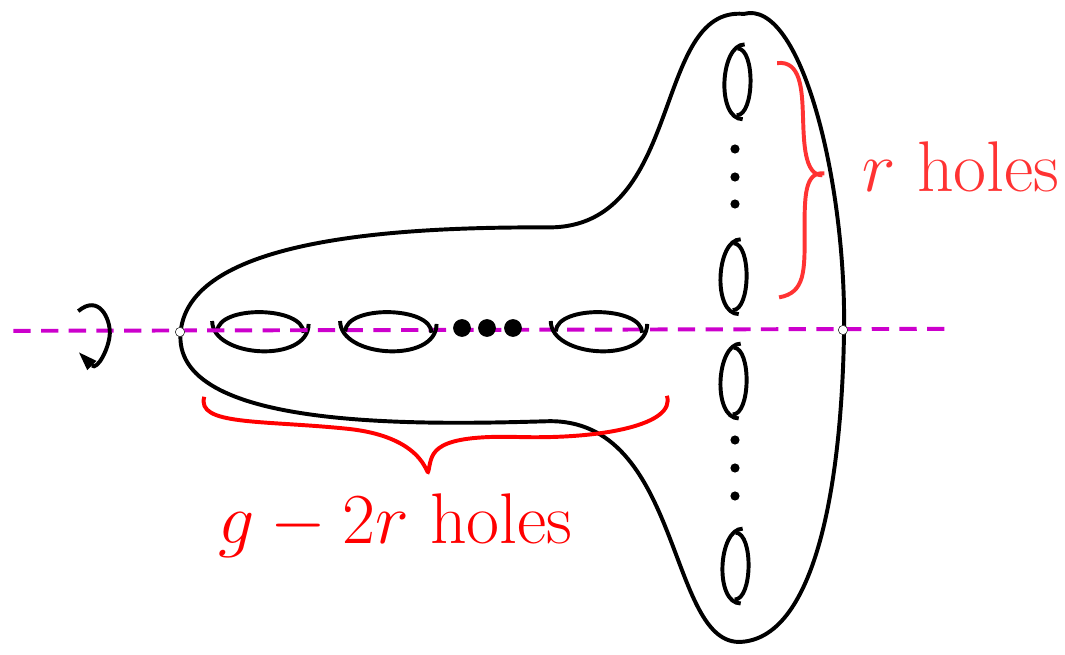}}
\put(180,29){\includegraphics[scale=0.5]{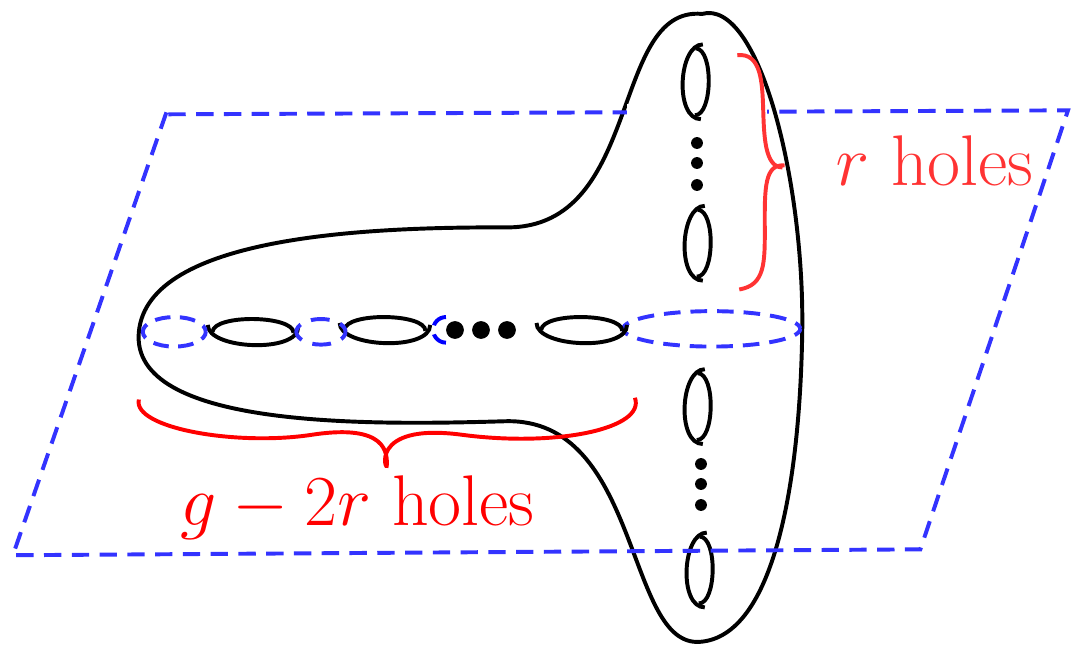}}
\put(10,10){\fbox{$T^{\spit}_g[F]$, \quad $F=2+2g-4r$}}
\put(190,10){\fbox{$T^{\refl}_g[C]$, \quad $C=1+g-2r$}}
\end{picture}

\noindent
The action in the first case 
is 180-degree rotation about the indicated line (the spit),
and in the second it is reflection in the indicated plane.  Note that in the
first case the parameter $F$
denotes the number of fixed points, and in the second case the
parameter $C$ denotes the number of ovals.  
We can also give surgery-based descriptions of these spaces:
\begin{align*}
T_g^{\spit}[F]&\iso
S^{2,2}+(\tfrac{F}{2}-1)[S^{1,1}-\antitube]+\bigl (
\tfrac{2g+2-F}{4}\bigr )[DT]\\
T_g^{\refl}[C]&\iso S^{2,1}+(C-1)[S^{1,0}-\antitube]+\bigl (
\tfrac{g+1-C}{2}\bigr )[DT].
\end{align*}

\vspace{0.1in}

\begin{thm}
\label{th:Tg}
A complete list of isomorphism classes of $C_2$-actions on $T_g$ is as follows:
\begin{enumerate}[(1)]
\item The trivial action,
\item The free actions $T_g^{\anti}$ and $T_g^{\rot}$ (the latter only
when $g$ is odd),
\item $T_g^{\refl}[C]$ for $1\leq C\leq g+1$ and $C\equiv g+1$ (mod
$2$).
\item $T_g^{\spit}[F]$ for $2\leq F\leq 2+2g$ and $F\equiv 2+2g$ (mod
$4$).  
\item $T_{g-C}^{\anti}+C[S^{1,0}-\antitube]$ for $1\leq C\leq g$.  
\end{enumerate}
\end{thm}

This result is proven in Section~\ref{se:orient} 
(see Theorem~\ref{th:T_g-action}).  It is classical, and should
probably be attributed to \cite{K}.  We do not know of a modern
reference for the proof.  Note that counting up all the actions listed
in Theorem~\ref{th:Tg} readily yields the number $4+2g$, both when $g$
is even and $g$ is odd.  

The following chart lists the signed taxonomies for the five classes of
nontrivial $C_2$-actions on $T_g$ (note that in this case, where the
space is orientable, the signed taxonomies have quite a bit of
redundant information):

\vspace{0.1in}

\begin{tabular}{rlll}
$T_g^{\anti}$ & $[0,0:(0,0),-]$  & 
\qquad $T_g^{\rot}$ & $[0,0:(0,0),+]$  \\ 
$T_{g-C}^{\anti}+C[S^{1,0}-\antitube]$ & $[0,C:(C,0),-]$ &
\qquad $T_g^{\spit}[F]$ & $[F,0:(0,0),+]$ \\ 
&&\qquad $T_g^{\refl}[C]$ & $[0,C:(C,0),+]$ \\
\end{tabular}

\vspace{0.1in}

\noindent
By inspection, Theorem~\ref{th:Tg} yields the following answer to
problem (P3)
in the orientable case:

\begin{cor}
\label{co:P3-orient-intro}
Two $C_2$-actions on $T_g$ 
 are isomorphic if and only if they have the
same 
 signed taxonomies.  
\end{cor}

Note that we have now solved problems
(P1)--(P3) for $C_2$-actions on orientable $2$-manifolds.

\begin{remark}
There is another invariant that can be used in place of the $Q$-sign
in the above result.  
It turns out (see Proposition~\ref{pr:conn}) 
that if $C_2$ acts on a connected $2$-manifold $X$
then $X-X^{C_2}$ has either one or two components.  In the latter case
we call the action \dfn{separating}, and in the former case
\dfn{non-separating}.  The spaces $T_g^{\refl}[C]$ are
separating, whereas the spaces $T_{g-C}^{\anti}+C[S^{1,0}-\antitube]$
are non-separating.  Let $\epsilon$ be the invariant whose values are
``separating'' or ``non-separating'' (historically, authors haved
sometimes used $1$ and $0$ here).  

The invariant $\epsilon$ can be used in place of the $Q$-sign in
Corollary~\ref{co:P3-orient-intro}.  It also plays a role in the story
for non-orientable manifolds below.  
\end{remark}

\subsection{The non-orientable case}
The situation for the non-orientable surfaces
$N_r$ is more complicated, in several respects.  Let us start
with an algorithm for listing all the $C_2$-actions:

\begin{itemize}
\item Make a list of all tuples $[F,C:(C_+,C_-)]$ satisfying $F+2C\leq
r+2$ and $F\equiv C_-\equiv r$ (mod $2$). 

\item For tuples with  $F+2C\leq r$:
If $C_->0$ or $F>0$, write down the space of
negative $Q$-sign 
\[ S^2_a+\tfrac{r-F-2C}{2}[DCC]+\tfrac{F+C_-}{2}[S^{1,1}-\antitube] +
C_+[S^{1,0}-\antitube] +C_-[FM].
\]

\item For tuples with $F+2C\equiv r+2$ (mod $4$):  
If $C_->0$ or ($0<F\leq r$ and
$C\geq 1$), write down the space of positive $Q$-sign
\[
T^{\spit}_{\frac{r-C_--2C_+}{2}}[F+C_-]+C_+[S^{1,0}-\antitube]+C_-[FM]
\]

\item If $F=C_-=0$ and $0<C\leq \frac{r}{2}$ and $2C\equiv r+2$ (mod
$4$),
write down the space of positive $Q$-sign
\[ T^{\rot}_{\frac{r}{2}-C}+C[S^{1,0}-\antitube].
\]
\item If $F=C_-=0$ and $0\leq C\leq \frac{r}{2}$, write down the following
spaces of negative $Q$-sign:

\vspace{0.1in}

\begin{tabular}{ll}
 $S^{2,1}+(\tfrac{r}{2}-C+1)[DCC]+(C-1)[S^{1,0}-\antitube]$ & \text{if
   $1\leq C$,} \\
 $S^2_a+(\tfrac{r}{2}-C)[DCC]+C[S^{1,0}-\antitube]$ & \text{if
  $C<\frac{r}{2}$,} \\
 $T_1^{\anti}+(\tfrac{r}{2}-C-1)[DCC]+C[S^{1,0}-\antitube]$ & \text{if
  $C<\frac{r}{2}-1$.} 
\end{tabular}

\vspace{0.1in}

\noindent
Note that when $1\leq C\leq \frac{r}{2}-2$ all three spaces are added to the
list.  
\end{itemize}

This algorithm is completely mechanical, and easy to implement on a
computer.  The tables in Appendix B list the results for $N_r$ where
$2\leq r\leq 7$.  Note that the last two steps of the algorithm
only occur in the case where $r$ is even.  

\begin{thm}
\label{th:P2}
For $r\geq 1$ the above algorithm gives a complete list of all
$C_2$-actions on $N_r$ up to equivariant isomorphism, with no action
being represented more than once on the list.
\end{thm}

At this point we have described the solutions to problems (P1) and
(P2) for $N_r$.  We {\it almost\/} have a complete solution to (P3), since we
know that
except for the cases where $F=C_-=0$ there is at most one
isomorphism class of $C_2$-action with a given signed taxonomy.
In the exceptional case where $F=C_-=0$, we need a way of deciding
which of the two or three isomorphism classes in the above list
corresponds to a given
action.  This can be done by adding one last invariant into the list.

If $X$ is a $2$-manifold with involution $\sigma$, then
$\sigma^*\colon H^1(X;\Z/2)\ra H^1(X;\Z/2)$ is an involution on
cohomology.  The cup product equips $H^1(X;\Z/2)$ with a nondegenerate
symmetric bilinear form; let $\Iso(H^1(X;\Z/2))$ denote the group of
isometries.  Then $\sigma^*$ is an involution in $\Iso(H^1(X;\Z/2))$.
It is shown in \cite{D} that conjugacy classes of involutions in such
an isometry group
are classified by the \dfn{double Dickson invariant} 
\[ DD(\sigma)\in \N\times \Z/2\times \N\times \Z/2.
\]
We recall the complete definition in Section~\ref{se:DD}, but for now
let us just give a bit of the main idea.  
For an involution $\sigma$ on a finite-dimensional vector space $V$
over $\Z/2$, define the \mdfn{$D$-invariant} of $\sigma$ to be the
dimension of $\im(\sigma+\Id)$.  This is an invariant of the
conjugacy class of $\sigma$ in $\GL(V)$, and it is in fact a complete
invariant: two involutions in $\GL(V)$ are conjugate if and only if
they have the same $D$-invariant.

When  $V$ has a nondegenerate, 
symmetric bilinear form and we replace $\GL(V)$ by $\Iso(V)$, further
invariants are needed.  The first coordinate of $DD(\sigma)$ is just
$D(\sigma)$, but the other three coordinates are defined in ways that 
make use of the bilinear form.  Again, see Section~\ref{se:DD} for the
detailed definition.  The main thing to know right now is that all the
coordinates of $DD(\sigma)$  are
algebraically computed by linear algebra.  

If $(X,\sigma)$ is a surface with involution then define
$DD(X)=DD(\sigma^*)$, where $\sigma^*$ is the induced map in $\Iso( H^1(X;\Z/2))$.
The following theorem is our solution to problem (P3):

\begin{thm}
Fix $r\geq 1$.  
\begin{enumerate}[(a)]
\item
The signed taxonomy $[F,C:(C_+,C_-),Q]$ gives a complete
invariant for $C_2$-actions on orientable surfaces: two actions on an orientable surface $X$ are isomorphic if
and only if they have the same signed taxonomy.
\item If $F+C_->0$ or if $Q$ is positive then any two $C_2$-actions on the non-orientable
surface $N_r$ having taxonomy
$[F,C:(C_+,C_-),Q]$ are isomorphic.  
\item Suppose $X$ and $Y$ are two $C_2$-actions on $N_r$ having
taxonomy $[0,C:(C,0),-]$.  Then $X$ and $Y$ are isomorphic if and only
if they have the same $\epsilon$-invariant and the same
$DD$-invariant.
\end{enumerate}
In summary, the invariants $F$, $C_+$, $C_-$, $Q$, $\epsilon$, and
$DD$---in addition to the purely topological invariant $H_1(X;\Z)$---constitute
a complete set of invariants for $C_2$-actions on surfaces.  
\end{thm}

\subsection{History and apology}
\label{se:history}
The classification problem for $C_2$-actions on surfaces has been
long studied.
When $X$ is orientable the problem first arose in the
nineteenth century, in connection with interest in real algebraic
geometry.  Some results were obtained by Harnack \cite{H}, and later
there was a more complete study by Klein
in the context of his interest in dianalytic surfaces: see
\cite{K} and also the earlier PhD thesis \cite{W} of Weichold (a
student of Klein).  
In this case there are relatively few
$C_2$-actions, and it is easy to describe them all.  The case where $X$
is non-orientable has more challenges, and seems to have first been
addressed by Scherrer \cite{Sc} in 1929.  However, Scherrer's paper
does not explicitly solve (P1)--(P3).   

After Scherrer there is a long hiatus in the literature, with the topic
being picked up again in the 1980s by Natanzon \cite{N1},
\cite[Section 6]{N2}.  The difficulty
here is that the classification results in these sources are somewhat
unwieldy, and still don't seem to solve (P1)--(P3).    
The paper \cite{BCNS} contains an improved classification, but with
two caveats.  First, the results are developed in the context of
dianalytic surfaces and non-Euclidean crystallographic groups, and as
a consequence there is a certain lack of geometric simplicity.  Also,
it seems again that (P1)--(P3) are not {\it explicitly\/} solved,
although certainly the germs of a solution are contained in those
papers.  The discussion at the end of the introduction of \cite{BCNS}
is close to our Proposition~\ref{pr:taxonomy}, 
but the cases where $F=C_-=0$ are omitted.

This is not to criticize any of the aforementioned papers.
In each case the authors were interested
in deeper and more complicated problems: generalizing from
involutions to periodic automorphisms of higher order, or from closed
$2$-manifolds to $2$-manifolds with boundary, or from a purely topological
problem to a more geometric one.  The unfortunate end result, though, is
that if a topologist wants to know all the ways $C_2$ can
act on a closed non-orientable surface, there doesn't seem to be a
place in the literature where the answer is completely explained.
It is also worth pointing out that the $DD$-invariant, which provides
the last piece to the classification puzzle, seems to have been
completely overlooked.

Finally, we close with a stylistic comment.  The road to understanding
$C_2$-actions is to first classify the free actions, then to
understand the actions on orientable surfaces, and ultimately to
understand the non-orientable case.  Each depends on knowledge
gained in the previous stages.  While the first two stages are fairly
well-documented in the literature, 
bringing all the techniques together in one place results in a more
coherent narrative.  This paper attempts to provide such a narrative,
although there is a resulting cost to brevity.

\subsection{Organization of the paper}
Sections 2 and 3 contain background information about
$C_2$-equivariant spaces and equivariant surgery constructions.  The
real work begins in Section 4, where we classify the free actions on
$2$-manifolds.  Section 5 explores the non-free $C_2$-actions, but
concentrating on {\it orientable\/} $2$-manifolds.  As a prelude to
the non-orientable case, Section 6 deals with general questions about
invariants of $C_2$-actions and their behavior under surgery.  Then
Section 7 completes the solution of problem (P2) for  actions
on non-orientable manifolds.  Section 8 counts the actions, thereby
solving (P1).  Section 9 introduces the $DD$-invariant and uses this
to complete the classification by solving (P3).  Section 10 briefly
discusses the connection with order 2 elements of the mapping class
group.  Finally, there are two appendices.  Appendix A gives proofs
for the fundamental surgery theorems in the $C_2$-equivariant context,
and Appendix B consists of tables listing all $C_2$-actions on the
non-orientable surfaces $N_r$ for $2\leq r\leq 7$.

\subsection{Notation and terminology}
Whenever $X$ is a $C_2$-space we will use $\sigma$ to denote the
involution $X\ra X$.  Also, by ``$2$-manifold'' we always mean a
connected, {\it
  closed\/} $2$-manifold unless otherwise indicated.  For convenience
we work in the smooth category.

\subsection{Acknowledgments}
The pictures in this paper were constructed using the software package
LaTeXDraw, which I
learned about from Paulo Lima-Filho and Pedro dos Santos.  I am 
grateful to them for sharing code from the pictures in \cite{LFdS}.

Eric Hogle and Clover May read through a preliminary version of this
paper and gave several important comments.  Additionally, their
knowledge of $C_2$-actions on tori and Klein bottles was crucial for
getting these results off the ground.  Finally, I am grateful to
Robert  Lipshitz for some useful conversations and encouragement.

\section{Background and basic constructions}
\label{se:background}

This section  describes the basic notation and constructions that will be used
throughout the paper.

\medskip

\subsection{Equivariant spheres}
Write $\R$ and $\R_-$ for the trivial and sign actions of $C_2$ on the
real number line.  For $p\geq q$  
let $\R^{p,q}=\R^{\oplus(p-q)}\oplus \R_-^{\oplus q}$, so that
$\R^{p,q}$ is a $p$-dimensional real vector space with a $C_2$-action
whose fixed set is $(p-q)$-dimensional.  This indexing convention
comes from the world of motivic and equivariant homotopy, though
in the present paper we will only
need the case where $p\in \{1,2\}$.

Write $S^{p,q}$ for the one-point compactification of $\R^{p,q}$.  
Note that $S^{1,0}$ is a circle with trivial $C_2$-action, whereas
$S^{1,1}$ is a circle whose action is reflection across a
diameter.  Similarly, we have 
the equivariant $2$-spheres $S^{2,0}$, $S^{2,1}$, and $S^{2,2}$.  
Finally, let $S^p_a$ denote a $p$-sphere with the antipodal action.

\subsection{Non-orientable surfaces}
It will be useful to establish some language and notation for talking about
surfaces.  
Let $X$ be a closed surface, and let $D$ be an embedded disk.  
Then $X-\intr(D)$ has a circle for its boundary, and identifying
antipodal points on this circle produces a new surface $X'$.  This process is called adding a
\dfn{crosscap} to $X$.  Note that this is equivalent to sewing  a
M\"obius band onto $X-\intr{D}$, with the boundary of the M\"obius band
wrapping once around $\bd D$.  One also has  
$X'\iso X\# \RP^2$.   We depict a crosscap as a circle with a cross 
inside of it; for example, the left picture below shows a torus with a
crosscap:

\begin{picture}(300,100)(0,40)
\put(-10,0){\includegraphics[scale=0.5]{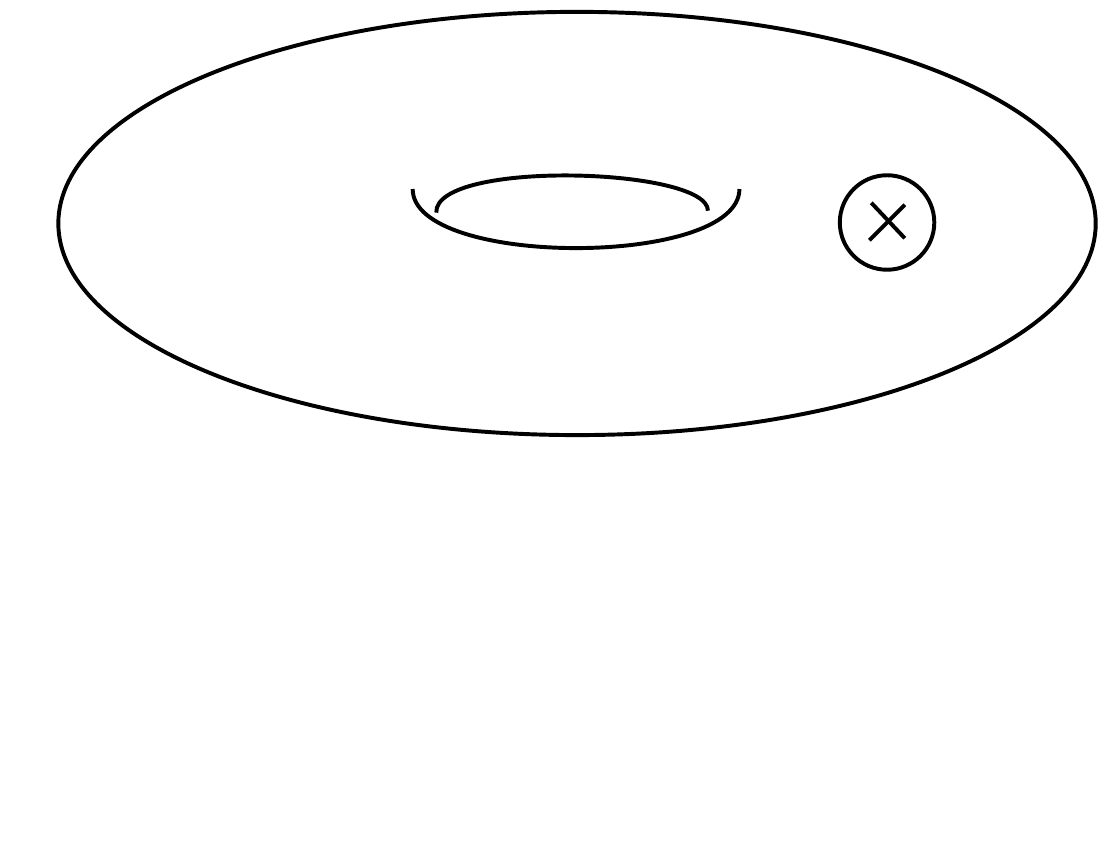}}
\put(180,0){\includegraphics[scale=0.5]{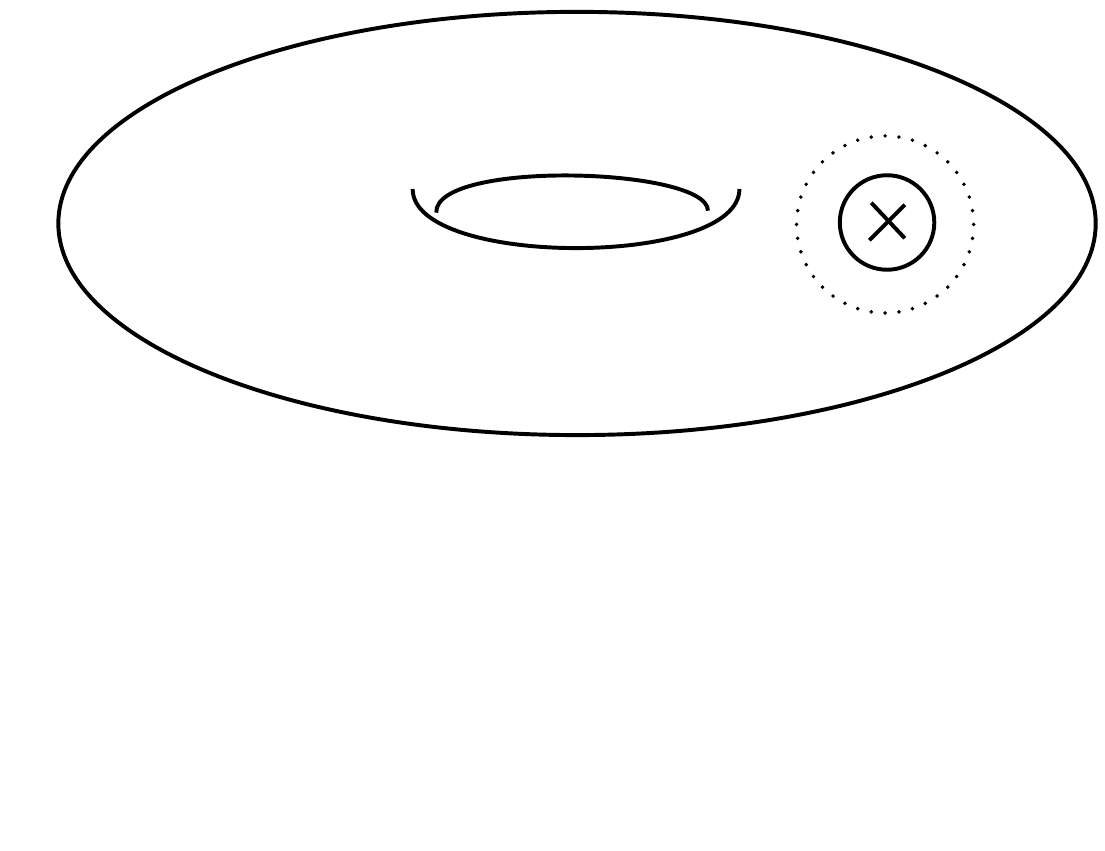}}
\end{picture}

\noindent
The circle depicts a hole, and the cross in the middle reminds us that
opposite points on the hole are identified.  To see the M\"obius
band in the picture, take a tubular neighborhood of the original 
circle---shown in the
picture on the right.  We then see an annulus with the opposite points
on its inner circle identified, and this is our M\"obius band.

Let $N_r$ denote the non-orientable surface of genus $r$,
i.e. $N_r=(\RP^2)^{\# r}$.  A useful model for $N_r$ is a $2$-sphere
with $r$ crosscaps.

\begin{remark}
\label{re:addcross}
An easy Euler characteristic argument shows that adding a crosscap
increases the $\beta$-genus by one, and of course it turns an
orientable space into a non-orientable one.
So 
$
 T_g \ + \ (\text{$r$ crosscaps}) \iso 
N_{2g+r}$.  
\end{remark}

\subsection{Equivariant connected sums}
Let $X$ be a $2$-manifold with nontrivial $C_2$-action, 
and let $M$ be any $2$-manifold.  We define a new $C_2$-space
$X \esum M$ as follows.  Let $M'$ denote $M$ with an open
disk removed.  Choose an open disk $D$
in $X$ that is disjoint from its conjugate $\sigma D$.  Let $X'$
denote $X$ with $D$ and $\sigma D$ removed.  Choose an isomorphism
$f\colon \bd M\ra \bd D$, and 
let
\[ X\esum M = \Bigl [X' \amalg (M'\times \{0\}) \amalg (M'\times \{1\}) \Bigr
]/\sim
\]
where the equivalence relation is $(m,0)\sim f(m)$ and $(m,1)\sim \sigma
f(m)$ for $m\in \bd M'$.  The $C_2$-structure on $X\esum M$ is the
evident one on $X'$ together with $\sigma(m,0)=(m,1)$ for $m\in M'$.  

The usual kinds of arguments show that when $X$ and $M$ are connected
this construction is
independent (up to isomorphim) on the choices of open disks and of the
map $f$.  See
Corollary~\ref{co:surg-inv}.

\begin{remark}
\label{re:esum}
Note the evident isomorphisms $X\esum (M\# N)\iso (X\esum M)\esum N$.
\end{remark}

The construction $X\esum N_1$ plays a special role for involutions on
non-orientable surfaces.  It can be regarded as adding ``dual
crosscaps'' to the $C_2$-manifold $X$, and we will usually denote
this as $X+[DCC]$.  Here ``DCC'' stands for
``Dual CrossCaps''.  We will write $X+2[DCC]$ for $X+[DCC]+[DCC]$, and so
forth.  This notation is not simpler than the $\esum$ notation, but it
works better with other surgery constructions that will be introduced
later.  

The following is a simple result using this notation:

\begin{prop}
\label{pr:DCC-esum}
Let $X$ be a connected $2$-manifold with $C_2$-action and let $M$ be any
connected $2$-manifold.  Then there is an equivariant isomorphism
\[ \bigl ( X+[DCC] \bigr )\esum M \iso X+(\beta(M)+1)[DCC].
\]
\end{prop}

\begin{proof}
We simply write
\begin{align*}
\bigl ( X+[DCC] \bigr) \esum M=(X\esum N_1)\esum M& \iso X\esum (N_1\#
M)\\
 &\iso X\esum
\bigl ( N_1^{\#(\beta(M)+1)} \bigr )\\
& \iso (((X \esum N_1)\esum N_1)\esum \cdots )\esum  N_1 \\
& \iso X + (\beta(M)+1)[DCC],
\end{align*}
where we have used Remark~\ref{re:esum} several times.
\end{proof}

\subsection{Constructions of free actions}
\label{se:cons-free}

Let $T_g$ be the genus $g$ torus.  If we assume this is embedded in
$\R^3$ in a standard way, with the ``center'' of the torus at the
origin, then the antipodal map $x\mapsto -x$ preserves the torus and
is an involution.  When $g$ is odd the origin is {\it inside\/} the
central hole of this embedded torus, and rotation by $180$ degrees
through an appropriate central axis gives another involution.  For $T_2$ and $T_3$ 
these are demonstrated in the pictures below:

\begin{picture}(300,180)(0,50)
\put(-10,120){\includegraphics[scale=0.35]{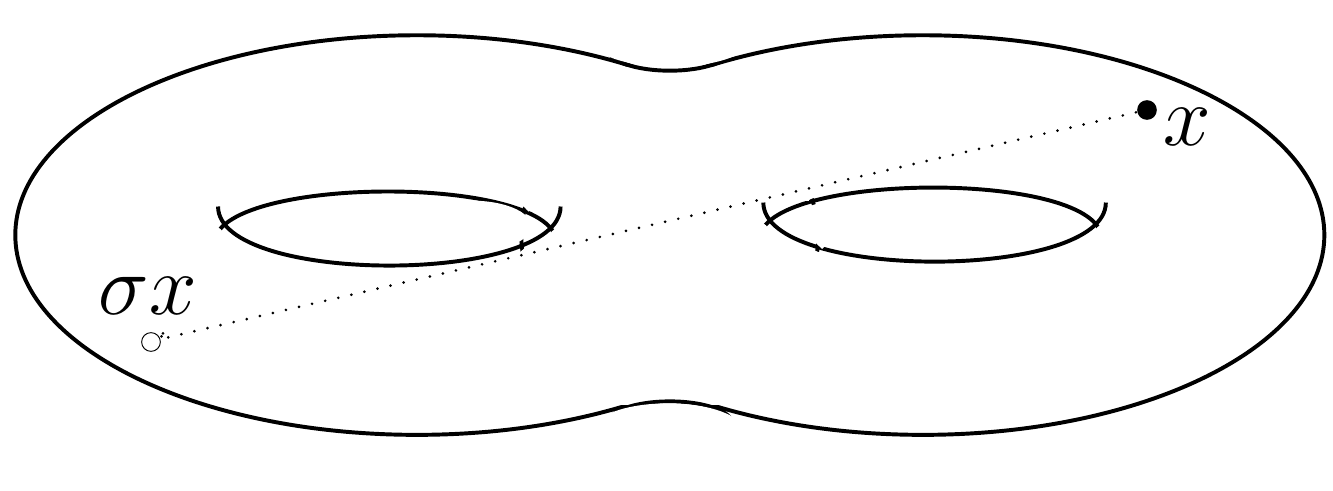}}
\put(100,90){\includegraphics[scale=0.45]{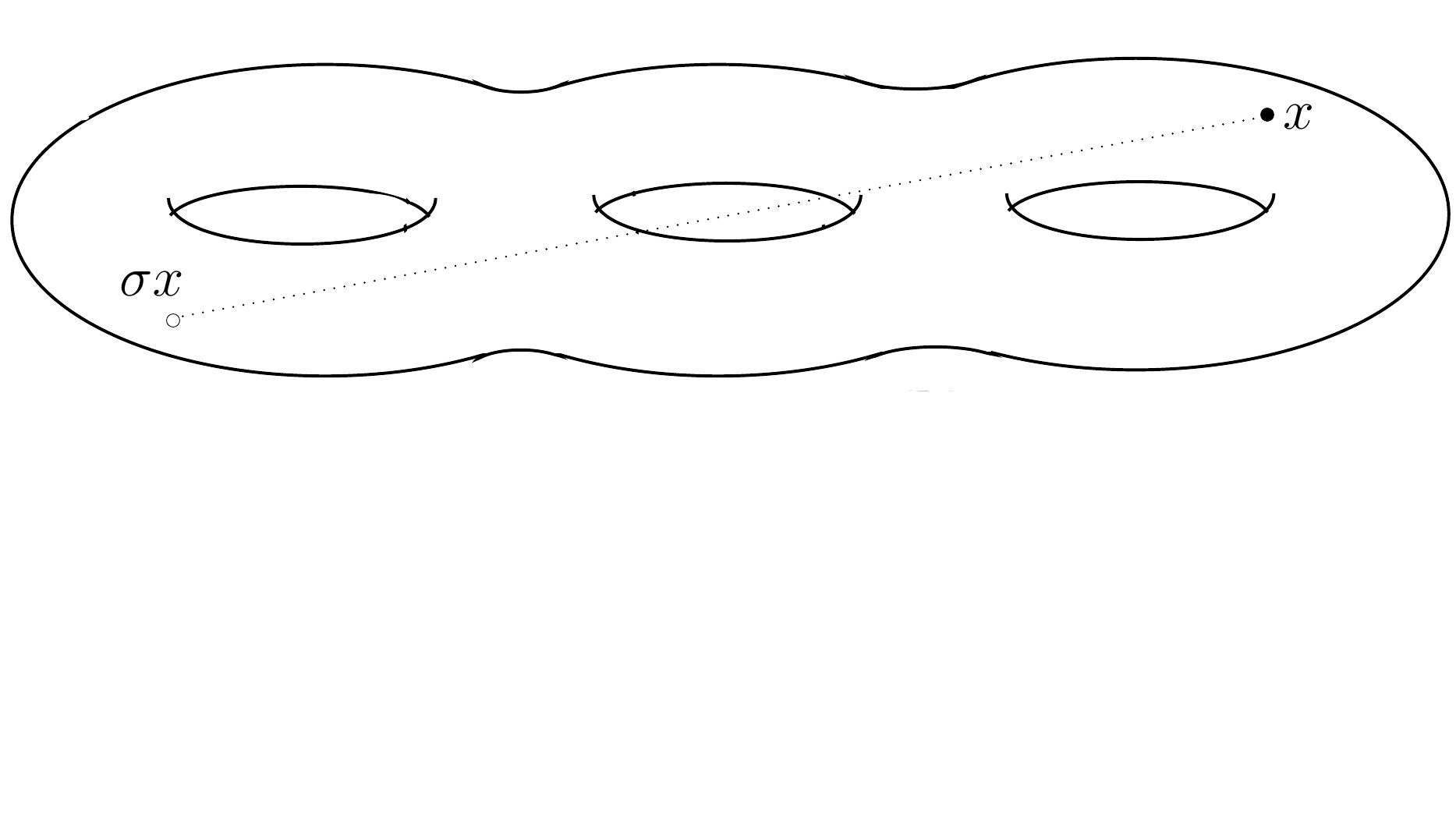}}
\put(100,0){\includegraphics[scale=0.45]{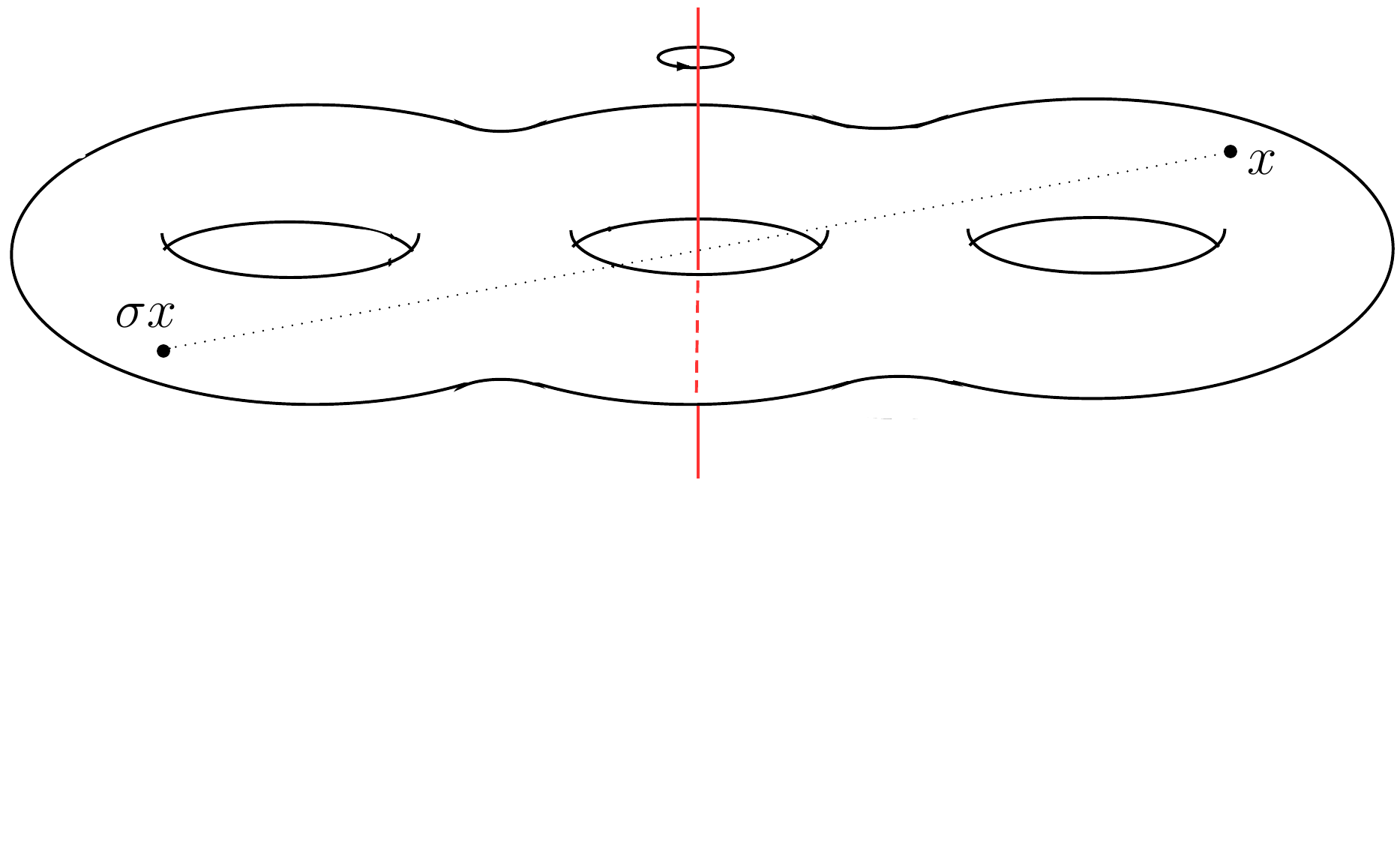}}
\end{picture}

\noindent
Note that in these pictures the black dots represent points on the top
side of the torus, whereas the open dots represent points on the
underside.  

We will write $T_g^{\anti}$ and $T_{g}^{\rot}$ to denote these two
surfaces with involutions.  
Note that 
\[T_{g}^{\rot}/C_2\iso T_{(1+g)/2} \quad\text{and}\quad
T_{g}^{\anti}/C_2\iso N_{g+1}
\]  
(recall that $g$ must be odd in the former case).  
The first is because
$T_{g}^{\rot}/C_2$ is orientable and its Euler characteristic
is half of $2-2g$.  Likewise,
the second is because
$T_{g}^{\anti}/C_2$ is non-orientable
with Euler characteristic also equal to $1-g$.

The following claims are easy to verify:

\begin{prop}
\label{pr:Tg-free}
For $g\geq 0$ there are equivariant isomorphisms 
\[ T_{2g}^{\anti}  \iso S^2_a \esum T_g, \qquad
T_{2g+1}^{\anti} \iso T_1^{\anti}\esum T_g, \qquad
T_{2g+1}^{\rot}  \iso T_1^{\rot}\esum T_g.
\]
\end{prop}

\begin{proof}
Left to the reader.
\end{proof}

Note that if $X$ is a $2$-manifold with free $C_2$-action and $M$ is
any $2$-manifold, then the action on $X\esum M$ is again free.
In particular, the constructions
$S^2_a+r[DCC]$, $T_1^{\anti}+(r-2)[DCC]$, and $T_1^{\rot}+(r-2)[DCC]$
give free actions on $N_r$. 
Here we are just adding crosscaps, as shown in the following pictures:

\begin{picture}(300,150)(0,-17)
\put(20,10){\includegraphics[scale=0.45]{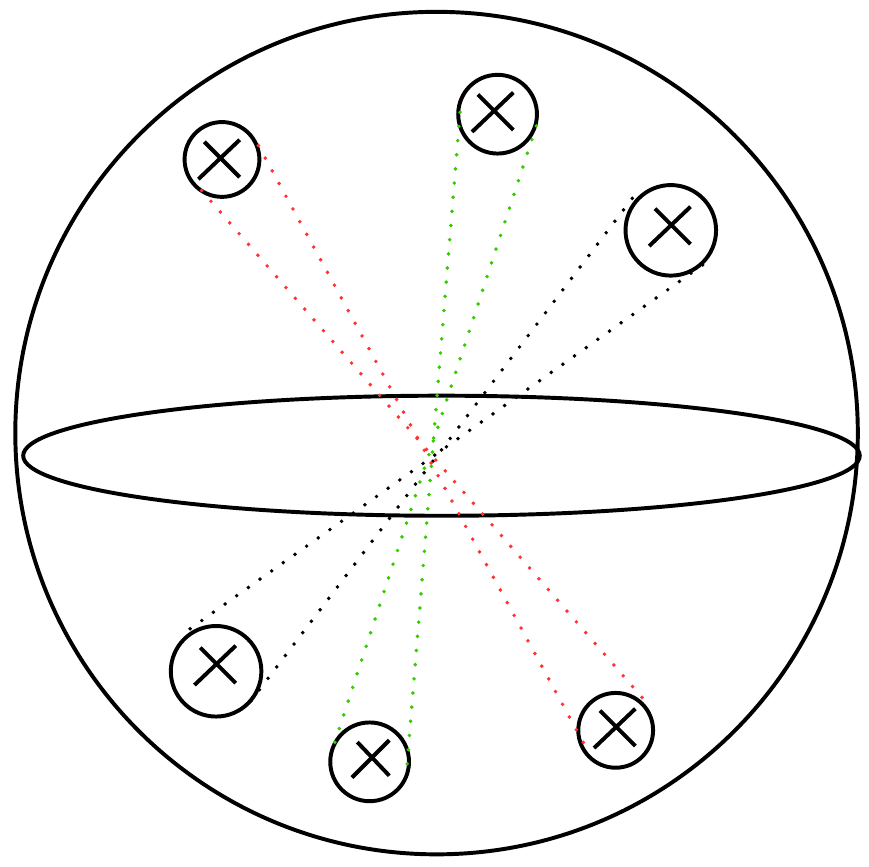}}
\put(50,-10){$S^2_a+3[DCC]$}
\put(170,-30){\includegraphics[scale=0.5]{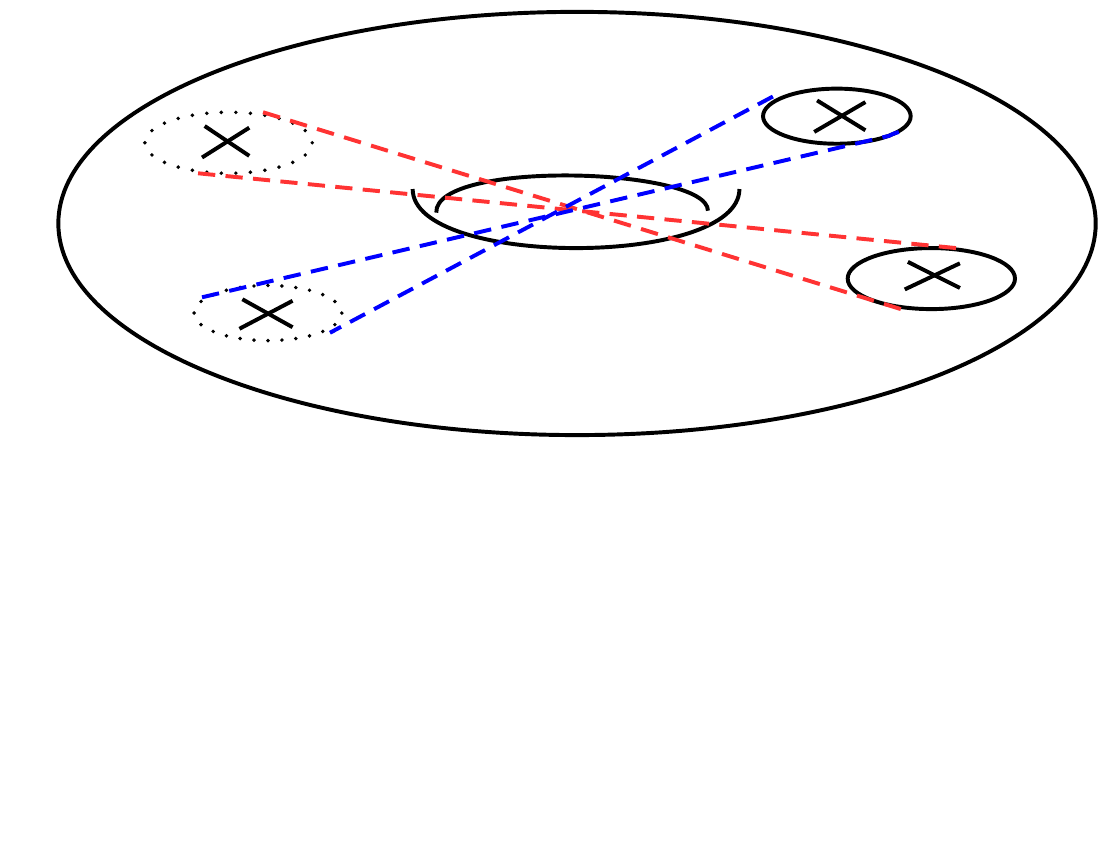}}
\put(220,0){$T_1^{\anti}+2[DCC]$}
\end{picture}

We will see later that $T_1^{\rot}+(r-2)[DCC]\iso
T_1^{\anti}+(r-2)[DCC]$ when $r>2$, but this requires a bit of work; see
Proposition~\ref{pr:rot=anti}.  The following are some easier
isomorphisms:

\begin{prop}
\label{pr:Tga-char}
For $g\geq 0$ and $s\geq 1$ there is a $C_2$-equivariant isomorphism 
\[ T_g^{\anti}+s[DCC]\iso
\begin{cases}
S^2_a+(g+s)[DCC] & \text{if $g$ is even,}\\
T_1^{\anti}+(g+s-1)[DCC] & \text{if $g$ is odd.}
\end{cases}
\]
Also, when $g$ is odd one has
\[ T_g^{\rot}+s[DCC]\iso T_1^{\rot}+(g+s-1)[DCC].
\]
\end{prop}

\begin{proof}
This is immediate from Proposition~\ref{pr:Tg-free}.  For example, 
when $g$ is even
\begin{align*}
 T_{g}^{\anti}+s[DCC]\iso (S^2_a\esum T_{\frac{g}{2}}) \esum N_s \iso S^2_a
\esum (T_{\frac{g}{2}}\# N_s) & \iso S^2_a\esum (N_{g+s})\\
&\iso S^2_a +
(g+s)[DCC].
\end{align*}
The other statements are proven similarly.  
\end{proof}


\section{Equivariant surgery and other generalities}

\label{se:general}
In this section we further develop the machinery for breaking down an
equivariant space into smaller pieces.  This mainly focuses on the
surgery-type constructions introduced  in Section~\ref{se:intro}.

\subsection{Connected components}

Let $X$ be a $2$-manifold with $C_2$-action.  Recall than an
\dfn{oval} is a component of $X^{C_2}$ that is homeomorphic to 
$S^1$.  We write $C$ for the number of ovals in $X$.

\begin{prop}
\label{pr:conn}
$X-X^{C_2}$ has either one or two path components.  Moreover, if
there are two components then every oval touches both of them.
\end{prop}

\begin{proof}
Paint $X-X^{C_2}$ different colors (Red, Blue, etc.), one for each
path component.  We will examine the colors attached to the ``sides''
of each oval in $X^{C_2}$ (where by side we mean the components of
$N-\cO$ where $\cO$ is the oval and $N$ is a tubular neighborhood).
Recall that an oval can have only one side; despite this, it will be
convenient to talk about the {\it two\/} sides of an oval even though
these two sides might be the same!  This is merely a linguistic
issue.  

Assume that $U$ and $V$ are ovals, each of which has a side that is
colored Red.  Pick $u\in U$ and $v\in V$.  Then there is a path
from $u$ to $v$ that stays entirely within the Red component
(except at its two endpoints).  The conjugate of this path starts on
the other side of $U$ and proceeds to the other side of $V$, never
crossing the fixed set, and so the other sides of $U$ and $V$ must
be colored the same.  So if two ovals have one side the same color,
the colors of the other sides must match as well.  

Continue to assume that $U$ is an oval with one side colored Red.
Let $\cC$ denote the color attached to the second side of $U$ (which
might or might not be Red).  Let $E$ be any oval that is different
from $U$, pick $u\in U$ and $e\in E$, and let $\alpha$ be a path from
$u$ to $e$ in $X$ that only crosses the fixed set finitely many
times (surely such a path exists).  Starting at $U$, whose sides are
colored Red and $\cC$, the next oval crossed by $\alpha$ will have a side
that matches one of these colors.  By the previous paragraph, the
two sides of this oval must be colored the same as $U$.  Now we apply
this argument to each successive oval crossed by $\alpha$, until we
reach $E$.  The conclusion is that the two sides of each oval are
colored the same as $U$.

The path components of $X-X^{C_2}$ only depend on the ovals, not the
isolated fixed  points in $X$.  We have just argued that if one oval
touches two (possibly equal) path components, then every oval touches
these same two components.  So $X-X^{C_2}$ has at most two components,
and if there are exactly two then every oval touches both of them.
\end{proof}

\begin{cor}
\label{co:conn}
If $X$ has a one-sided oval then $X-X^{C_2}$ is connected.
\end{cor}

\begin{proof}
Again color the components of $X-X^{C_2}$, with the component touching
our distinguished oval colored Red.  
By Proposition~\ref{pr:conn} {\it every\/} oval has both of its sides
colored Red, and so $X-X^{C_2}$ can have only one component.  
\end{proof}

When $X-X^{C_2}$ has two components we say that the involution
is \dfn{separating}.  Otherwise we call the involution
\dfn{non-separating}.  Let $\epsilon=\epsilon(X)$ denote this binary
invariant.  In a slightly different use of the term, we will also say that an oval $\cO$ is separating if
$X-\cO$ has two components.

\subsection{Doubled spaces and separating ovals}

Start with a non-equivariant surface $S$ and remove $C$
disjoint open disks to produce a space $X$.  Let $Y=[(X\times \{0\})
\amalg (X\times \{1\})]/\sim$ where the equivalence relation has
$(x,0)\sim (x,1)$ if $x\in \bd X$.  Give $Y$ the evident $C_2$-action
where $\sigma(a,0)=(a,1)$ for all $a\in X$.  We call $Y$ the
\mdfn{double of $S$ across $C$ boundary circles}, and we will denote
it $\Doub(S,C)$.  It is clear that $S$ is orientable if and only if
$\Doub(S,C)$ is orientable.  
Moreover, an Euler characteristic
argument shows if $S\iso N_s$ then $\Doub(S,C)\iso N_{2(s+C-1)}$ and
if $S\iso T_g$ then $\Doub(S,C)\iso T_{2g+C-1}$.  Uniting these cases,
we can write $\beta(\Doub(S,C))=2\beta(S)+2(C-1)$.

In the case of the doubling construction, the fixed set separates the
surface into two path components.  This turns out to be the {\it only\/}
case where this happens:

\begin{prop}
\label{pr:doub}
Let $X$ be a $2$-manifold with $C_2$-action.  If $X-X^{C_2}$ has two
path components then all ovals are two-sided and $X\iso \Doub(S,C)$
for some surface $S$ (where $C$  is the number of ovals in $X$).
\end{prop}

\begin{proof}
We know that all ovals are two-sided by Corollary~\ref{co:conn}.  Let
$P_1$ and $P_2$ denote the two path components of $X-X^{C_2}$, and let
$\bar{P}_1=P_1\cup X^{C_2}$ and $\bar{P}_2=P_2\cup X^{C_2}$.  The
involution then gives a homeomorphism $\sigma \colon \bar{P}_1\ra \bar{P}_2$.  

Certainly $\bar{P}_1$ can be obtained by
removing $C$ open disks from a $2$-manifold $S$.  
Construct a map $\Doub(S,C)\ra X$ by having
it be the inclusion on $\bar{P}_1\times \{0\}$ and $\sigma$ on
$\bar{P}_1\times \{1\}$.  This is an equivariant isomorphism.
\end{proof}

\begin{cor}
Suppose that $X$ has a separating oval.  Then $\beta(X)$ is even and
$X\iso \Doub(S,1)$ for
some $2$-manifold $S$ having $\beta(S)=\frac{\beta(X)}{2}$.  
\end{cor}

\begin{proof}
If $\cO$ is a separating oval then of course it must be two-sided, and
$X-X^{C_2}$ has two components.  
If $X$ has another oval then by Proposition~\ref{pr:conn} it touches
both components, which contradicts the statement that $X-\cO$ is
disconnected.  So in fact $X$ has only the single oval $\cO$.  
By Proposition~\ref{pr:doub} we know $X\iso \Doub(S,1)$ for some
$2$-manifold $S$.  But $\beta(\Doub(S,1))=2\beta(S)$, and so
$\beta(X)$ was even.  
\end{proof}

\begin{remark}
\label{re:doub}
Note that $\Doub(X,1)\iso S^{2,1}\esum X$.  
\end{remark}

\begin{cor}
\label{co:isfp=>ns}
Let $X$ be a $2$-manifold with $C_2$-action.  If $X$ has isolated
fixed points, then $X$ is non-separating.  In particular, if $\OO$ is
the union of the ovals in $X$ then any two isolated fixed points are
connected by a path in $X-\OO$.  
\end{cor}

\begin{proof}
By Proposition~\ref{pr:doub}, if $X$ is separating then $X\iso
\Doub(S,C)$ for some surface $S$; but this doubled space clearly has
no isolated fixed points.
\end{proof}

\subsection{Cylinders and caps}
\label{se:cylcap}
There are three types of equivariant circles: $S^{1,0}$, $S^{1,1}$,
and $S^1_a$.  If one of these circles lies inside an equivariant
$2$-manifold, there are multiple possibilities for what its
equivariant tubular neighborhood can look like.  The normal bundle to
the circle can be twisted or untwisted, and when untwisted there are
two possibilities depending on the $C_2$-representation type of the
normal direction.  The following pictures show all of the possible
normal bundles with nontrivial action.  Our convention is that the
fixed set is always shown in blue, and that identical symbols
represent conjugate points (so for examples, two points labelled with
a square in the same picture are conjugate).  On the  oriented spaces
we indicate whether the action is orientation-preserving or
reversing.  

\begin{picture}(300,120)
\put(0,90){\fbox{$S^{1,0}$-normal bundles}}
\put(0,0){\includegraphics[scale=0.5]{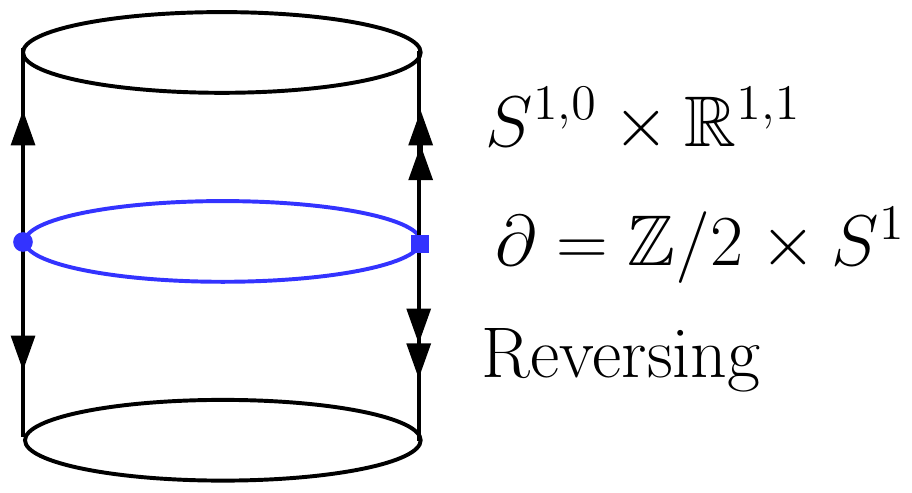}}
\put(150,0){\includegraphics[scale=0.5]{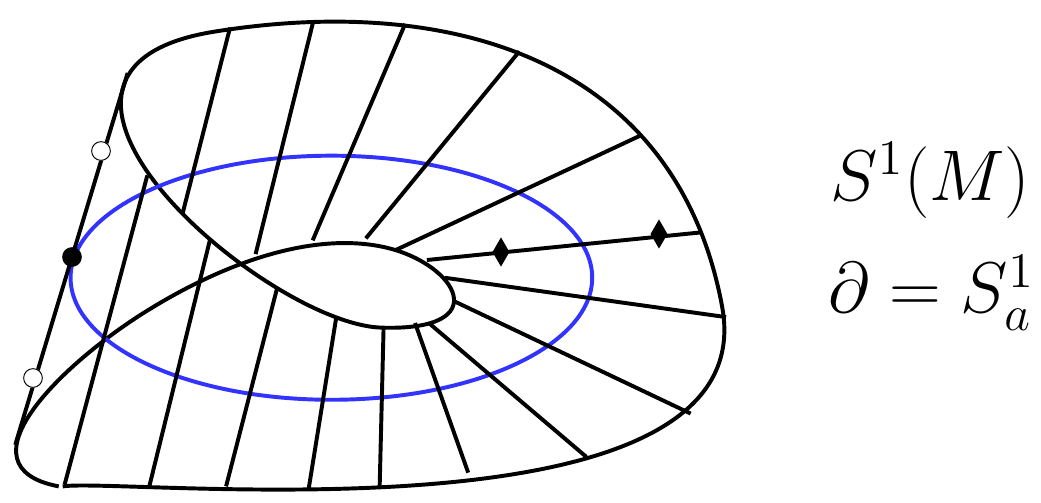}}
\end{picture}

\begin{picture}(300,120)
\put(0,90){\fbox{$S^{1}_a$-normal bundles}}
\put(0,0){\includegraphics[scale=0.5]{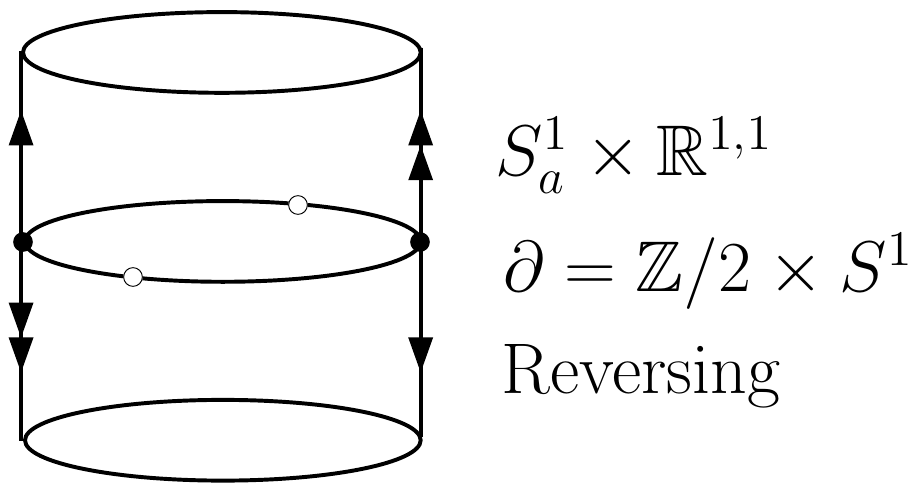}}
\put(180,0){\includegraphics[scale=0.5]{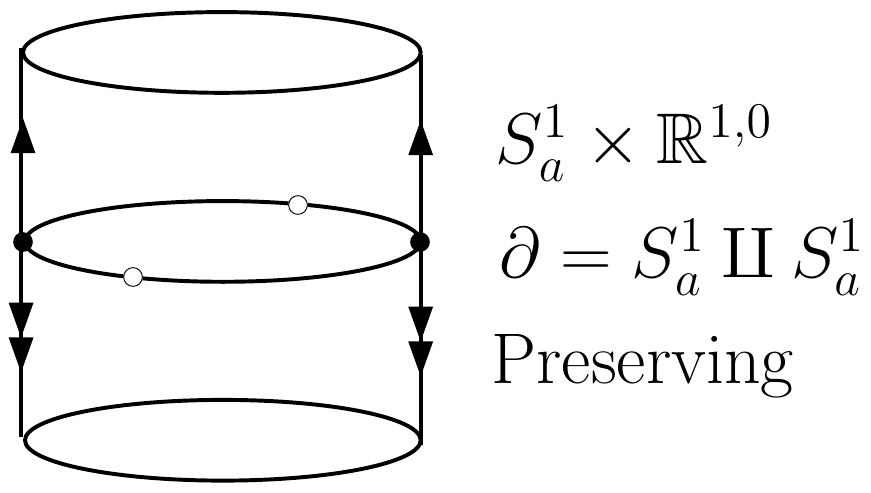}}
\end{picture}

\begin{picture}(300,170)(0,-10)
\put(0,130){\fbox{$S^{1,1}$-normal bundles}}
\put(0,0){\includegraphics[scale=0.5]{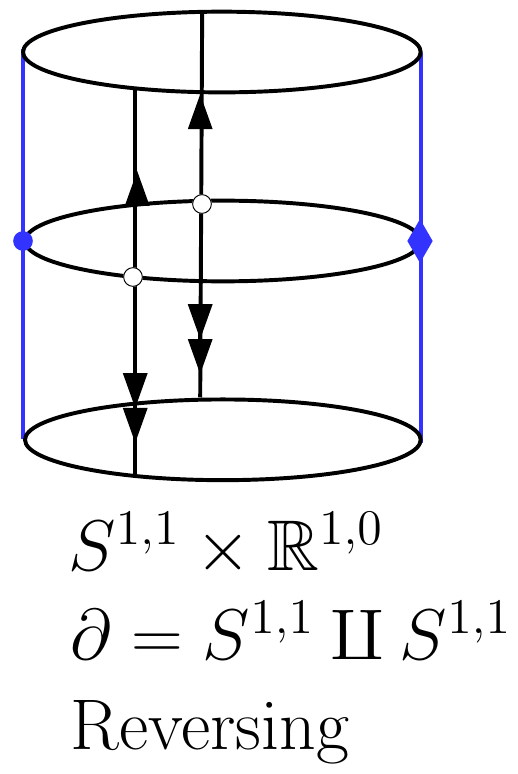}}
\put(110,0){\includegraphics[scale=0.5]{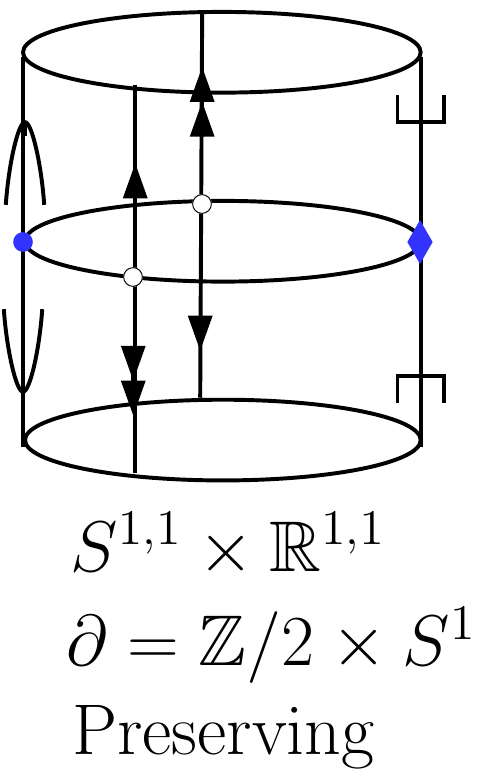}}
\put(220,0){\includegraphics[scale=0.5]{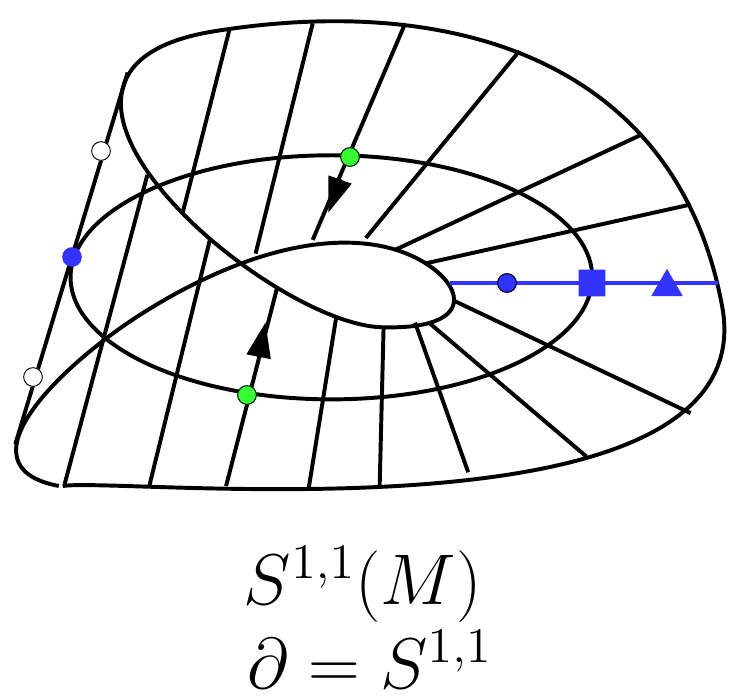}}
\end{picture}

For $S^{1,1}(M)$, note that there are two fixed points on the inner
circle (which is a copy of $S^{1,1}$); the fiber over one has the trivial
action, and the fiber over the other has nontrivial action.  
If $C$ is a circle type ($S^{1,1}$, $S^{1,0}$, or $S^1_a$) then we
will use the phrases ``$C$-tube'' and ``$C$-antitube'' to denote
copies of $C\times \R^{1,0}$ and $C\times \R^{1,1}$, respectively.

\begin{remark}
We did not include $S^{1,0}\times \R^{1,0}$ in our list because
it has trivial action, and so this space occurs inside a connected
$C_2$-space $X$  only if the action on $X$ is trivial.  But this
brings up an important point: there are three possible normal bundles
for $S^{1,0}$ and three for $S^{1,1}$, but in contrast for $S^1_a$ there are only
two.  There is no $C_2$-equivariant structure
on the  M\"obius bundle over $S^1_a$, by a routine argument.
\end{remark}

In addition to the ``cylinders'' (liberally interpreted)
that we just considered, it is also useful to think about equivariant
{\it caps\/}, i.e.  equivariant versions of $2$-disks.  There are
three of them where the action is nontrivial:

\begin{picture}(300,90)(0,-5)
\put(0,0){\includegraphics[scale=0.5]{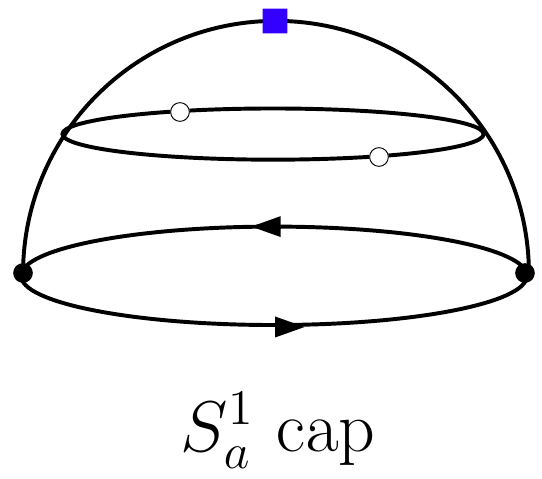}}
\put(120,0){\includegraphics[scale=0.5]{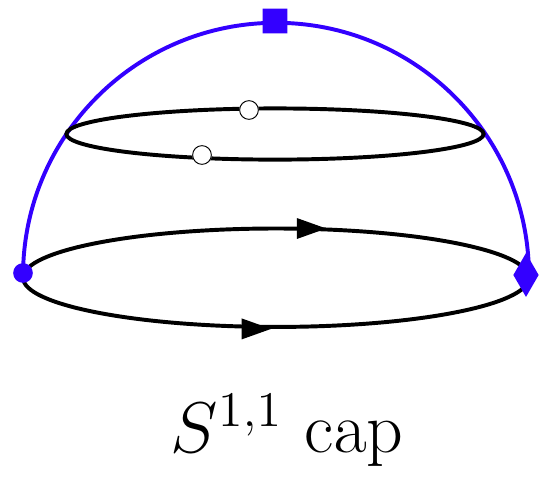}}
\put(240,0){\includegraphics[scale=0.5]{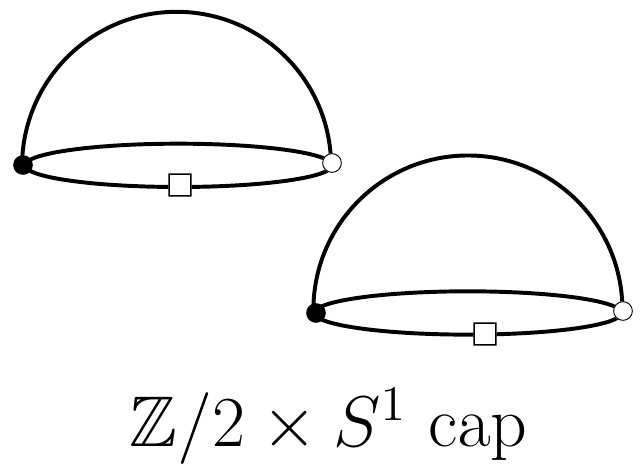}}
\end{picture}

\noindent
Note that from the point of view of surgery both $S^1(M)$ and
$S^{1,1}(M)$ also function like caps, in that they give us ways to ``cap
off'' an equivariant boundary circle.

\subsection{The generalized doubling construction }

Again start with a surface $S$ and remove $d$ disjoint open disks to
produce a space $X$.  Let $Y=(X\times \{0\})\amalg (X\times \{1\})$.
For each boundary component $\cO$ of $X$, attach a cylinder connecting
$\cO\times \{0\}$ to $\cO\times \{1\}$.  If we use $S^{1,0}$-antitubes
for all these cylinders, we get the doubling construction $\Doub(S,C)$
as previously discussed.  But we can also use $S^{1,1}$-antitubes and
$S^{1}_a$-antitubes, or any combination of these three types.  Let
$\Doub(S,a\colon S^{1,0}, b\colon S^{1,1}, c\colon S^{1}_a)$ denote such a construction,
where $a$, $b$, and $c$ count the number of each type of tube used
(so that $a+b+c=d$).  Note that $\beta(\Doub(S,a\colon S^{1,0},
b\colon S^{1,1},
c\colon S^{1}_a))=2\beta(S)+2(a+b+c-1)$.  

When $a=c=0$ and $b=1$ we will write $\Doub(X,S^{1,1})$, and similarly
for the other evident cases.  Generalizing Remark~\ref{re:doub}, note that
\[\Doub(X,S^{1,1})\iso S^{2,2}\esum X \quad\text{and}\quad 
\Doub(X,S^{1}_a)\iso S^2_a\esum X.
\]

\subsection{Surgeries}

Let $X$ be a $2$-manifold with $C_2$-action.  Given an equivariant
disk $D$ that is disjoint from its conjugate $\sigma D$, we can remove
the interiors of $D$ and $\sigma D$ from $X$ and then glue in an
equivariant cylinder whose boundary is $\Z/2\times S^1$: there are three
such cylinders, namely the $S^{1,0}$-, $S^1_a$-, and
$S^{1,1}$-antitubes.  We will write $X+(S^{?}_?-\text{antitube})$ to
denote the result of this process, with appropriate adornments on the
$S$. 

Conversely, if we find a copy of one of these antitubes inside $X$
then we can cut out the middle portion of the tube
and sew in two
$\Z/2\times S^1$-caps to replace it.  We will refer to any of these
procedures---whether sewing in an antitube or removing 
one---as $S^{1,0}$-, $S^1_a$-, or $S^{1,1}$-{\it surgery\/}, as
appropriate.

There is another type of surgery that will also be very useful to us.
If $x$ is an isolated fixed point in $X$, then locally around $x$ the
space looks like an $S^1_a$-cap.  We can remove the interior of this
cap from $X$ and then sew in a copy of $S^1(M)$: this will be
called \mdfn{$FM$-surgery} (since it replaces a Fixed point with a
M\"obius band).  The opposite process of removing a copy of $S^1(M)$
and replacing it with an $S^1_a$-cap will sometimes be called
\mdfn{$MF$-surgery}, though we will be a bit lax about $FM$ versus $MF$.

There are other kinds of surgeries one can perform, involving the
other cylinders and caps, but we will not make  use of these.  

\begin{remark}[$S^{1,1}$-surgery around fixed points]
\label{re:S11-fp}
Consider a $C_2$-action on a surface $X$, and let $a$ and $b$ be distinct
isolated fixed points.  By Corollary~\ref{co:nicepath}, there 
is a simple path
$\alpha$ from $a$ to $b$ in $X$ 
having the property that
$\alpha$ and $\sigma\alpha$ do not have any points in common except
the endpoints.  
In particular, $\alpha$  does not meet
$X^{C_2}$ except at the endpoints.  

Taken together, the pair $(\alpha,\sigma\alpha)$ gives an equivariant embedding
$S^{1,1}\ra X$.   Taking a tubular  neighborhood of the image gives an
$S^{1,1}$-antitube inside of $X$ having $a$ and $b$ as its fixed
points.  

If we remove $\alpha$ and $\sigma\alpha$ from $X$ there are two possibilities:
either this disconnects $X$ or it doesn't.  In the former case, $X$ is
isomorphic to the generalized doubling construction
$\Doub(S,1:S^{1,1})$ 
for an appropriate surface $S$.  
In the latter case, we can
put a $\Z/2\times S^1$-cap on the two open ends of the cut cylinder to
produce an equivariant space of smaller genus, where the number of
fixed points was reduced by two.

Note that different choices for the path $\alpha$ can lead to
different surgery scenarios; see Remark~\ref{re:strange} for an example.
\end{remark}

\subsection{Surgeries and isomorphisms}

In general, one has to be careful about ``cancelling'' surgeries
that appear 
inside isomorphism statements.  But here is one case where it works:

\begin{prop}
\label{pr:1,0-cancel}
Let $X$ and $Y$ be two $C_2$-spaces with only isolated fixed points.
Let $C\geq 0$ and assume $X+C[S^{1,0}-\text{antitube}]\iso
Y+C[S^{1,0}-\text{antitube}]$.  Then $X\iso Y$ (as $C_2$-spaces).  
\end{prop}

\begin{proof}
Let $\hat{X}=X+C[S^{1,0}-\text{antitube}]$ and 
$\hat{Y}=Y+C[S^{1,0}-\text{antitube}]$.  Choose an equivariant
isomorphism $\hat{X}\ra \hat{Y}$.  Because the homeomorphism preserves
the fixed sets, it must send the ovals in the antitubes of $\hat{X}$ 
to the ovals in the antitubes for $\hat{Y}$.  We can then choose a
collared neighborhood for each oval in $\hat{X}$ that maps to a
collared neighborhood of the image oval in $\hat{Y}$.  Let $\tilde{X}$
be obtained from $X$ by removing these neighborhoods and adding
$\Z/2\times S^1$-caps, and similarly for $\tilde{Y}$.  Then our
equivariant isomorphism $\hat{X}\ra \hat{Y}$ induces an equivariant
isomorphism $\tilde{X}\ra \tilde{Y}$.  But clearly $X\iso \tilde{X}$
and $Y\iso \tilde{Y}$.  
\end{proof}

Let $\cA$ denote a type of antitube ($S^{1,0}$, $S^{1,1}$, or
$S^1_a$).  If $X$ is an equivariant $2$-manifold containing an
$\cA$-antitube, write $X-[\cA-\antitube]$ for the space obtained by
removing the antitube and sewing in a $\Z/2\times S^1$-cap.  The
following is easy, but will be often used:

\begin{prop}
\label{pr:surg-subtract}
Let $X$ and $Y$ be equivariant $2$-manifolds, both containing an
$\cA$-antitube.  If
$X-[\cA-\antitube]\iso Y-[\cA-\antitube]$ then $X\iso Y$.  
\end{prop}

\begin{proof}
Let $\hat{X}$ denote $X-[\cA-\antitube]$, and similarly for
$\hat{Y}$.  Let $f\colon \hat{X}\ra \hat{Y}$ be an isomorphism.  The
space $\hat{X}$ has a pair of conjugate disks $D$ and $\sigma D$
corresponding to the cap that was sewn in, and these map to a pair of
conjugate disks $f(D)$ and $f(\sigma D)$ in $Y$.  The space $X$ is
obtained from $\hat{X}$ by doing surgery on these disks and sewing in
an $\cA$-antitube, so $f$ yields an isomorphism from $X$ to the space
obtained from $Y$ by doing the same surgery to $f(D)$ and $f(\sigma
D)$.  These latter disks are not necessarily the same as the caps we sewed in
when we made $\hat{Y}$, but by Corollary~\ref{co:surg-inv}
 it doesn't matter: any two
$\cA$-surgeries on $\hat{Y}$ yield isomorphic spaces, so we conclude
$X\iso Y$.  
\end{proof}


\section{Classifying Free actions}
\label{se:free}

In this section we classify all the free actions on $T_g$ and $N_r$.  
The techniques are classical and have been used by several authors;
see \cite{S} and \cite{A}, for example.

\medskip
Here is the  main classification result for free $C_2$-actions:

\begin{thm}[Classification of free actions]
\label{th:free-classify}
\mbox{}\par
\begin{enumerate}[(a)]
\item When $g$ is even 
there is a unique free $C_2$-structure on $T_g$; it is
represented by the antipodal action.
\item When $g$ is odd, there are two free $C_2$-structures on $T_g$: one
that is orientation-preserving, and one that is orientation-reversing.
The first is represented by a $180$-degree rotation about the central
hole, whereas the latter is represented by the antipodal action.
\item There are no free $C_2$-structures on $N_r$ when $r$ is odd.
\item There is exactly one free $C_2$-structure on $N_2$, represented
by $S^2_a+[DCC]$.  
\item For $s\geq 2$ there are exactly two free $C_2$ structures on
$N_{2s}$, represented by the two $C_2$-spaces
$S^2_a+s[DCC]$ and 
and $T_1^{\anti}+(2s-2)[DCC]$.  
\end{enumerate}
\end{thm}

\begin{remark}
Note that  
$T_1^{\rot}+s[DCC]$  is a free $C_2$-structure on $N_{2+2s}$, and so
the above theorem implies it is equivariantly isomorphic to one of 
$S^2_a+(1+s)[DCC]$ or
$T_1^{\anti}+s[DCC]$; but the theorem does not specify 
which one.  We will
see the answer in Proposition~\ref{pr:rot=anti} below.    
\end{remark}

It will take a while for us to prove Theorem~\ref{th:free-classify}.  
We will start with a very
general but coarse result in the next section, and then apply it to the case of
surfaces.  

\subsection{General results on classifying free actions}
Let us first recall the natural bijections
\[ (\text{principal $\Z/2$-bundles over $Y$}) \longleftrightarrow
[Y,B\Z/2] \longleftrightarrow H^1(Y;\Z/2).
\]
Given a principal $\Z/2$-bundle $P\ra Y$, the corresponding element of
$H^1(Y;\Z/2)$ is called its characteristic class and will be denoted
$\Lambda_P$.  
To describe it,
assume that $Y$ is path-connected and choose a basepoint $b$ in $Y$.
Define a map $\lambda_P\colon \pi_1(Y,b)\ra \Z/2$ by letting
$\lambda_P(\sigma)=0$ if the loop $\sigma$ lifts to a loop in $P$, and
$\lambda_P(\sigma)=1$ otherwise.  One can readily check that this is a
group map, and we therefore get the factorization
\[ \xymatrix{
\pi_1(Y,b) \ar[d]_{\lambda_P} \ar[r] & H_1(Y)\ar@{.>}[dl]^{\Lambda_P} \\
\Z/2
}
\]
since $H_1(Y)$ is the abelianization of $\pi_1(Y,b)$.  
So we have produced an element of $\Hom(H_1(Y),\Z/2)$, which is 
naturally isomorphic to $H^1(Y;\Z/2)$.  A little thought shows that
the construction of $\Lambda_P$ is independent of the choice of
basepoint $b$.  

Let $Y$ be a fixed path-connected space.  Let $\cS(Y)$ be the set of isomorphism
classes of free $C_2$-spaces $X$ that are path-connected and have
 the property that $X/C_2\iso Y$
(note that a choice of this isomorphism is not part of our data).  

\begin{prop}
\label{pr:cover-classify}
There is a bijection between $\cS(Y)$ and the set of nonzero orbits in
$H^1(Y;\Z/2)/\Aut(Y)$.  
\end{prop}

\begin{proof}
Let $X$ be a $C_2$-space representing an isomorphism class in
$\cS(Y)$.  Choose an isomorphism $Y\ra X/C_2$.  Pulling $X\ra X/C_2$ back along
this isomorphism gives a principal $\Z/2$-bundle over $Y$, which has a
characteristic class  $\Lambda_X\in H^1(Y;\Z/2)$.  Since $X$ is path-connected
the bundle is not trivial, and so $\Lambda_X$ is not
zero.   The class $\Lambda_X$ depends on the choice of isomorphism $Y\ra
X/C_2$, but another choice differs from this one by an element of
$\Aut(Y)$.  So we get a well-defined function $\cS(Y)\ra
H^1(Y;\Z/2)/\Aut(Y)$.

In the other direction, any element $u$ of $H^1(Y;\Z/2)$ is the
characteristic class of a principal $\Z/2$-bundle $E\ra Y$.  The space
$E$ with its inherent $\Z/2$-action then gives us an element of
$\cS(Y)$.  One readily checks that we have a bijection.
\end{proof}

\begin{remark}
Galois theory tells us that $2$-fold covers of $Y$ are classified by
index $2$ subgroups of $\pi_1(Y,*)$, and so one could conceivably
approach the above classification problem 
by using $\pi_1$ instead of $H^1$.  However,
there is a technical problem here because $\Aut(Y)$ does not act in a
natural way on $\pi_1(Y,*)$, due to the fact that automorphisms are
not required to fix the basepoint.  This problem is surmountable, but
it is easier to just use $H^1$ as we did above.  Note that giving a
nonzero element of $H^1$ is the same as giving a surjective map
$H_1(Y)\ra \Z/2$, which is the same as giving an index two subgroup of
$H_1(Y)$.  By the Hurewicz Theorem, the latter is equivalent to giving
an index two subgroup of $\pi_1(Y,*)$ (where $*$ is any chosen basepoint).  
\end{remark}

\begin{example}
\label{ex:T_2-antipodal}
Consider the genus two torus $T_2$ with its antipodal action.  The
quotient space is a torus with a crosscap, as demonstrated in the
following picture:

\begin{picture}(300,110)(0,-20)
\put(-30,0){\includegraphics[scale=0.5]{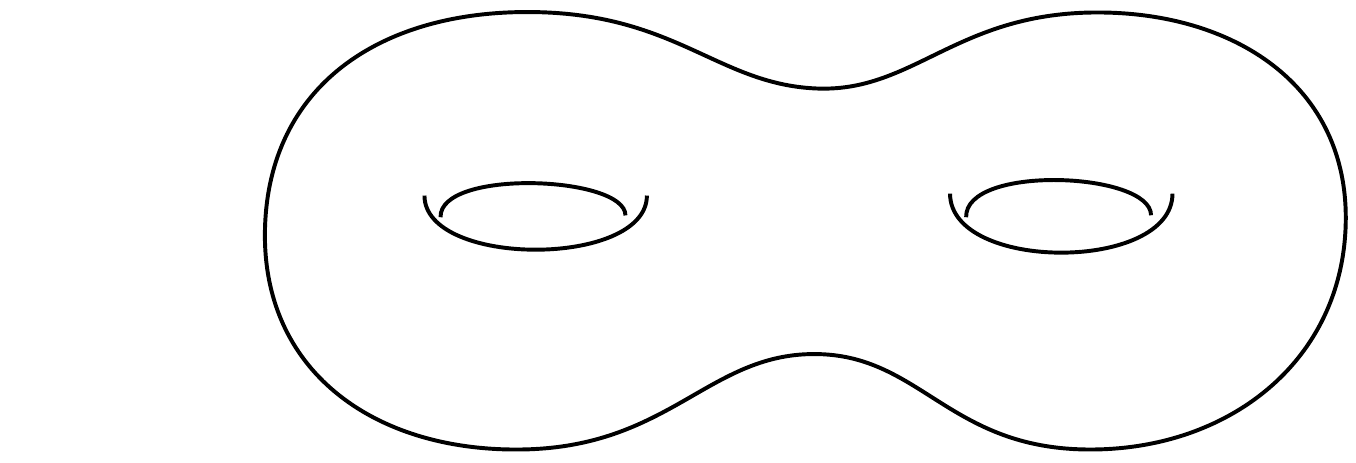}}
\put(190,30){$\lra$}
\put(190,0){\includegraphics[scale=0.55]{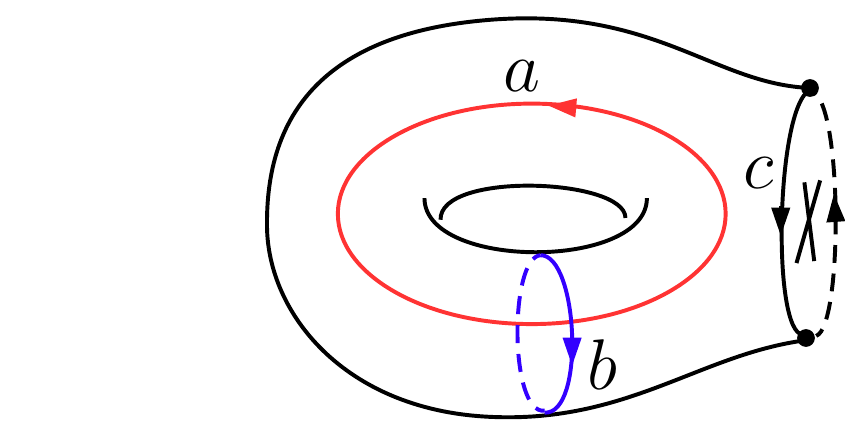}}
\end{picture}

So $T_2/C_2\iso N_3$, and $H_1(N_3;\Z_2)$ is generated by the elements
$a$, $b$, and $c$ from the picture.  Both $a$ and $b$ lift to loops
under the projection $T_2\ra T_2/C_2$, whereas $c$ does not.  
So under the bijection of Proposition~\ref{pr:cover-classify} the
$C_2$-space $T_2$ corresponds to the map $H_1(N_3;\Z_2)\ra \Z/2$
sending $a\mapsto 0$, $b\mapsto 0$, $c\mapsto 1$.  
\end{example}

Let us now return to develop a bit more of the general theory.
The action of $\Aut(Y)$ on $H^1(Y;\Z/2)$ is a group homomorphism 
\begin{myequation}
\label{eq:H^1-rep}
 \Aut(Y) \ra \Aut(H^1(Y;\Z/2)).
\end{myequation}
Our next goal will be to understand the image of this map when $Y$ is a closed
$2$-manifold, as this will allow us to compute the orbits.
To this end, let $\cI(Y)\subseteq \Aut(Y)$ be the
(normal) subgroup of automorphisms that are isotopic to the identity.  The
\dfn{full mapping class group} of $Y$ is $\cM(Y)=\Aut(Y)/\cI(Y)$.  
Note that the action of $\Aut(Y)$ on $H^1(Y)$ factors through an
action of $\cM(Y)$, since $\cI(Y)$ acts trivially on $H^1(Y)$.  

When $Y$ is a closed $2$-manifold, 
the map in (\ref{eq:H^1-rep}) turns out to be surjective only when
$Y=S^2$.   
Indeed, the cup product gives a nondegenerate form on $H^1(Y;\Z/2)$,
and the action of $\Aut(Y)$ must preserve this form.  In the case
where $Y$ is orientable of genus $g$, this form is symplectic and the map
(\ref{eq:H^1-rep}) therefore factors through the symplectic group
$\Sp(2g,\Z/2)$.  In the case when $Y\iso N_r$ the form is
orthogonal,  and the map (\ref{eq:H^1-rep}) therefore
factors through an orthogonal group $O(r,\Z/2)$.  

We are thereby led to consider the two maps
\begin{myequation}
\label{eq:T_g}
 \cM(T_g) \lra  \Iso(H^1(T_g;\Z/2),\langle \blank,\blank\rangle)\iso \Sp(2g,\Z/2)
\end{myequation}
and
\begin{myequation}
\label{eq:N_s}
 \cM(N_r) \lra \Iso(H^1(N_r;\Z/2),\langle \blank,\blank\rangle)\iso
O(r,\Z/2).
\end{myequation}
Here $O(r,\Z/2)$ is the orthogonal group of $r\times r$ matrices $A$ with
entries in $\Z/2$ satisfying $AA^T=I_s$, 
and the isomorphism with our isometry group
depends on a choice of
orthonormal basis for $H^1(N_r;\Z/2)$.  
Likewise, the isomorphism with the sympectic group depends on a choice of
symplectic basis for $H^1(T_g;\Z/2)$,

The following result must be classical in the theory of mapping class
groups.  We include the proof for lack of a reference.

\begin{thm}
\label{th:mcg-H^1}
The homomorphisms of (\ref{eq:T_g}) and (\ref{eq:N_s}) are both
surjective.  
\end{thm}

\begin{proof}
Let $\cM^+(T_g)\subseteq \cM(T_g)$
denote the subgroup consisting of orientation-preserving automorphisms.
Note that there is a short exact sequence of groups
\[ 1\ra \cM^+(T_g) \inc \cM(T_g) \ra \Z/2 \ra 1 \]
where the surjective map records the determinant of the induced map on
$H^1(T_g;\Z)$.  We have a commutative square
\[ \xymatrix{
\cM^+(T_g) \ar[r]\ar[d] &
\Sp(2g,\Z) \ar[d]\\
\cM(T_g) \ar[r] & \Sp(2g,\Z/2).
}
\]

The right vertical map is surjective; this is fairly easy to prove by hand,
but it also follows from \cite[Theorem 8.5]{T}.
The top horizontal
map is also known to be surjective; see 
\cite[Theorem 6.4]{FM}.  So the lower horizontal
map is surjective as well.

The result for $N_r$ is the subject of the paper \cite{MP}, but
that paper was never published and the online version has some
cosmetic blemishes.  
So we include a sketch of the proof here.  Note, however, that the
argument is entirely taken from \cite{MP}.   

Model $N_r$ as a sphere with $r$ crosscaps, as in the following
picture (where the boundary of the given disk is identified to a
point):

\begin{picture}(300,100)
\put(40,10){\includegraphics[scale=0.6]{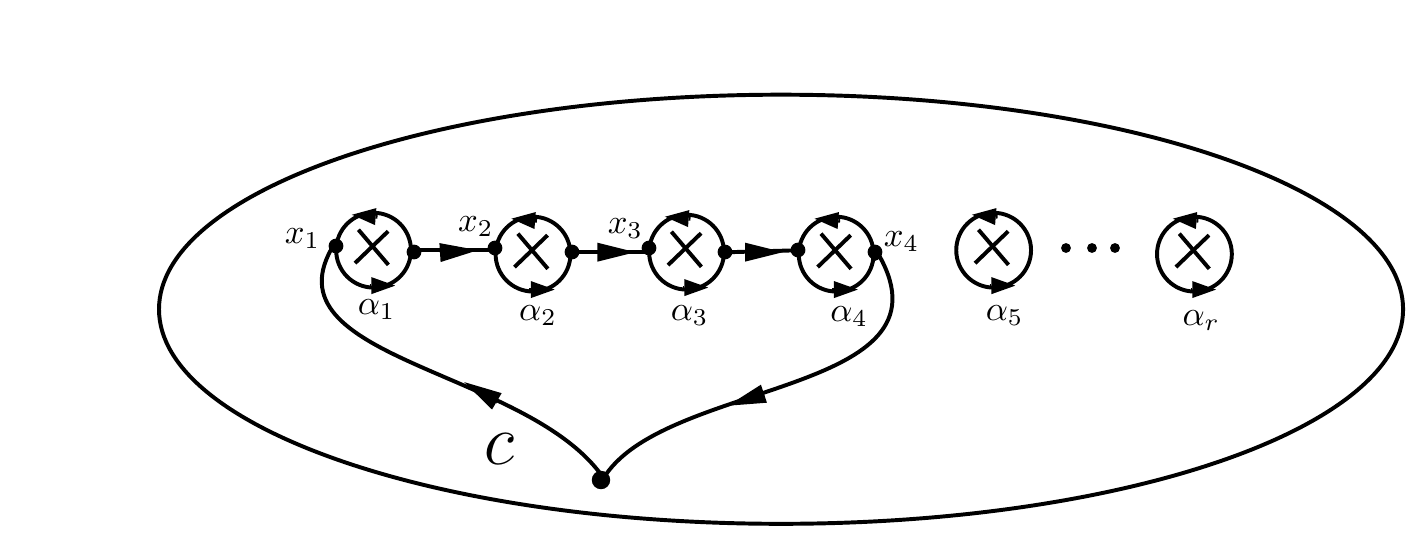}}
\end{picture}

\noindent
Note that $\alpha_1,\ldots,\alpha_r$ is an orthonormal basis for
$H_1(N_r;\Z/2)$.  When $r\geq 4$ let
$c$ be the indicated path, which sequentially reaches points
$x_1$, $x_2$, $x_3$, and $x_4$, hopping over the crosscaps (for
example, the path $\alpha_1$ is not a subpath of $c$).  Then
$c=\alpha_1+\alpha_2+\alpha_3+\alpha_4$ in $H_1(N_r;\Z/2)$.

The path $c$ is two-sided (as it crosses an even number of crosscaps),
so we can consider the Dehn twist $\tau_c$ associated to $c$.  On homology it
induces the map 
$x\mapsto x+\langle x,c\rangle c$.  So $(\tau_c)_*(\alpha_i)=\alpha_i$ for
$i>4$, whereas $(\tau_c)_*(\alpha_i)=\alpha_i+c$ for $i\leq 4$.  

It is clear that there are elements of $\cM(N_r)$ that transpose any
two crosscaps and leave the others fixed (if it is not clear, use the
Dehn twist about an analog of $c$ that hops across exactly two
crosscaps).  So all the permutation matrices are in the image of
$\cM(N_r)\ra O(r,\Z/2)$.  By Corollary~\ref{co:O-generate} below (this forward
reference is awkward but convenient), the group
$O(r,\Z/2)$ is generated by these permutation matrices together with
$(\tau_c)_*$ (when $r\geq 4$).  So $\cM(N_r)\ra O(r,\Z/2)$ is surjective.
\end{proof}

\begin{cor}
\label{co:classify}
The set $\cS(T_g)$ is in bijective correspondence with the set of
nonzero orbits in $(\Z/2)^{2g}/\Sp(2g,\Z/2)$.  Likewise, the set $\cS(N_r)$
is in bijective correspondence with the nonzero orbits in
$(\Z/2)^r/O(r,\Z/2)$.  
\end{cor}

\begin{proof}
This is immediate from 
Proposition~\ref{pr:cover-classify} and Theorem~\ref{th:mcg-H^1}. 
\end{proof}

\subsection{Algebraic calculations}

Consider the vector space $\F_2^n$ with the dot product.  
In this section we will write $O(n)$ for $O(n,\F_2)$.  
For any
vector $v=[v_1,\ldots,v_n]\in \F_2^n$, define the \dfn{content} of 
$v$ to be $c(v)=\sum_i v_i \in \F_2$.  Note that
\[ c(v)=\sum_i v_i^2 = v\cdot v
\]
and so the action of $O(n)$ preserves the content.  

The vector $\Omega=[1,1,\ldots,1]$ is the unique vector in $\F_2^n$ having
the property that $\Omega\cdot x=x\cdot x$ for all vectors $x$.
 As such, $\Omega$ must be
preserved by $O(n)$.

\begin{lemma}
\label{le:one}
When $n\geq 3$ there are exactly four orbits of $O(n)$ on $\F_2^n$,
represented by the elements
\[ 0=[0,0,\ldots,0], \quad a_1=[1,0,\ldots, 0], \quad a_2=[1,1,0,\ldots,0],\quad
\Omega=[1,1,\ldots,1].
\]
When $n=2$ there are exactly three orbits, represented by $0$, $a_1$,
and $\Omega$.  If $v$ is a vector whose coordinates have at least one $0$
and at least one $1$, then the orbit of $v$ is determined by the
parity of the number of $1$s.  
\end{lemma}

\begin{proof}
Write $b_1,\ldots,b_n$ for the standard basis of $\F_2^n$.
Transposing any two basis elements is an isometry, so over $\F_2$ the 
symmetric matrices are all orthogonal.  In particular, the
vectors $b_1,\ldots,b_n$ are all in the same orbit of $O(n)$;
the vectors $b_i+b_j$ ($i\neq j$) are all in the same orbit, the
vectors $b_i+b_j+b_k$ ($i$, $j$, $k$ all distinct) are all in the same
orbit, etc.  This proves that there are at most $n+1$ orbits,
represented by the vectors
\[ 0, \quad b_1, \quad b_1+b_2, \quad b_1+b_2+b_3,\quad \ldots,\quad b_1+b_2+\cdots+b_n.
\]
We also know that $\{0\}$ and $\{b_1+\cdots+b_n\}$ are singleton
orbits.
The $n=2$ case is now complete.  For the $n=3$ case we merely observe that
the content function shows that $b_1$ and
$b_1+b_2$ are in distinct orbits, so this case is also done.  

Now suppose $n\geq  4$.  Then the linear map 
\[ \begin{cases}
b_1 \mapsto b_2+b_3+b_4, \qquad 
b_2 \mapsto b_1+b_3+b_4, \\
b_3\mapsto b_1+b_2+b_4,\qquad
b_4\mapsto b_1+b_2+b_3, \\
b_i\mapsto b_i \  (i\geq 5)
\end{cases}
\]
is readily checked to be an isometry.  This shows that $b_2+b_3+b_4$
and $b_1$ are in the same orbit, and so $b_1+b_2+b_3$ and $b_1$ are
also in the same orbit; this completes the $n=4$ case.
Moreover, when $n>4$ we have that $b_1+b_2+b_3+b_5$ is in the same
orbit as $b_1+b_5$, $b_1+b_2+b_3+b_5+b_6$ is in the same orbit as
$b_1+b_5+b_6$, and so forth.  This completes the proof.
\end{proof}

For the following corollary, if $A$ is a $k\times k$ matrix and $B$ is an
$l\times l$ matrix write $A\oplus B$ for the $(k+l)\times (k+l)$
block diagonal matrix
$\begin{bsmallmatrix} A & 0
\\ 0 & B\end{bsmallmatrix}$.

\begin{cor}
\label{co:O-generate}
Let $n\geq 1$.  Then $O(n)$ is 
generated by the permutation
matrices together with (in the case $n\geq 4$) the single matrix
$A\oplus I_{n-4}$ where $A=\begin{bsmallmatrix} 
0 & 1 & 1 & 1 \\
1 & 0 & 1 & 1 \\
1 & 1 & 0 & 1 \\
1 & 1 & 1 & 0 \end{bsmallmatrix}$.
\end{cor}

\begin{proof}
This is the main content of \cite[Theorem 1.1]{MP}.  
Let
$e_1,\ldots,e_n$ be an orthonormal basis for $\F_2^n$.  
Note that if $M\in O(n)$ then 
$c(Me_i)$ must be odd for each $i$, since $\langle Me_i,Me_i\rangle =
\langle e_i,e_i\rangle =1$.  

Let $H\subseteq
O(n)$ be the subgroup generated by the permutation matrices and
(in the case $n\geq 4$) the matrix $A\oplus I_{n-4}$.  
The fact that $H=O(n)$ is trivial when $n=1$.

We proceed by induction, so assume $n\geq 2$.  
If $O(n)\neq H$, we can choose an $M\in O(n)-H$ such that
$k=\#\{i\,|\, Me_i=e_i\}$ is as large as possible.  If $k=r$ then
$M=\Id$, which contradicts $M\notin H$; so $k<r$.  By composing $M$ with
a permutation matrix we can assume $Me_i=e_i$ for $1\leq i\leq k$,
and therefore $M=I_{k}\oplus M'$ where $M'\in O(n-k)$.  
If $k>0$ then $M'$ belongs to $H$ by induction, and so $M\in H$.  
So we must have $k=0$.  

The proof of Lemma~\ref{le:one} actually shows
that $H$ acts transitively on the vectors $v$ in $\F_2^n-\{\Omega\}$ that
have $c(v)=1$.  Consider $v=Me_1$, and note that $c(v)=1$.  We cannot
have $v=\Omega$, as $M\Omega=\Omega$ and $\Omega\neq e_1$ (since
$r>1$).  
So there is a matrix $B\in H$ such that $Bv=e_1$.  The composite $BM$
therefore fixes $e_1$, but is not in $H$.  This contradicts the fact
that $M$ was chosen to make $k$ maximum.       
\end{proof}

The group $\Sp(2g,\Z/2)$ acts on
$(\Z/2)^{2g}$ via left multiplication. 
The following result describes the orbits:

\begin{lemma}
\label{le:two}
There are exactly two orbits of $\Sp(2g,\Z/2)$ acting on $(\Z/2)^{2g}$: 
one is the singleton orbit consisting of the zero vector,
and the other is the set of all nonzero vectors.
\end{lemma}

\begin{proof}
We first consider the case $g=1$.  The matrices
\[ 
\begin{bmatrix} 
0 & 1 \\
1 &  1 
\end{bmatrix}\quad\text{and}\quad
\begin{bmatrix} 
1 &  0 \\
1 &  1 
\end{bmatrix}
\]
are both symplectic and carry $[1,0]$ to $[0,1]$ and $[1,1]$,
respectively.  So all three of these elements are in the same orbit,
and this case is done.

For the general case, let $e_1,\ldots,e_{2g}$ be a symplectic basis
for $(\Z/2)^{2g}$, where our convention is that
\[ \langle e_{2i-1},e_{2i}\rangle =1 = \langle e_{2i},e_{2i-1}\rangle
\] 
for all $i\in \{1,\ldots,g\}$, and all other pairings between basis
elements are zero.  We will refer to each pair $\{e_{2i-1},e_{2i}\}$
as a ``symplectic block''.  

We can denote elements of $(\Z/2)^{2g}$ as
$v=[B_1,\ldots,B_g]$ where each $B_i\in (\Z/2)^2$, via  the convention 
\[ v=(B_1)_1e_1 + (B_1)_2 e_2 + (B_2)_1 e_3 + (B_2)_2 e_4 +\cdots
\]
In other words, $B_i$ contains the coordinates of $v$ with respect to
the $i$th symplectic block.  

Write $\Sp(n)$ for $\Sp(n,\Z/2)$.  Inclusion of block diagonal
matrices gives a
group homomorphism $\Sp(2)\times \cdots \times \Sp(2)\ra
\Sp(2g)$,  and the
$g=1$ case now shows that there are at most $2^g$ orbits on
$(\Z/2)^{2g}$, represented by vectors $[B_1,\ldots,B_g]$ with each
$B_i\in \{[0,0],[1,0]\}$.  Additionally, permutations of the
symplectic blocks are all elements of $\Sp(2g)$ and so we can do a bit
better: there are at most $g+1$ orbits, represented by the vectors
\begin{myequation}
\label{eq:T-list}
 [O,O,\ldots,O], \ [T,O,\ldots,O], \ [T,T,O,\ldots,O], \ldots, \
[T,T,\ldots,T]
\end{myequation}
where $T=[1,0]$ and $O=[0,0]$.  

Next we observe that when $g=2$ the following matrix is symplectic:
\[ A=\begin{bmatrix}
1 & 1 & 1 & 1 \\
0 & 0 & 0 & 1 \\
1 & 0 & 0 & 1 \\
0 & 1 & 0 & 1
\end{bmatrix}.
\]
For $g\geq 2$, the direct sum $A\oplus \id_{(\Z/2)^{2g-4}}$ takes
$[T,O,\ldots,O]$ 
to $[T,T,O,\ldots,O]$.  The direct sum $\id_{(\Z/2)^2}\oplus A\oplus
\id_{(\Z/2)^{2g-6}}$ takes $[T,T,O,\ldots,O]$ to $[T,T,T,O,\ldots,O]$, and
continuing in this way we see that all of the nonzero vectors in
(\ref{eq:T-list}) are in the same orbit.
\end{proof}

\subsection{Two-fold coverings of surfaces}
Putting Corollary~\ref{co:classify} together with Lemma~\ref{le:one}
and Lemma~\ref{le:two}, we immediately see that  up to isomorphism
there is only one free $C_2$-action on a surface whose quotient is
$T_g$, and for $r\geq 3$ there are three free $C_2$-actions on surfaces
whose quotient is $N_r$.  It remains to explicitly identify these.  To
this end, we start by considering some examples.

\begin{example}
\label{ex:T_2-anti-cc}
First, consider the antipodal action on $T_2$.  As we saw
in Example~\ref{ex:T_2-antipodal}, the quotient is isomorphic to a torus with a
crosscap (a copy of $N_3$), having basis $\{a,b,c\}$ for
$H_1(N_3;\Z/2)$.  Here is the picture of $T_2/C_2$ again:

\begin{picture}(300,90)
\put(80,0){\includegraphics[scale=0.6]{T2quotNEW.pdf}}
\end{picture}

\noindent
We also saw that the $C_2$-space $T_2$ is classified
by the linear functional $\lambda$ defined by $a,b\mapsto 0$, $c\mapsto 1$.  

So far this is all fine, but the basis $\{a,b,c\}$ is not an
orthonormal basis for $H_1(N_3;\Z/2)$.  Indeed, one has
\[ 0=a\cdot a=b\cdot b=b\cdot c=a\cdot c, \quad
1=a\cdot b=c\cdot c.
\]
In order to tie in with our classification results we need to use a
different basis.  A moment's thought verifies that $\{a+c,b+c,a+b+c\}$ is
an orthonormal basis.  Our linear function $\lambda$ sends all of these
elements to $1$; so in the notation of Lemma~\ref{le:one} our
characteristic class for the antipodal action on $T_2$ is the orbit of
$[1,1,1]$.  
\end{example}

\begin{example}
It is worth looking at $T_3$ with its antipodal action as well.  Here
is the quotient space $T_3/C_2$:

\begin{picture}(300,100)(0,0)
\put(80,-40){\includegraphics[scale=0.5]{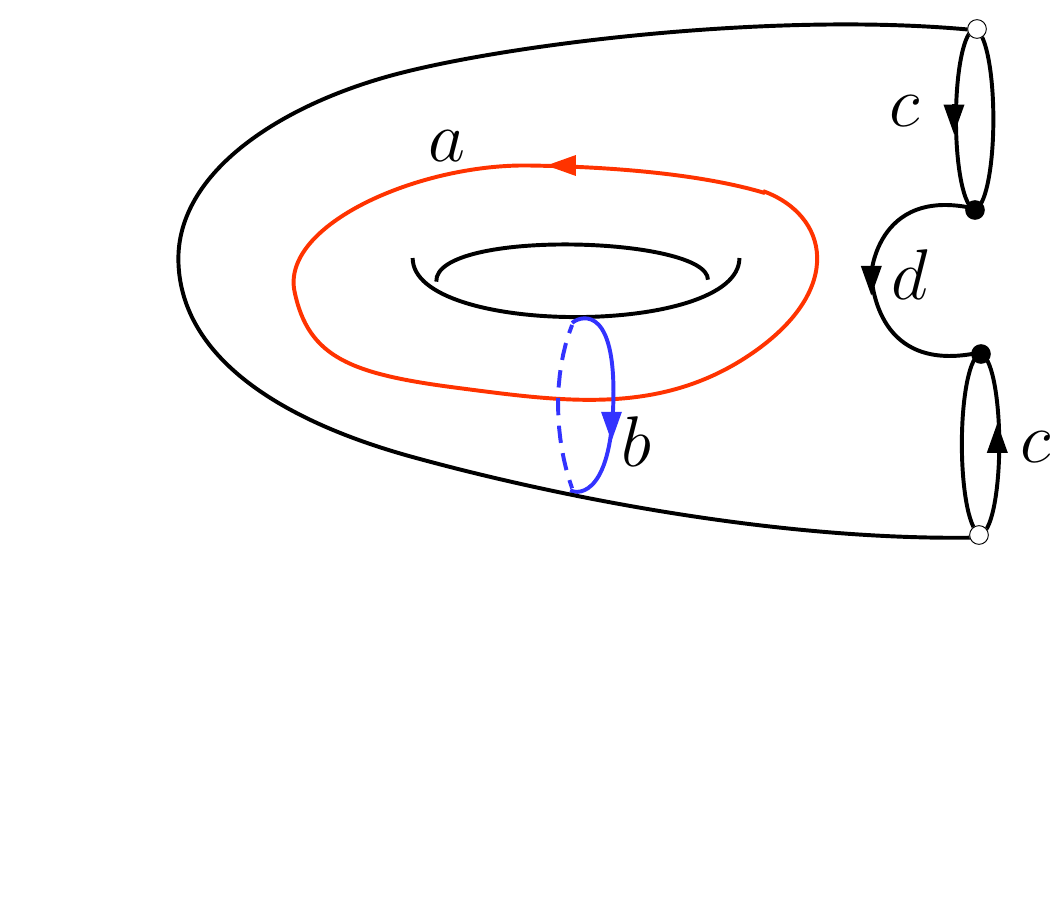}}
\end{picture}

\noindent
Note that the loop $c$ goes all the way around the drawn circle, not
half-way around; the open dot and solid dot are not identified.
So $a$, $b$, and $c$ lift to loops in $T_3$, whereas $d$ does not.  

The intersection products are given by
\[ 0=a\cdot a=b\cdot b=c\cdot c=a\cdot c=a\cdot d=b\cdot d, \quad
1=a\cdot b=d\cdot d=c\cdot d.
\]
One readily checks that $a+d$, $b+d$, $a+b+d$, $c+d$ is an orthonormal
basis.  The values of our characteristic class are therefore
$[1,1,1,1]$.  
\end{example}

The last two examples are representative of all cases.  We leave it to
the reader to check that for $T_g$ with its antipodal action the
characteristic class is always $[1,1,\ldots,1]$.

\begin{example}
\label{ex:T_0,2s}
Consider the space $X=S^2_a+s[DCC]$ and its quotient space $Q=X/C_2$
shown below:

\begin{picture}(300,110)(0,-10)
\put(80,-10){\includegraphics[scale=0.6]{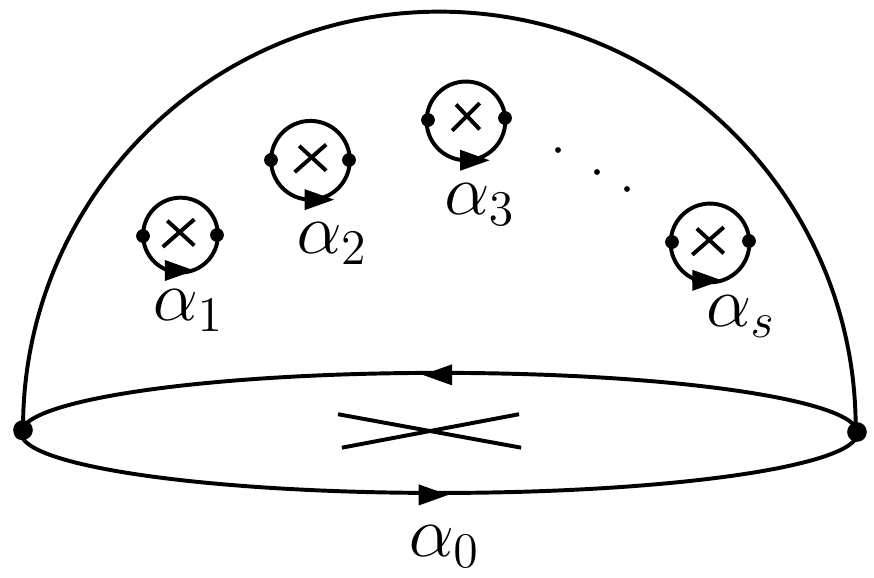}}
\end{picture}

\noindent
Note that there are $s+1$ crosscaps in this picture.
As our basis for $H_1(Q;\Z/2)$ we take
$\alpha_0,\alpha_1,\ldots,\alpha_s$, as shown.
This is an orthonormal basis with respect to the intersection from.  
The loops $\alpha_1,\ldots,\alpha_s$ lift to loops in $X$, but
$\alpha_0$ does not.  So our characteristic class is represented by
the vector $[1,0,0,\ldots,0]\in \F_2^{s+1}$.  
\end{example} 

\begin{example}
\label{ex:T_1,2s}
Next consider $X=T_1^{\anti}+s[DCC]$ and its quotient space $Q=X/C_2$:

\begin{picture}(300,110)(0,-10)
\put(60,-65){\includegraphics[scale=0.6]{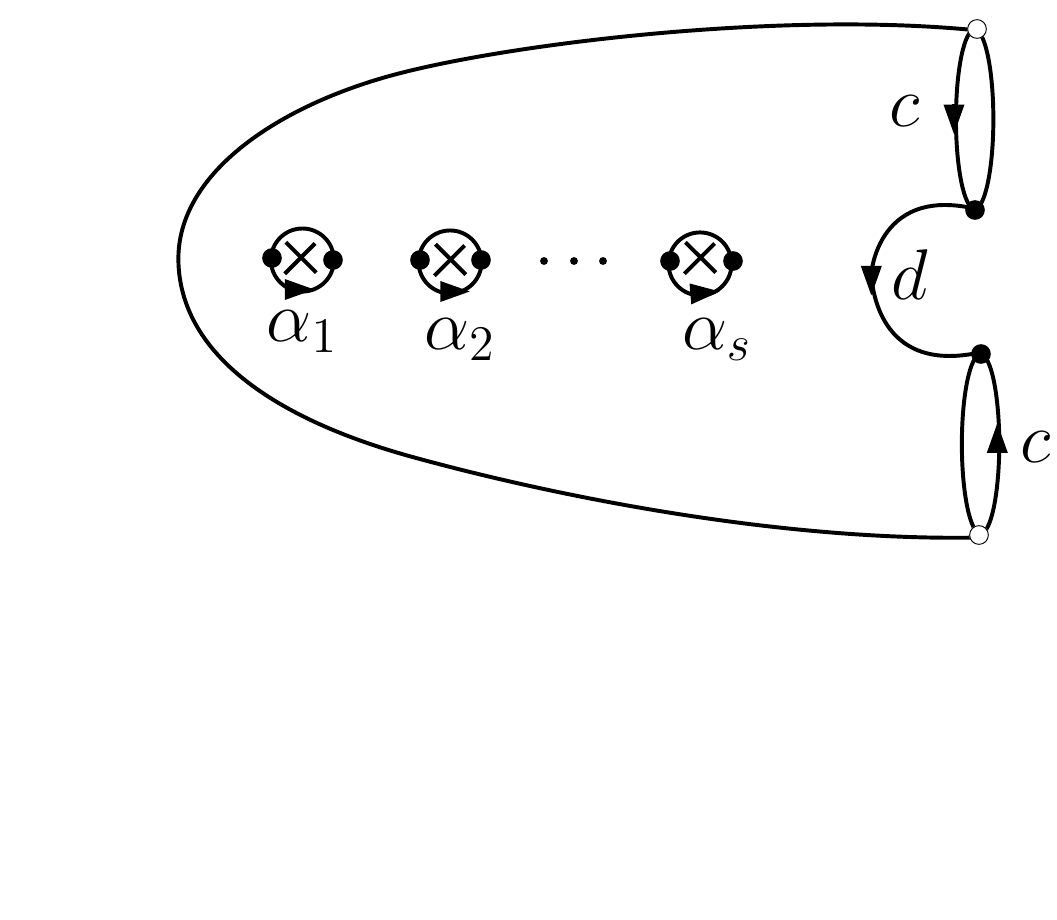}}
\end{picture}

\noindent
Note that $\alpha_1,\ldots,\alpha_s,c$ all lift to loops in
$X$, but $d$ does not.  
Here we again have the problem that $c,d,\alpha_1,\ldots,\alpha_s$
is not an orthonormal basis for $H_1(Q;\Z/2)$.  
We instead use
$c+d,d,\alpha_1,\ldots,\alpha_s$.  The values of our
characteristic class are then $[1,1,0,0,\ldots,0]\in \F_2^{s+2}$. 
\end{example}

\begin{example}
As the final example we consider
 $X=T_2^{\anti}+s[DCC]$ and its quotient space $Q=X/C_2$:

\begin{picture}(300,110)(0,-10)
\put(60,10){\includegraphics[scale=0.6]{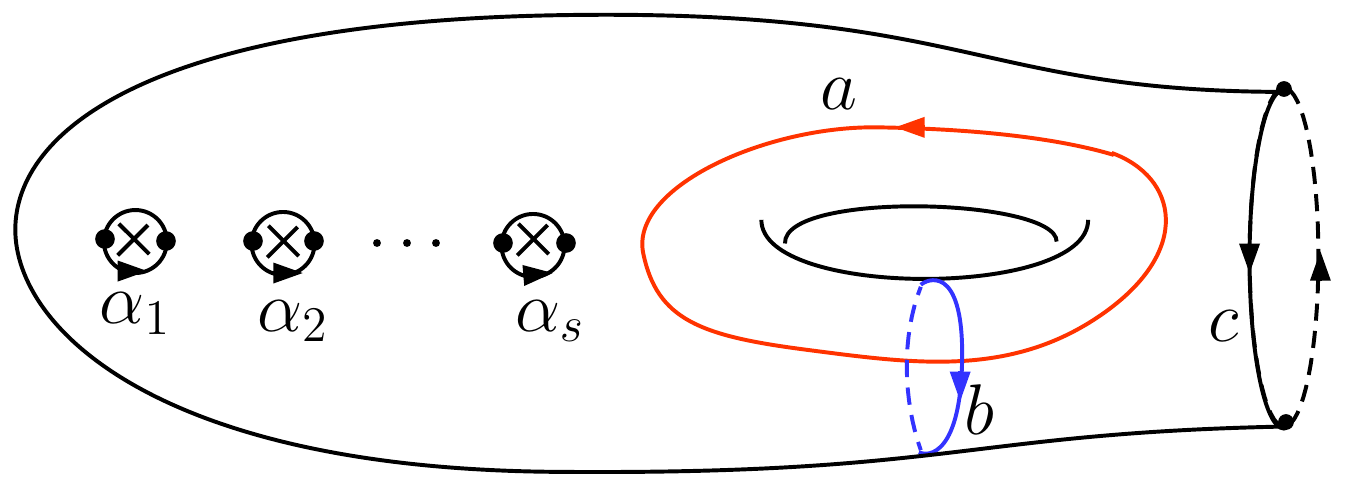}}
\end{picture}

\noindent
Here $\alpha_1,\ldots,\alpha_s,a,b$ all lift to loops in
$X$, but $c$ does not.  The basis
 $a,b,c,\alpha_1,\ldots,\alpha_s$
is not orthonormal, but similarly to 
Example~\ref{ex:T_2-anti-cc}  
we instead use
$a+c,b+c,a+b+c,\alpha_1,\ldots,\alpha_s$.  The values of our
characteristic class are then $[1,1,1,0,\ldots,0]\in \F_2^{s+3}$. 
\end{example}

Hopefully the pattern in the above examples is now clear.  In each
case the extra crosscaps play essentially no role, as the cycles
moving around them are orthonormal and are orthogonal to where the
``main action'' is.  One finds that the characteristic class for
$T_g^{\anti}+s[DCC]$ is really the same as for $T_{g}^{\anti}$ except with $s$
extra zeros on the end.  But as the characteristic class for
$T_{g}^{\anti}$ is $[1,1,\ldots,1]\in \F_2^{g+1}$, this implies the following:

\begin{prop}
\label{pr:Tgs-char}
The
characteristic class for $T_g^{\anti}+s[DCC]$ is the vector
$[1,1,\ldots,1,0,0,\ldots,0]$
in 
$\F_2^{g+1+s}$, where there are $s$ zeros appearing.
\end{prop}

\begin{example}
\label{ex:T1rot}
As one more example, we consider the $C_2$-space
$X=T_1^{\rot}+[DCC]$.  Then $X/C_2$ is a torus with a crosscap, which
is isomorphic to $N_3$.  We depict this space below (where the two
circles that look like boundary components are actually identified):

\begin{picture}(300,100)(0,40)
\put(60,0){\includegraphics[scale=0.5]{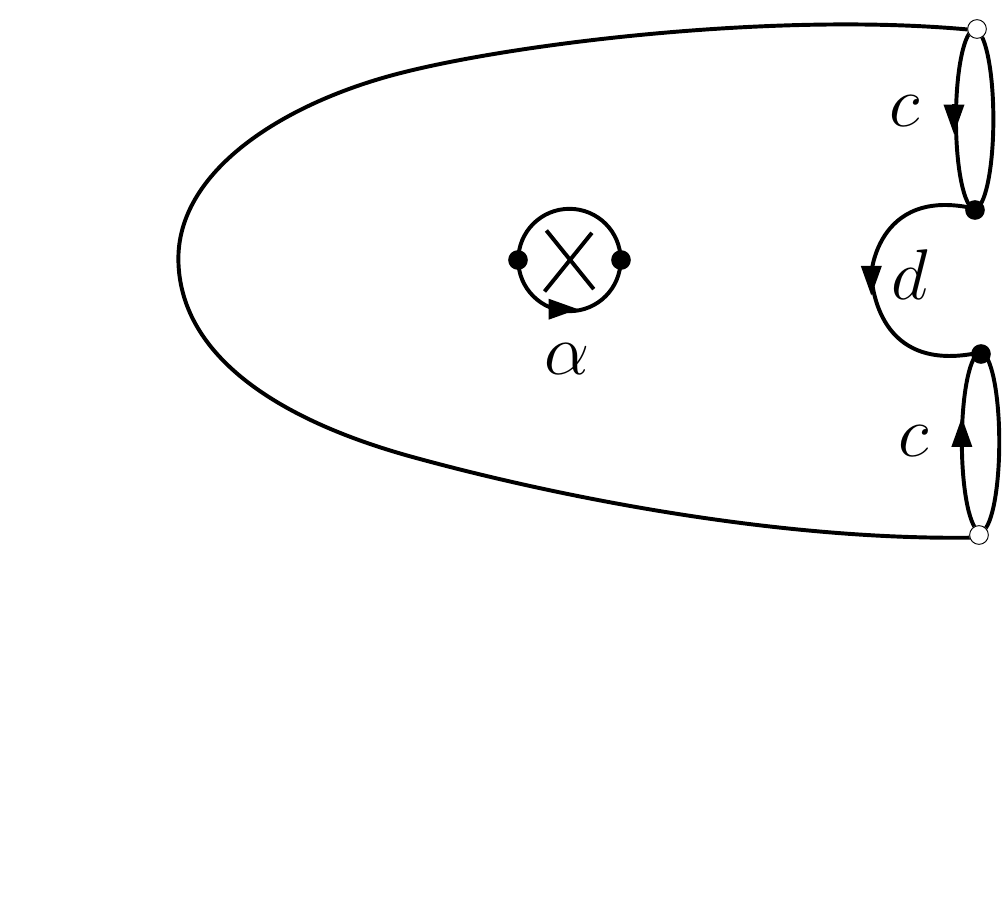}}
\end{picture}

\noindent
The elements $\alpha$, $c$, and $d$ are a basis for $H_1$, and we have
\[ \alpha\cdot \alpha=c\cdot d=1, \quad \alpha\cdot c=\alpha\cdot
d=c\cdot c=d\cdot d=0.
\]
So $\alpha+c+d$, $\alpha+d$, $\alpha+c$ is an orthonormal basis.  
Both $\alpha$ and $c$ lift to loops in $X$, but $d$ does not.  So the
characteristic class for $X$ is $[1,1,0]$.  This is the same
characteristic class as $T_1^{\anti}+[DCC]$, and so we conclude that
these are isomorphic $C_2$-spaces.  The result below is a routine
extension of this:
\end{example}

\begin{prop}
\label{pr:rot=anti}
When $g$ is odd there are $C_2$-equivariant isomorphisms
\[ T_g^{\rot}+s[DCC] \iso T_1^{\anti}+(g+s-1)[DCC]. 
\]
\end{prop}

\begin{proof}
For $g=1$ and $s=1$ this is the content of Example~\ref{ex:T1rot}.
For $g=1$ and $s>1$ the result follows immediately from the $s=1$ case
by adding more
crosscaps.  Finally, the general case then follows
from the $g=1$ case by using
Proposition~\ref{pr:Tg-free} to write $T_g^{\rot}\iso T_1^{\rot}\esum
T_{(g-1)/{2}}$. Then
\[ T_g^{\rot}+s[DCC]\iso T_1^{\rot}+(g-1+s)[DCC]\iso
T_1^{\anti}+(g-1+s)[DCC].
\]
\end{proof}

\subsection{Classification theorems}
The following result completely classifies all free actions on
surfaces with a specified quotient space:

\begin{prop}
\label{pr:classify-base}
\mbox{}\par
\begin{enumerate}[(a)]
\item
For $g\geq 1$ the set $\cS(T_g)$ contains exactly one element,
represented by the 180-degree rotation on $T_{2g-1}$.  
\item For $r\geq 3$ the set $\cS(N_r)$ has three elements.  One is represented
by the antipodal map on $T_{r-1}$, and the others are represented by the
$C_2$-spaces $S^2_a+(r-1)[DCC]$ and $T_1^{\anti}+(r-2)[DCC]$.  
\item The set $\cS(N_1)$ has one element, represented by the antipodal
map on $S^2$.  The set $\cS(N_2)$ has two elements, represented by
$T_1^{\anti}$ and $S^2_a+[DCC]$.
\end{enumerate}
\end{prop}

Note that part (c) can in some ways be regarded as an instance of (b),
where one simply ignores the spaces in the list that don't make sense
(or are repetitions) when $r=1$ and $r=2$.  

\begin{proof}
Corollary~\ref{co:classify} and Lemma~\ref{le:two} imply that
$\cS(T_g)$ has exactly one element.  The rotation action on $T_{2g-1}$
{\it is\/} a free $C_2$-action whose quotient is isomorphic to $T_g$ (for
example because the quotient is orientable with Euler characteristic
$\frac{1}{2}\chi(T_{2g-1})$), and so this rotation action represents
the unique element of $\cS(T_g)$.

For (b), we again note that Corollary~\ref{co:classify} and
Lemma~\ref{le:one} imply that $\cS(N_r)$ has exactly three elements.
The antipodal action on $T_{r-1}$ has a non-orientable quotient of
Euler characteristic equal to $\frac{1}{2}(2-2(r-1))=2-r$, and so it
must be isomorphic to $N_r$.  So this antipodal action represents one
element of $\cS(N_r)$.  

The following paragraph is not essential to our argument, but it is
worth noting that it is easy to see that the antipodal action on
$T_{r-1}$ is the only orientable element of $\cS(N_r)$.  Let
$\tilde{N_r} \ra N_r$ be the orientation cover (a point of which is a
point in $N_r$ together with a local orientation at that point).  If
$X\ra N_r$ is any $2$-fold cover where $X$ is orientable, then there
is a canonical map $X\ra \tilde{N_r}$ giving an isomorphism of
covers of $N_r$.  Being a covering map, this is necessarily
$C_2$-equivariant and therefore an equivariant isomorphism.  Applying
this to $X=T_{r-1}$ with the antipodal action, or to any other free
action on $T_{r-1}$, one sees they are all isomorphic to the action on
$\tilde{N_r}$.

By the last paragraph, the remaining two elements of $\cS(N_r)$ must
be represented by $C_2$-actions on non-orientable surfaces.  Since
the Euler characteristic of such a surface must be
$2\chi(N_r)=2(2-r)=2-(2r-2)$, we conclude we are looking for
$C_2$-actions on $N_{2r-2}$.   Now we simply observe that
$S^2_a+(r-1)[DCC]$ and $T_1^{\anti}+(r-2)[DCC]$ are two such spaces and their
characteristic classes have been computed to be different;
see Example~\ref{ex:T_0,2s}, Example~\ref{ex:T_1,2s}, and
Lemma~\ref{le:one}.

The proof for part (c) is similar, and left to the reader.  
\end{proof}

The main classification theorem for free actions is now an easy consequence
of Proposition~\ref{pr:classify-base}:

\begin{proof}[Proof of Theorem~\ref{th:free-classify}]
Suppose we have a free $C_2$-action on $T_g$, and let $Q=T_g/C_2$.
Then $Q$ is a closed $2$-manifold and
$\chi(Q)=\frac{1}{2}\chi(T_g)=1-g$.  If $g$ is even then $Q$ must be
non-orientable and so $Q\iso N_{g+1}$.  We have computed
$\cS(N_{g+1})$, and we know that it contains only one element whose
underlying space is $T_g$: this is precisely $T_{g}^{\anti}$.  This
proves (a).

When $g$ is odd we can have $Q\iso T_{(1+g)/2}$ or $Q\iso N_{1+g}$.  In
the latter case the reasoning is the same as in the last paragraph.
In the former case, we have computed $\cS(T_{(1+g)/2})$ and know that
it contains exactly one element, represented by $T_{g}^{\rot}$.  This
proves (b).

Now suppose we have a free $C_2$-action on $N_s$, and again let
$Q=N_s/C_2$.  Since an orientable surface cannot be covered by a
non-orientable surface, $Q$ must be non-orientable.  
We also have $\chi(Q)=\frac{1}{2}(2-s)$, and so this
forces $s$ to be even.  We have now proven (c).  

Finally, if $s=2t$ then $\chi(Q)=1-t=2-(1+t)$ and so $Q\iso N_{t+1}$.
We have computed $\cS(N_{t+1})$ and know that when $t\geq 2$ it
contains only two elements whose underlying space is $N_{2t}$, namely
$S^2_a+t[DCC]$ and $T_1^{\anti}+(t-1)[DCC]$.  This proves (e), and the
same analysis applies for (d).
\end{proof}


\section{$C_2$-actions on tori}
\label{se:orient}

Our goal in this section is to understand all $C_2$-actions on the
genus $g$ torus $T_g$.   There are two main cases, depending on
whether the action is orientation-preserving or reversing.
We start by defining two special classes of
actions,
called the {\it spit\/} action and the 
{\it reflection\/} action,   and then we will promptly state the main result.

\subsection{Special constructions}
\label{se:special}
For $r\leq \frac{g}{2}$ let
$T_{g,r}^{\spit}$ be the following $C_2$-space:

\begin{picture}(300,140)
\put(70,0){\includegraphics[scale=0.7]{sail2.pdf}}
\end{picture}

\noindent
The action is the $180$-degree rotation about the dotted axis (the
``spit'').  Note that the number of doughtnut holes lying along the
spit is $g-2r$, and so the number of fixed points is
$2(g-2r)+2=2+2g-4r$.  Usually it is more convenient to remember the
number of fixed points in the notation, and so we will also write
$T_g^{\spit}[F]$ for the version of this construction with $F$ fixed
points (meaning that $F=2+2g-4r$ here).  

Likewise, for $r\leq \frac{g}{2}$ we define $T_{g,r}^{\refl}$ to be
the $C_2$-space depicted by:

\begin{picture}(300,140)
\put(70,0){\includegraphics[scale=0.7]{sail3.pdf}}
\end{picture}

\noindent
The action is reflection in the indicated plane.  Here the fixed set is a
disjoint union of circles, numbering one more than the doughnut
holes that meet the plane.  So the fixed set consists of $g-2r+1$
circles.  Again, usually it is more convenient to have the number of
ovals in the notation: so we will write $T_g^{\refl}[C]$ for the
version of this construction with $C$ ovals (so $C=g-2r+1$).  

Finally, let us define a third type of $C_2$-space denoted
$T^{\anti}[g,r]$.  This is obtained by starting with
a genus $g$ torus with antipodal action, cutting out
the interiors of $r$ disjoint disks together with their conjugates,
and finally identifying the boundary points with
their antipodes:

\begin{picture}(300,140)
\put(60,20){\includegraphics[scale=0.6]{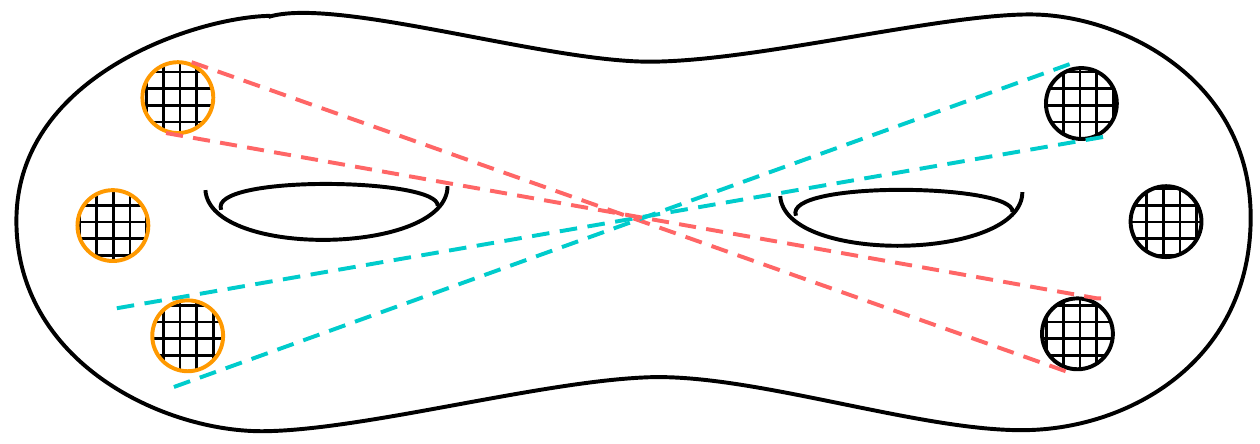}}
\put(287,56){$T^{\anti}[2,3]$}
\end{picture}
 
\noindent
Note that $T^{\anti}[g,r]\iso T_{g+r}$, and the fixed set consists of
$r$ disjoint circles.  Note as well that
$T^{\anti}[g,r]-T^{\anti}[g,r]^{C_2}$ is path-connected, so this
equivariant space is different from the $T_{?,?}^{\refl}$ spaces
defined above.

\begin{remark}
\label{re:T-surgery}
The $C_2$-spaces introduced above all have surgery-based descriptions
as well:
\addtocounter{subsection}{1}
\begin{align}
T_g^{\spit}[F] &\iso 
\Bigl [S^{2,2}+\bigl (\tfrac{F}{2}-1\bigr
)[S^{1,1}-\antitube]\Bigr ] \esum T_{\bigl (\frac{2+2g-F}{4}\bigr )}
\\
\notag
&\iso 
\Doub\Bigl (T_{\bigl (\frac{2+2g-F}{4}\bigr )};S^{1,1}\Bigr )+
(\tfrac{F}{2}-1\bigr
)[S^{1,1}-\antitube]
\end{align}
\addtocounter{subsection}{1}
\begin{align}
 T_g^{\refl}[C]&\iso 
\Bigl [S^{2,1}+\bigl (C-1\bigr
)[S^{1,0}-\antitube]\Bigr ] \esum T_{\bigl (\frac{1+g-C}{2}\bigr )} \\
\notag
&\iso 
\Doub\Bigl (T_{\bigl (\frac{1+g-C}{2}\bigr )};S^{1,0}\Bigr )+
(C-1\bigr
)[S^{1,0}-\antitube]
\end{align}
and
\addtocounter{subsection}{1}
\begin{align}
 T^{\anti}[g,r]\iso T_g^{\anti}+r[S^{1,0}-\antitube].
\end{align}
\end{remark}

\begin{prop}
\label{pr:neededonce}
For $g\geq 0$ and $0\leq n\leq \frac{2+2g}{4}$, there are equivariant
isomorphisms
\[ T_g^{\spit}[2+2g-4n]\iso \begin{cases}
S^{2,2}+g[S^{1,1}-\antitube] & \text{if
  $n=0$},\\[0.1in]
T_{2n-1}^{\rot}+(g+1-2n)[S^{1,1}-\antitube] & \text{if $n>0$.}
\end{cases}
\]
\end{prop}

\begin{proof}
In the case $n=0$ this is obvious.  The general case follows
immediately once we establish $T_g^{\spit}[2]\iso
T_{g-1}^{\rot}+[S^{1,1}-\antitube]$ when $g$ is even and $g\geq 2$.  
Here the proof is geometric.  Observe that the $C_2$-action on a
standard $S^{1,1}$-antitube is 180-degree rotation about the spit
passing through the two fixed points.  
Start with the standard model of $T_{g-1}^{\rot}$, with its central
axis of rotation, and sew in an $S^{1,1}$-antitube so that its axis of
rotation matches the one of the torus.  The resulting space is
transparently isomorphic to $S^{2,2}\esum T_{g/2}$, which is a model
for $T_g^{\spit}[2]$.  The following pictures demonstrate the case
$g=4$:

\begin{picture}(300,150)(0,-60)
\put(-50,10){\includegraphics[scale=0.45]{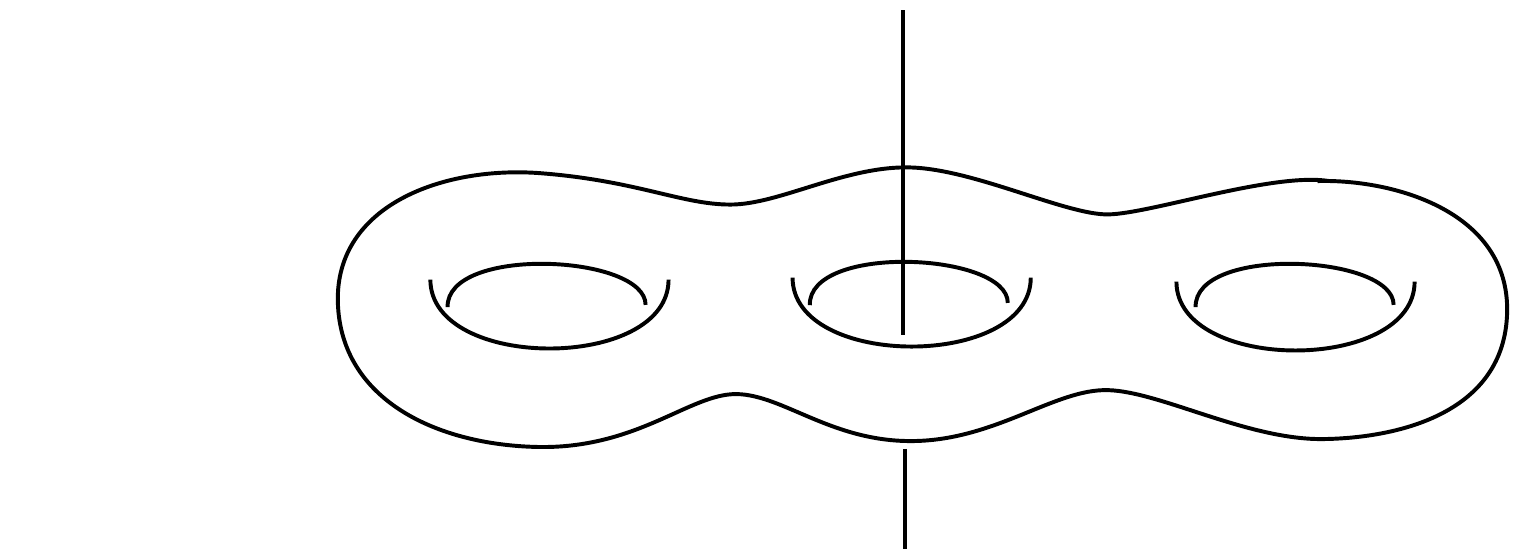}}
\put(160,40){$\longrightarrow$}
\put(150,10){\includegraphics[scale=0.45]{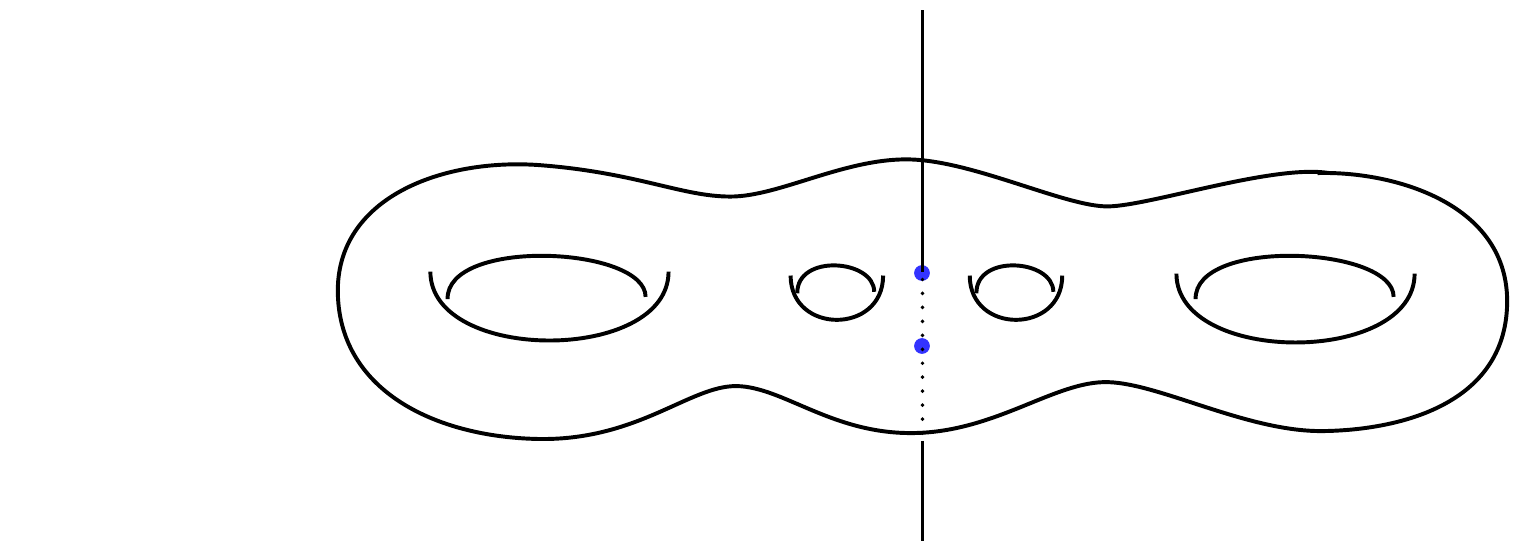}}
\put(68,-50){\includegraphics[scale=0.45]{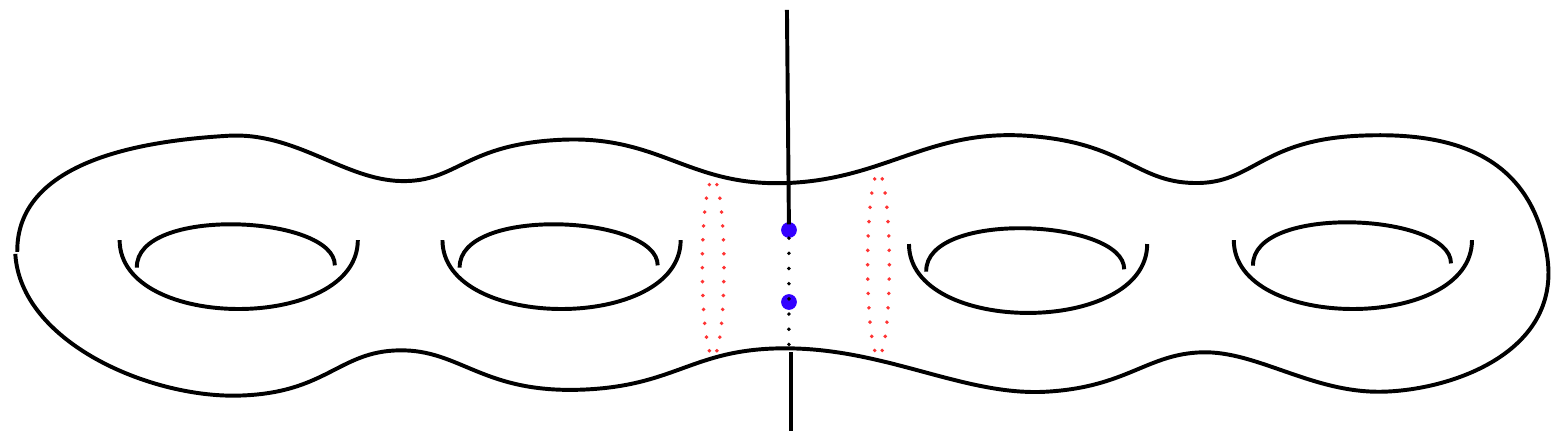}}
\put(280,-5){$\swarrow$}
\end{picture}
\end{proof}

Here is the main theorem concerning $C_2$-actions on $T_g$:

\begin{thm}
\label{th:T_g-action}
For $g\geq 0$ there are exactly $4+2g$ distinct involutions on $T_g$.
These are:
\begin{enumerate}[(i)]
\item The $2+\lceil \frac{g}{2}\rceil$ orientation-preserving actions,
namely 
\begin{itemize}
\item The trivial action,
\item $T_{g,r}^{\spit}$ for $0\leq r\leq \frac{g}{2}$, (equivalently,
$T_{g}^{\spit}[F]$ for $2\leq F\leq 2+2g$ and $F\equiv 2+2g$
mod $4$),  
\item  $T_{g}^{\rot}$ (when $g$ is odd).
\end{itemize}
\item The $2+g+\lfloor \frac{g}{2}\rfloor$ orientation-reversing
actions, namely
\begin{itemize}
\item $T_{g,r}^{\refl}$ for $0\leq r\leq \frac{g}{2}$, 
(equivalently, $T_g^{\refl}[C]$ for $1\leq C\leq g+1$ and $C\equiv g+1$
mod $2$), 
\item $T^{\anti}[u,g-u]$ for $0\leq u\leq g$, also known as
$T_u^{\anti}+(g-u)[S^{1,0}-\antitube]$.  
\end{itemize}
\end{enumerate}
\end{thm}

We will need to develop some preliminary results before giving the
proof.  Note, however, that  that the 
following corollary is an immediate consequence:

\begin{cor}
\label{co:P3-orient}
An involution on $T_g$ is completely determined, up to equivariant
isomorphism, by the following invariants:
\begin{enumerate}[(i)]
\item Whether it is orientation-preserving or reversing;
\item The number of isolated fixed points;
\item The number of circles in the fixed set;
\item Whether $T_g-(T_g)^{C_2}$ is path-connected or not.  
\end{enumerate}
Note that (i) is superfluous information as long as the fixed set is
nonempty.  
\end{cor}

%
%
%

\subsection{General considerations}

An involution on $T_g$ will be either orientation-preserving or
orientation-reversing, and of course this can be detected by looking
in a neighborhood of any single point. 
The fixed set of an involution on $T_g$ must consist of isolated
points and circles, and it cannot consist of both at once:

\begin{lemma}
\label{le:isolated-or-ovals}
The fixed set of a nontrivial involution on $T_g$ must either be a finite set of
points or else a disjoint union of copies of $S^1$.  If the fixed set
is nonempty and finite, the action is orientation-preserving.  If the
fixed set is a nonempty disjoint union of circles, the action is
orientation-reversing.  
\end{lemma}

\begin{proof}
If there is an isolated fixed point, then locally around this point
the action must look like an $S^1_a$-cap; so the action is
orientation-preserving around this point, and therefore
orientation-preserving overall.  In contrast, if there is a fixed
circle then locally around the circle the action looks like
$S^{1,0}\times \R^{1,1}$, and so it is orientation-reversing.
Since an involution cannot be both orientation preserving and
reversing, the fixed set cannot consist of both isolated points and
copies of $S^1$.
\end{proof}

\begin{prop}
\label{pr:F-bounds}
Suppose an orientation-preserving involution on $T_g$ 
has finite fixed set of size $n$.  Then
$n\leq 2+2g$ and $n\equiv 2+2g$ mod $4$.  
\end{prop}

\begin{proof}
Let $F\subseteq T_g$ denote the fixed set.  Then $T_g-F \ra
(T_g-F)/C_2$ is a $2$-fold covering space, so
\[ 2-2g-n=\chi(T_g-F)=2\cdot \chi((T_g-F)/C_2).
\]
Since the action is orientation-preserving, the quotient space
$(T_g-F)/C_2$ is an orientable manifold with boundary: it is
$T_{g'}-F$ for some $g'$.  So our Euler characteristic identity
becomes
\[ 2-2g-n=2[(2-2g')-n]=4-4g'-2n \]
and so $n=2+2g-4g'$.  
\end{proof}

\begin{remark}
For good measure we give a second proof of the above result, using the
Lefschetz Fixed Point Theorem.  Indeed, that theorem implies that
$n=1-\tr(\sigma_*)+1$ where $\sigma_*$ is the map on $H_1(T_g;\Z)$
induced by $\sigma$.  This can be represented by a $2g\times 2g$
matrix whose square is the identity.  So the Jordan canonical form of
$\sigma_*$ will have only $1$ and $-1$ along the diagonal, say $s$
copies of $1$ and $2g-s$ copies of $-1$.  Then $\tr(\sigma_*)=2s-2g$,
and $n=2+2g-2s$.  To end the proof we note that since $\sigma$ is
orientation-preserving the map $\sigma_*$ is symplectic (it preserves
the cup product), and so it has determinant $1$; hence $2g-s$ must be
even, and so $s$ is even.
\end{remark}

The following result goes back to Harnack \cite{H}:

\begin{prop}[Harnack's Theorem]
\label{pr:C-bounds}
Let $\sigma$ be an orientation-reversing action on $T_g$, and let $C$
be the number of ovals in the fixed set.  If the action is separating
then $C\leq 1+g$ and $C\equiv 1+g$ (mod $2$).  If the action is
non-separating then $C\leq g$.  
\end{prop}

\begin{proof}
For the separating case, the two path components are homeomorphic
copies of an
orientable $2$-manifold  with boundary consisting of $C$ circles.  So
the Euler characteristic of one of these path components is $2-2g'-C$,
where $g'$ is the genus.
The torus $T_g$ is obtained by gluing these two copies together along
their boundary, so we have
\[ 2-2g=\chi(T_g)=2(2-2g'-C).\]
So $1-g=2-2g'-C$, or $C=1+g-2g'$.  The conclusions are immediate.  

For the non-separating case, we simply recall one interpretation of
the genus: removing any set of $1+g$ disjoint closed curves from $T_g$ must
disconnect the manifold.  Therefore the non-separating hypothesis
implies that  $C\leq g$.  
\end{proof}

\begin{remark}
All of the possibilities for $C$ left open by Proposition~\ref{pr:C-bounds}
actually occur, as we have seen from the constructions in
Section~\ref{se:special}.
\end{remark}

%

\subsection{The proof of the main classification theorem in the
  orientable case}
Our main techniques will be $S^{1,0}$-surgery around ovals and
$S^{1,1}$-surgery around pairs of fixed points, as in
Remark~\ref{re:S11-fp}.

\begin{lemma}
\label{le:S^2}
Up to isomorphism, the only $C_2$-actions on $S^2$ are $S^{2,0}$,
$S^{2,1}$, $S^{2,2}$, and $S^2_a$.
\end{lemma}

\begin{proof}
Let $X$ denote the $2$-sphere with a nontrivial $C_2$-action.  If the
action is orientation-preserving then the fixed set is finite by
Lemma~\ref{le:isolated-or-ovals}.  The number of fixed points is at most $2$ and
equivalent to $2$ modulo $4$ by Proposition~\ref{pr:F-bounds}, so there are
exactly two fixed points.  As in Remark~\ref{re:S11-fp}, there is an
$S^{1,1}$-antitube in $X$ surrounding these fixed points.  Since $X$
has genus zero, removing this antitube must separate $X$ into two
components.  The only possibility is that $X$ consists of this
antitube together with a $\Z/2\times S^1$-cap, showing that $X\iso 
S^{2,2}$.  

Next suppose that the action on $X$ is orientation-reversing.  If there are
no fixed points then we know $X\iso S^2_a$ by
Theorem~\ref{th:free-classify}, since there is only one free action on
$S^2$.  If there are fixed points then the fixed set is a union of
finitely-many ovals by Lemma~\ref{le:isolated-or-ovals}, and the
number of ovals is at most $1$ by Harnack's Theorem.  
So there is exactly one oval, and a tubular neighborhood of this oval must be an
$S^{1,0}$-antitube.  Just as in the orientation-preserving case, the
fact that $X$ has genus zero implies that it must consist of this
antitube together with a $\Z/2\times S^1$-cap.  That is, $X\iso S^{2,1}$.
\end{proof}

\begin{proof}[Proof of Theorem~\ref{th:T_g-action}]
Let $X$ denote $T_g$ with a nontrivial $C_2$-action.
First consider the case where the action is orientation-reversing.
The action
is either separating or non-separating.  If it is separating, 
then by Proposition~\ref{pr:doub}  $X\iso\Doub(S,C)$ for
some orientable surface $S$.  The doubled space is evidently homeomorphic to
$T_{g}^{\refl}[C]$, and so this case is complete.

Next assume the action on $X$ is non-separating.  We proceed by
induction on the genus.  When $g=0$ there are only four $C_2$-actions
on $S^2$ by Lemma~\ref{le:S^2} and the only one that is both
orientation-reversing and non-separating is $S^2_a\iso
T^{\anti}[0,0]$.  Now assume that $g\geq 1$, and let $C$ be the number
of components in the fixed set of $X$.  If $C=0$ then we know by
Theorem~\ref{th:free-classify} that $X\iso T_g^{\anti}=T^{\anti}[g,0]$.  When
$C\geq 1$, pick one of the ovals to remove, 
replacing it with a
$\Z/2\times S^1$-cap: write $X-[S^{1,0}-\antitube]$ for
this new space.
The action here is
still orientation-reversing and non-separating, and the genus has
decreased to $g-1$.  So by induction there is a $C_2$-equivariant isomorphism
\[ X-[S^{1,0}-\antitube]\iso T^{\anti}[u,g-1-u]
\]
for some $0\leq u\leq g-1$.  But then
\[ X-[S^{1,0}-\antitube]\iso
T^{\anti}[u,g-1-u]\iso T^{\anti}[u,g-u]-[S^{1,0}-\antitube]
\]
and so Proposition~\ref{pr:surg-subtract} immediately shows $X\iso
T^{\anti}[u,g-u]$.

Finally, we deal with the case where the action on $X$ is
orientation-preserving.  We know in this case that the fixed set is
finite and consists of an even number of points by
Proposition~\ref{pr:F-bounds}.  
We proceed by
induction on the number of fixed points $F$.  If $F=0$ then we know by
Theorem~\ref{th:free-classify}
 that $g$ is odd and $X$ is $C_2$-equivariantly isomorphic to
 $T_g^{\rot}$.
Since $F$ must be odd, the next case is $F=2$.    As 
in Remark~\ref{re:S11-fp} we can find an $S^{1,1}$-antitube in $X$
that passes through the two fixed points.  If slicing the tube
disconnects the space then $g$ is even and our $C_2$-space is
$\Doub(T_{\frac{g}{2}},1:S^{1,1})$, but this is
$T^{\spit}_{g}[2]$.  If slicing the tube does not disconnect the space
then we can do $S^{1,1}$-surgery to obtain a new space
$X-[S^{1,1}-\antitube]$ of genus $g-1$.  This new space has no fixed
points and the action is still orientation-preserving, so by induction $g-1$ is odd 
and $X-[S^{1,1}-\antitube]\iso T_{g-1}^{\rot}$.  It follows that
\[ X\iso T_{g-1}^{\rot}+[S^{1,1}-\antitube]\iso T_g^{\spit}[2]
\]
where we have used Proposition~\ref{pr:neededonce} for the second isomorphism.
If it seems strange that both cases led to the space $T_g^{\spit}[2]$,
see Remark~\ref{re:strange} below.

The final case is when $F\geq 4$.
Pick distinct fixed points $a$ and
$b$ and again use Remark~\ref{re:S11-fp} to produce  an $S^{1,1}$-antitube
surrounding these points.  If slicing this tube disconnects the space,
our $C_2$-space would be
$\Doub(T_{\frac{g}{2}},1:S^{1,1})$; but this is not possible as
$F>2$.  So we can slice the tube and do 
$S^{1,1}$-surgery to obtain a new connected space
$X-[S^{1,1}-\antitube]$.  The action here
is still
orientation-preserving, the genus has gone down to $g-1$, and the number of
fixed points is $F-2$.  So
by induction we know there is a $C_2$-equivariant isomorphism 
\[ 
X-[S^{1,1}-\antitube] \iso 
T_{g-1}^{\spit}[F-2].
\] 
But
\[ T_{g-1}^{\spit}[F-2]\iso T_{g}^{\spit}[F]-[S^{1,1}-\antitube] 
\]
and therefore Proposition~\ref{pr:surg-subtract} implies $X\iso T_{g}^{\spit}[F]$.  
\end{proof}

\begin{remark}
\label{re:strange}
In the above proof we saw something slightly strange.  In
$T^{\spit}_g[2]$ one way of doing $S^{1,1}$-surgery about the two
fixed points results in a disconnected space, whereas another way of
doing the surgery results in a connected space.    This of course
happens all the time, but this is the first place we have had to
confront the phenomenon.  The picture below shows the case $g=2$.

\begin{picture}(300,80)
\put(70,5){\includegraphics[scale=0.5]{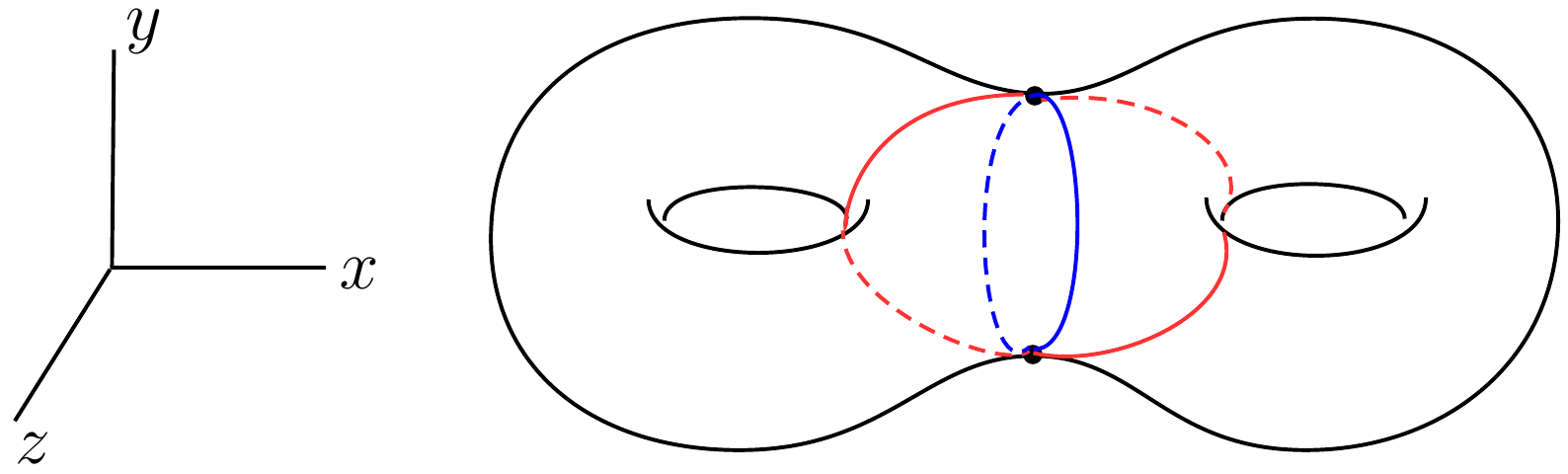}}
\end{picture}

\noindent
The action is rotation about the $z$-axis, which is where
the two marked points lie.  
Removing a tubular neighborhood of the blue $S^{1,1}$ disconnects the
space, whereas removing a tubular neighborhood of the red $S^{1,1}$
does not.  

\end{remark}


\section{Invariants}

In this section we return to some generalities, as a prelude to our
classification of $C_2$-actions in the non-orientable case.  We
investigate the basic invariants of $C_2$-actions, how they are
affected by certain types of surgery, and some general restrictions
that exist on such invariants.

\medskip

If $C_2$ acts on a surface $X$ then the fixed set $X^{C_2}$ necessarily
consists of isolated fixed points together with copies of $S^1$.  When
$X$ is orientable the classification of $C_2$-actions turns
out to be simple because both cannot happen at the same time: the fixed set
can contain either isolated fixed points or copies of $S^1$, but not
both.  In the non-orientable case this no longer holds true;
there are many more $C_2$-actions around, and classifying them is
much more involved.


For the result below, recall the invariants $F$, $C$, $C_+$, $C_-$, $\beta$, and the
$Q$-sign,
attached to any $C_2$-action on a $2$-manifold.  These were defined in
Section~\ref{se:intro}.

\begin{prop}
\label{pr:surg-invariance}
The following quantities are 
invariant under $S^{1,0}$-surgery,
$S^{1,1}$-surgery, and $FM$-surgery:
\begin{enumerate}[(i)]
\item $F+2C-\beta$,
\item The $Q$-sign,
\item The residue of $F-C_-$ modulo $2$.
\end{enumerate}  
\end{prop}

\begin{proof}
All three statements are trivial to check.  
For example, for (i) note that $FM$-surgery decreases $F$ by
one while increasing both $\beta$ and $C$ by one.  Similarly,
adding an $S^{1,1}$-antitube increases $F$ by two, increases $\beta$
by two, and has no effect on $C$.  Adding an $S^{1,0}$-antitube
increases $C$ by one and $\beta$ by two.  

We leave the reader to think about (iii), which is similar.  

For (ii), note that at the level of quotients 
$S^{1,0}$-surgery has the effect
of removing an open disk and gluing on one end of a cylinder.  
$FM$-surgery turns out to have the same effect, since $S^1(M)/C_2$ is
a cylinder.  
Similarly, $S^{1,1}$-surgery has
the effect of  removing an open disk (in the quotient) and then gluing
back on a cap.  None of these procedures change orientability.  
\end{proof}

\begin{cor}
\label{co:Qsign-rule}
The $Q$-sign of the equivariant space
\[ X+q[DCC]+r[S^{1,0}-\antitube]+s[S^{1,1}-\antitube]+t[FM]
\]
is negative if $q>0$, and is equal to the $Q$-sign of $X$ when $q=0$.
\end{cor}

\begin{proof}
For the first statement, just note that adding dual crosscaps to $X$
has the affect of adding a single crosscap to the quotient.  
The second statement is just Proposition~\ref{pr:surg-invariance}(ii).
\end{proof}

\begin{lemma}
\label{le:iso-fp}
There does not exist a $C_2$-action on a $2$-manifold $X$ having exactly
one fixed point.
\end{lemma}

\begin{proof}
Assume $X$ has exactly one fixed point.
If the underlying space of $X$
is orientable then the action is orientation-preserving, and
so by 
Proposition~\ref{pr:F-bounds} the number of fixed points is even.  So
$X\iso N_r$, for some $r\geq 1$.

Let $Y$ be the space
obtained from $X$ by cutting out a small open
disk around the fixed point, so that the boundary of $Y$ is a copy of
$S^1_a$.  Then $\chi(Y)=2-r-1=1-r$, and the $C_2$-action on $Y$ is
free.  So $1-r=\chi(Y)=2\chi(Y/C_2)$, hence $r$ is odd.  

Clearly $Y/C_2$ is a $2$-manifold whose boundary is a circle.  Since
$Y$ is non-orientable, so is $Y/C_2$.  An Euler characteristic
argument shows $Y/C_2$ is $N_{\frac{r+1}{2}}$ with an open disk
removed.
Make this identification and regard $Y/C_2$ as a subspace of
$N_{\frac{r+1}{2}}$.    

Since $Y/C_2\inc N_{\frac{r+1}{2}}$ is
an isomorphism on $H^1(\blank;\Z/2)$, every principal $C_2$-bundle
on $Y/C_2$ is pulled back from one on 
$N_{\frac{r+1}{2}}$.   By an Euler characteristic argument,
the only $C_2$-bundle on $N_{\frac{r+1}{2}}$
having
non-orientable total space is $N_{r-1}$, and so we conclude that the
underlying space of $Y$ is $N_{r-1}$ with two open disks removed.
However, $Y$ started out as an  $N_r$ with an open disk removed.  This is a
contradiction.  
\end{proof}

\begin{prop}
\label{pr:surg-reduce}
Let $X$ be a $C_2$-equivariant $2$-manifold where the action is
non-free, nontrivial, and where $X$ is not of the form $\Doub(S,1:S^{1,0})$ or
$\Doub(S,1:S^{1,1})$ for any surface $S$.  
If $\beta(X)>0$ then
 $X$ can be obtained via $S^{1,1}$-, $S^{1,0}$-, or $FM$-surgery
from an equivariant $2$-manifold of smaller $\beta$-genus.  
\end{prop}

\begin{proof}
Suppose first that $C=0$.  By Lemma~\ref{le:iso-fp} 
 we then have $F>1$.  
Choose two isolated fixed points $a$ and $b$, and obtain an
$S^{1,1}$-antitube surrounding them as in Remark~\ref{re:S11-fp}.  
Since $X$ is not of the form $\Doub(S,1:S^{1,1})$, cutting this
antitube does not disconnect the surface.  So we can do 
$S^{1,1}$-surgery on this antitube to reduce the genus.  

Next suppose $C>0$, and let $\cO$ be an oval.  If $\cO$ is one-sided
then we can do $MF$-surgery around $\cO$ to reduce the genus.
If $\cO$ is two-sided then our hypotheses guarantee that 
it is non-separating, so we can do
$S^{1,0}$-surgery around $\cO$ to  again reduce the genus.    
\end{proof}

\begin{cor}
\label{co:basic}
Start with the $C_2$-equivariant $2$-spheres, the free $C_2$-actions
on surfaces, the spaces $\Doub(S,1)$, and the spaces
$\Doub(S,1:S^{1,1})$.  Then
every nontrivial $C_2$-equivariant $2$-manifold can be built from these basic
spaces via repeated $S^{1,0}$-, $S^{1,1}$-, and $FM$-surgeries.
\end{cor}

\begin{proof}
Immediate.
\end{proof}

The first statement of the following result is \cite[Theorem 10]{S}:

\begin{thm}[Scherrer's Theorem]
\label{th:invariants}
Let $X$ be a $2$-manifold with a nontrivial $C_2$-action.  Then
$F+2C\leq \beta+2$ and $F\equiv \beta\equiv C_-\ (\text{mod}\ 2)$.  
Moreover, if the $Q$-sign of $X$ is negative then
$F+2C\leq \beta$.  
\end{thm}

\begin{proof}
For the moment ignore the statement about the $Q$-sign.  By ``the
relations'' we will mean the inequality and the congruences, and the
proof of these
will proceed via induction on $\beta$.  When $\beta=0$ there are
three possibilities for $X$: $S^2_a$, $S^{2,1}$, and $S^{2,2}$.  The
relations can be checked by hand for each of these.

Next assume $\beta>0$.  There are four cases to be handled
separately.  If $X\iso \Doub(S,1)$ then $F=C_-=0$, $C=1$, and $\beta$
is even: so the relations hold.  If $X\iso \Doub(S,1:S^{1,1})$ then
$F=2$, $C=C_-=0$, and $\beta$ is again even: so the relations hold.  
If $X$ is free then $F=C=0$ and $\beta$ must be even by 
Theorem~\ref{th:free-classify}, so the
relations are again verified.  

In the final case, where $X$ is none of the above things, then by 
Proposition~\ref{pr:surg-reduce} we can do $S^{1,0}$-, $S^{1,1}$-, or 
$MF$-surgery on $X$ to reduce the $\beta$-genus.  By
Proposition~\ref{pr:surg-invariance} this surgery does not change $F+2C-\beta$ or
the mod $2$ residues of $F-\beta$ and $F-C_-$.  So we are done by induction.

Now consider the final statement about the $Q$-sign.  This is again
done by induction on $\beta$, following exactly the pattern of the
above.  In the $\beta=0$ case one must have $X=S^2_a$, and the
inequality is checked by hand.  For the induction step we must run
through the four cases again.  If $X=\Doub(S,1)$ then $S$ must be
nonorientable since the $Q$-sign of $X$ is negative; so $\beta(S)\geq
1$, hence $\beta(X)\geq 2$,
and the inequality is immediate.  The analysis for
$X=\Doub(S,1:S^{1,1})$ is identical, and when $X$ is free the
inequality is trivial.  The remaining case, where $X$ is none of these
things, proceeds exactly as before.  
\end{proof}

\subsection{$C_2$-actions on projective spaces and Klein bottles}
As an application of what we have done so far, we can completely
classify all the $C_2$-actions on $\RP^2$ and on a Klein bottle.

\begin{cor}
Up to isomorphism there is exactly one nontrivial $C_2$-structure on $\RP^2$, namely
$S^{2,2}+[FM]$.  
\end{cor}

\begin{proof}
By Theorem~\ref{th:invariants} we must have $F+2C\leq 3$ and $F\equiv
C_-\equiv 1 \ (\text{mod}\ 2)$.  The only solution is $F=1$ and
$(C_+,C_-)=(0,1)$.  Now do $MF$-surgery on the one-sided oval.  This
produces a $C_2$-space of genus zero, with two fixed points.  This
space must of course be $S^{2,2}$ (using Lemma~\ref{le:S^2}),
which shows that our original space was $S^{2,2}+[FM]$.
\end{proof}

The situation for Klein bottles is much more interesting.  
Here $\beta=2$, so
we must have $F+2C\leq 4$ and $F\equiv C_-\equiv 0\ (\text{mod}\  2)$.  
The possible taxonomies are therefore the following:

\vspace{0.1in}

\bgroup
\def\arraystretch{1.5}
\begin{tabular}{cccc}
 $[0,0:(0,0)]$ & $[0,1:(1,0)]$ & $[0,2:(2,0)]$ & $[0,2:(0,2)]$ 
\\
$[2,0:(0,0)]$ & $[2,1:(1,0)]$ & $[4,0:(0,0)]$.
\end{tabular}
\egroup

\vspace{0.1in}

\noindent
Morover, we know we can make all equivariant Klein bottles using
$S^{1,0}$-, $S^{1,1}$-, and $FM$-surgery starting from the basic
spaces listed in Corollary~\ref{co:basic}.  By genus considerations we
can only start with the three $2$-spheres, the spaces $\Doub(S,1)$
and $\Doub(S,1:S^{1,1})$ where $S=\RP^2$, and the unique free
action on the Klein bottle, $S^2_a+[DCC]$.  We can only do surgery on the
$2$-spheres, since in the remaining cases the genus would become too
large.

One last thing before we just list all the possibilities.  The action
on $S^{2,1}$ is orientation-reversing, and the action on an
$S^{1,0}$-antitube is also orientation-reversing.  The space
$S^{2,1}+[S^{1,0}-\text{antitube}]$ is therefore orientable, and so it
is a torus not a Klein bottle.  In general, we will get a Klein bottle
only when the orientation type of our sphere opposes the
orientation-type of the antitube we are adding to it.  With this in
mind, here is a list of all possible nontrivial $C_2$-actions on Klein
bottles:

\vspace{0.1in}

\bgroup
\def\arraystretch{1.5}
\begin{tabular}{c||ccccc}
Space & $F$ & $C$ & $C_+$ & $C_-$ & $Q$-sign \\
\hline
free action & 0 & 0 & 0 & 0 & - \\
$\Doub(\RP^2,1)$ & 0 & 1 & 1 & 0 & -\\
$\Doub(\RP^2,1:S^{1,1})$ & 2 & 0 & 0 & 0 & -\\
$S^2_a+[S^{1,1}-\text{antitube}]$ & 2 & 0 & 0 & 0 & -\\
$S^{2,1}+[S^{1,1}-\text{antitube}]$ & 2 & 1 & 1 & 0 & +\\
$S^{2,2}+[S^{1,0}-\text{antitube}]$ & 2 & 1 & 1 & 0 & +\\
$S^{2,2}+2[FM]$ & 0 & 2 & 0 & 2 & + \\ 
\end{tabular}
\egroup

\vspace{0.1in}

Note that the $[2,2:(2,0)]$ and $[4,0:(0,0)]$ taxonomies are not 
realizable.   Also note
that what our analysis shows is that there are {\it at most\/} seven
possible nontrivial $C_2$-actions on a Klein bottle.  There are two
places in our list where we might have actions that are isomorphic,
because they have the same taxonomy.  In these cases it is not hard to
see by hand that in fact the actions {\it are\/} isomorphic, so that
there are exactly five nontrivial $C_2$-actions on the Klein bottle.  
We leave this final piece to the reader, but see Corollary~\ref{co:fundiso}
and the proof of Theorem~\ref{th:fundiso} for help if needed.

\section{$C_2$-actions on non-orientable surfaces}
\label{se:non-orient}

Given nonnegative integers $r$, $F$, $C$, $C_+$, $C_-$, let
$N_r[F,C:(C_+,C_-)]$ denote the set of all isomorphism classes of
$C_2$-spaces whose underlying space is $N_r$ and whose taxonomy is
$[F,C:(C_+,C_-)]$.  Our goal will be to completely describe this set,
both listing all of the elements and giving explicit formuas for
their number.

\subsection{Some fundamental isomorphisms}

\begin{thm}
\label{th:fundiso}
There are equivariant isomorphisms
\[ S^{2,2}+[DCC] \iso S^2_a + [S^{1,1}-\antitube]\iso
S^{2,2}+[S^1_a-\antitube]
\]
and
\[ S^2_a+[DCC]+[S^{1,1}-\antitube] \iso T_1^{\anti} + [S^{1,1}-\antitube].
\]
\end{thm}

\begin{proof}
For the isomorphisms on the first line we argue as follows.  The space
$S^{2,2}$ can be modeled as an $S^{1,1}$-antitube with caps added to
the top and bottom.  So $S^{2,2}+[DCC]$ can be modeled by cutting out
these caps and replacing them with crosscaps:

\begin{picture}(300,120)(20,0)
\put(50,0){\includegraphics[scale=0.65]{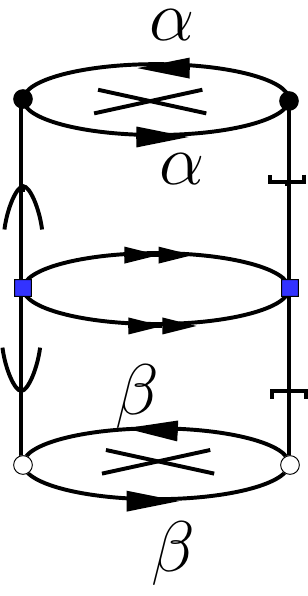}}
\put(160,60){\huge{$\iso$}}
\put(215,0){\includegraphics[scale=0.6]{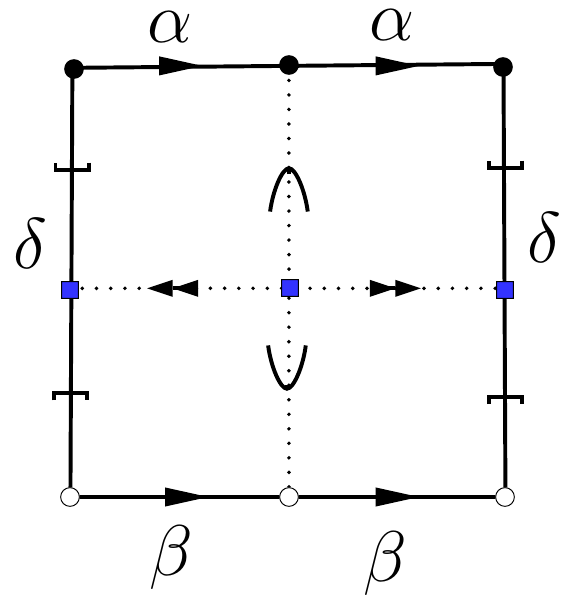}}
\end{picture}

\noindent
There are competing conventions in these pictures, so let us
explain.  For the most part we use dual symbols to depict the
$C_2$-action, but here we also need to depict identifications.  The
two edges labelled $\alpha$ are identified with each other, as are the
two edges labelled $\beta$.  The square on the right is obtained by
cutting the cylinder along a vertical seem and unrolling it, and so
the two edges labelled $\delta$ are identified.  Note that in this
square the $C_2$-action is 180-degree rotation about the center.  

The next thing we will do is cut up the square in two clever ways, as
depicted in the following diagrams:

\begin{picture}(300,120)(20,0)
\put(50,0){\includegraphics[scale=0.6]{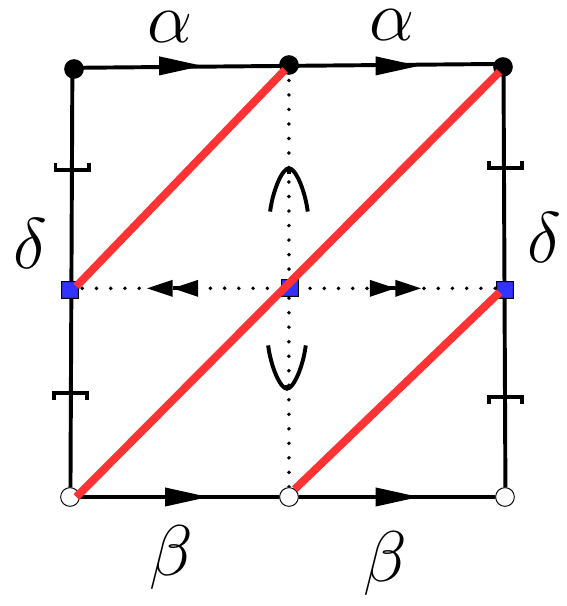}}
\put(215,0){\includegraphics[scale=0.6]{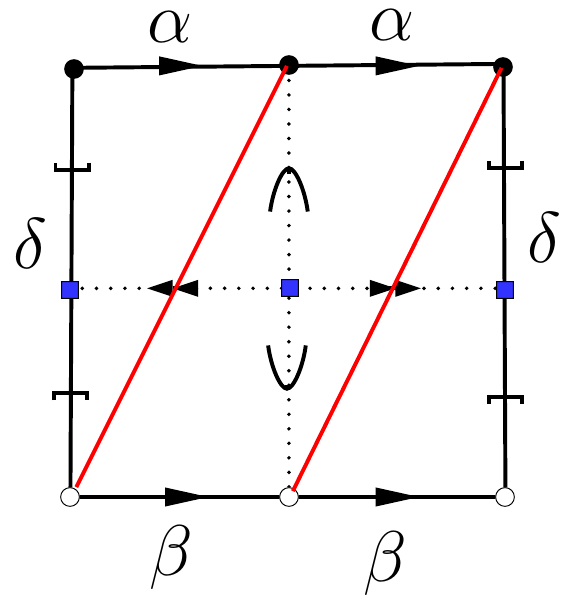}}
\end{picture}

\noindent
In the first picture the red lines depict a copy of $S^{1,1}$, and
removing a tubular neighborhood results in a connected space (this is
easy to check from the picture).  So doing $S^{1,1}$-surgery here
yields $S^2$, by genus considerations.  Since the action on this $S^2$
is clearly free, it is a copy of $S^2_a$ by Lemma~\ref{le:S^2}.  So we
have shown that $S^{2,2}+[DCC]\iso S^2_a+[S^{1,1}-\antitube]$.  

A similar argument applies to the second picture, where the red lines
depict a copy of $S^1_a$.  Doing surgery leaves the space connected,
and so it produces a $2$-sphere with two fixed points---which by
Lemma~\ref{le:S^2} must be $S^{2,2}$.  So we have shown that
$S^{2,2}+[DCC]\iso S^{2,2}+[S^1_a-\antitube]$.

For the second statement of the theorem, start with the easy isomorphism
$S^2_a+[S^{1,1}-\antitube]\iso S^{2,2}+[S^1_a-\antitube]$ (see
the pictures below).  Then 
\begin{align*}
S^2_a+[DCC]+[S^{1,1}-\antitube] &\iso
S^2_a+[S^{1,1}-\antitube]+[DCC]\\
&\iso S^{2,2}+[S^1_a-\antitube]+[DCC] \\
&\iso S^{2,2}+[DCC] + [S^1_a-\antitube] \\
&\iso S^2_a+[S^{1,1}-\antitube]+[S^1_a-\antitube]\\
&\iso \Bigl [S^2_a+[S^{1}_a-\antitube] \Bigr ] + [S^{1,1}-\antitube]
\\
&\iso T_1^{\anti}+[S^{1,1}-\antitube].
\end{align*} 
In the fourth isomorphism we have used the portion of the theorem
already proven.  The content of the first three isomorphisms can be
represented pictorially as follows:

\begin{picture}(300,120)(0,-10)
\put(0,0){\includegraphics[scale=0.45]{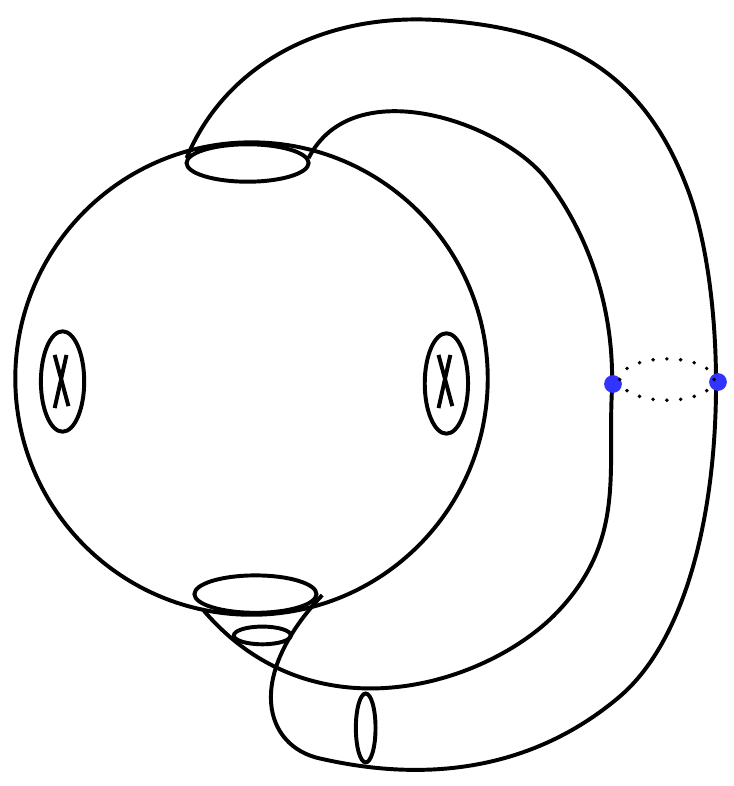}}
\put(105,50){$\iso$}
\put(120,0){\includegraphics[scale=0.45]{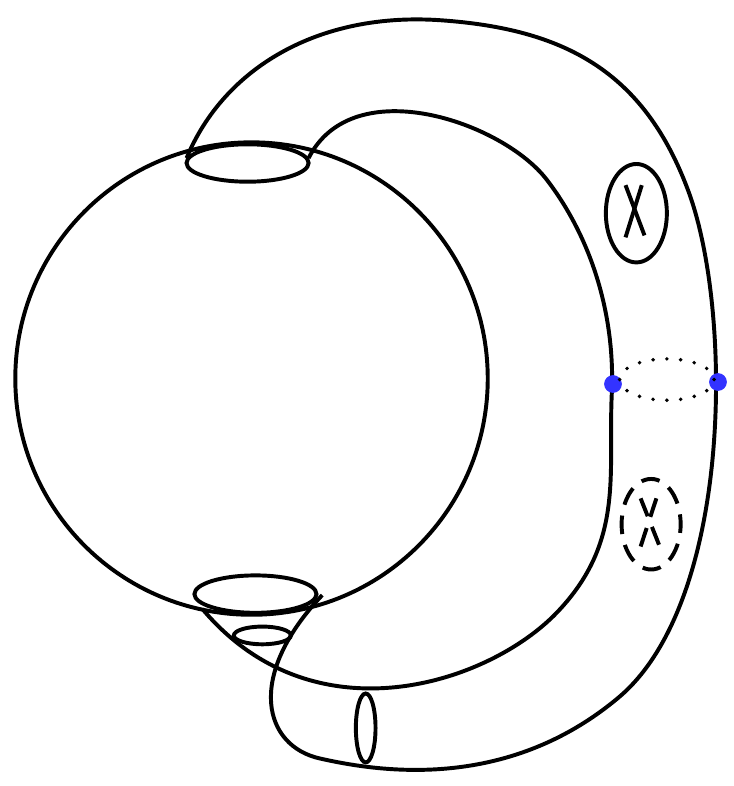}}
\put(223,50){$\iso$}
\put(240,-5){\includegraphics[scale=0.45]{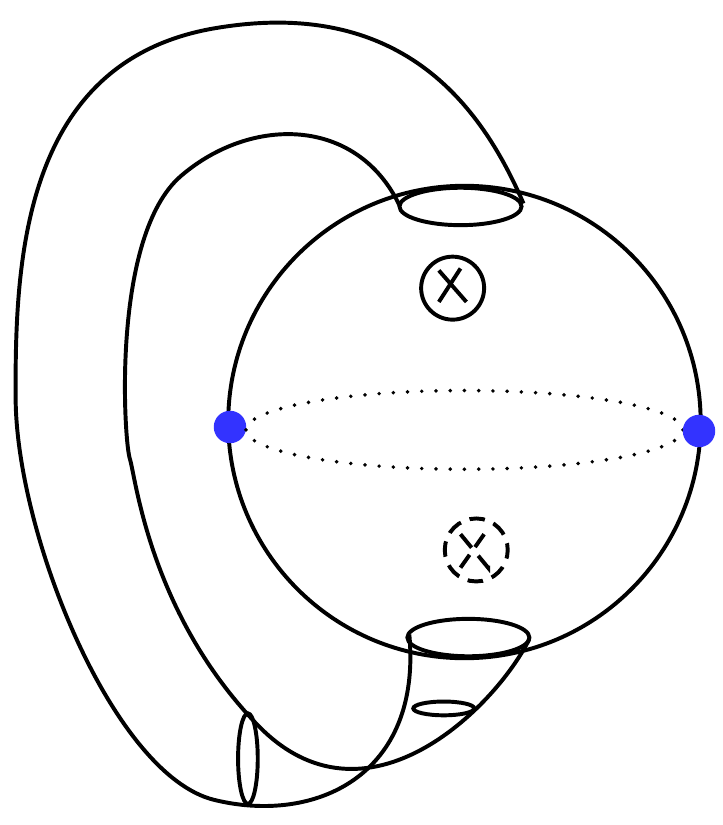}}
\end{picture}

\noindent
At the bottom of these pictures, the tubes are being twisted 180
degrees before being attached to the sphere.  
\end{proof}

\begin{cor}
\label{co:fundiso}
For $r\geq 1$ there are equivariant isomorphisms 
\[ \Doub(N_r,S^{1,1})\iso
S^{2,2}+r[DCC] \iso S^2_a + (r-1)[DCC] + [S^{1,1}-\antitube].
\]
\end{cor}

\begin{proof}
The first isomorphism holds because $\Doub(N^r,S^{1,1})\iso
S^{2,2}\esum N_r \iso S^{2,2}+r[DCC]$.  The second isomorphism is
immediate from Theorem~\ref{th:fundiso}.
\end{proof}

\begin{cor}
\label{co:fundiso2}
For any $g\geq 0$ there is an equivariant isomorphism 
\[ T_g^{\anti} + [S^{1,1}-\antitube]\iso
S^2_a+g[DCC]+[S^{1,1}-\antitube].
\]
\end{cor}

\begin{proof}
The case $g=0$ is trivial, and $g=1$ is part of Theorem~\ref{th:fundiso}.  When
$g\geq 2$ is even we can write $T_g^{\anti}\iso S^2_a\esum
T_{\frac{g}{2}}$ by Proposition~\ref{pr:Tg-free}.  Then we argue that
\begin{align*}
(S^2_a\esum T_{\frac{g}{2}}) + [S^{1,1}-\antitube]
&\iso \bigl (S^2_a + [S^{1,1}-\antitube] \bigr ) \esum T_{\frac{g}{2}} \\
&\iso
\bigl (S^{2,2}+[DCC]\bigr )\esum
T_{\frac{g}{2}} \qquad (\text{Theorem~\ref{th:fundiso}})\\
&\iso S^{2,2}+({1+g})[DCC] \qquad (\text{Proposition~\ref{pr:DCC-esum}})\\
&\iso S^2_a + g[DCC]+[S^{1,1}-\antitube]\quad(\text{Thm.~\ref{th:fundiso}}).
\end{align*}
\noindent
The case where $g$ is odd is similar, starting with $T_g^{\anti}\iso
T_1^{\anti}\esum T_{\frac{g-1}{2}}$.  
\end{proof}
\subsection{Classification in the case \mdfn{$C=0$}}

The following theorem has already been proven in our study of free
actions (see Theorem~\ref{th:free-classify}):

\begin{thm}
\label{th:N-free}
\[ N_r[0,0:(0,0)]=
\begin{cases}
\{ S^2_a \} & \text{when $r=0$},\\[0.1in]
\{ S^2_a+[DCC] \} & \text{when $r=2$},\\[0.1in]
\bigl \{ S^2_a + \tfrac{r}{2}[DCC], \
T_1^{\anti}+(\tfrac{r}{2}-1)[DCC]\bigr \} &
\text{when $r>2$ is even} \\[0.1in]
\emptyset & \text{otherwise}.   
\end{cases}
\]
Note that all the elements in $N_r[0,0:(0,0)]$ have negative $Q$-sign.
\end{thm}

The next theorem builds from the above result and moves into new territory:

\begin{thm}
\label{th:F,0}
Let $F>0$.  
\begin{enumerate}[(a)]
\item
We have $N_0[2,0:(0,0)]=\{S^{2,2}\}$ and
$N_0[F,0:(0,0)]=\emptyset$ for $F\neq 2$.
\item For $r>0$ the set
$N_r[F,0:(0,0)]$ consists of the single element
\[ S^2_a+ \bigl (\tfrac{r-F}{2} \bigr )[DCC] + \tfrac{F}{2}\bigl
[S^{1,1}-\antitube\bigr ]
\]
when $F\leq r$ and $r$ and $F$ are both even; it is empty otherwise.
\end{enumerate}
\end{thm}

\begin{proof}
Part (a)  is trivial, using Lemma~\ref{le:S^2}.  
For (b) we proceed by induction on $F$.  Let
$X$ be in $N_r[F,0:(0,0)]$ where $r>0$ and $F>0$.  We know from
Theorem~\ref{th:invariants} that both
$r$ and $F$ are even.  If
$F=2$ attempt to do $S^{1,1}$-surgery around these two fixed points as in
Remark~\ref{re:S11-fp}.  This fails only if we have a doubled manifold
$\Doub(N_s,S^{1,1})$, where in this case $s=\frac{r}{2}$.
However, by Corollary~\ref{co:fundiso} this is isomorphic to the space
$S^2_a+(\tfrac{r}{2}-1)[DCC]+[S^{1,1}-\antitube]$, which is what we
wanted.  

Continuing with the case $F=2$, we must analyze what happens when
our $S^{1,1}$-surgery succeeds.  It gives that
$X\iso Y+[S^{1,1}-\antitube]$ where the action on $Y$ is
free and $\beta(Y)=r-2$.  If $r=2$ then $Y=S^2_a$ and we are done.  
If $r=4$ then there are cases depending on whether or not $Y$ is
orientable.  If $Y$ is non-orientable then $Y=S^2_a+[DCC]$ by
Theorem~\ref{th:N-free}, and we are
done.
If $Y$ is orientable then either $Y\iso T_1^{\anti}$ or $Y\iso
T_1^{\rot}$.  The latter is not possible since
$T_1^{\rot}+[S^{1,1}-\antitube]$ is orientable, whereas $X$ is not.  
So we conclude $X\iso T_1^{\anti}+[S^{1,1}-\antitube]$, which by
Theorem~\ref{th:fundiso} is isomorphic to $S^2_a+[DCC]+[S^{1,1}-\antitube]$.
So we are again done.  

For $r>4$, by Theorem~\ref{th:N-free}
there are two possibilities for $Y$ when it is
non-orientable: the spaces $S^2_a+\frac{r-2}{2}[DCC]$ and
$T_1^{\anti}+\frac{r-4}{2}[DCC]$.  When $Y$ is orientable there are
also two possibilities: $T_{\frac{r-2}{2}}^{\anti}$ and
$T_{\frac{r-2}{2}}^{\rot}$ (the latter only when $\frac{r-2}{2}$ is
odd).  However, the latter is not truly a possibility as it would lead
to an orientable $X$.  So we have shown that $X$ is isomorphic to one
of the three spaces
\begin{align*}
& S^2_a+\tfrac{r-2}{2}[DCC]+[S^{1,1}-\antitube] \\
& T_1^{\anti}+\tfrac{r-4}{2}[DCC]+[S^{1,1}-\antitube] \\
& T_{\frac{r-2}{2}}^{\anti} + [S^{1,1}-\antitube]
\end{align*}
The first two are isomorphic by Theorem~\ref{th:fundiso}.  The third
is isomorphic to the first two by Corollary~\ref{co:fundiso2}.
This finally completes the case when $F=2$.  

Now assume that $F>2$. Pick two isolated fixed points and again
attempt to do $S^{1,1}$-surgery.  This fails only if $X$ is
$\Doub(N_{r/2},S^{1,1})$, but this is not possible since $X$ has more
than two fixed points.  So the surgery gives that $X\iso
Y+[S^{1,1}-\antitube]$ where $Y$ is an equivariant $2$-manifold with $F-2$
fixed points and no ovals. 
If $Y$ is
orientable then the action is orientation-preserving (due to the
presence of isolated fixed points), which implies that $X$ is
orientable as well; this is a contradiction.  
So $Y$ is non-orientable, hence $Y\in
N_{r-2}[F-2,0:(0,0)]$ and we are done by induction.  
\end{proof}

\subsection{Classification in the case \mdfn{$F>0$ and $C>0$}.}
\mbox{}\par

We first deal with the case where $C_-=0$:

\begin{thm}
\label{th:F>0a}
Let $F>0$ and $C\geq 1$. 
\begin{enumerate}[(a)]
\item $N_r[F,C:(C,0)]$ is nonempty only if $r$ and $F$ are even,
$r\geq 2$,
and
$F+2C\leq r+2$, in which case it consists of at most two elements: at
most one element of positive $Q$-sign, and at most one element of
negative $Q$-sign.
\item Suppose that $r$ and $F$ are even and $F+2C\leq r+2$.  The element of negative
$Q$-sign is
\[ S^2_a+\tfrac{r-F-2C}{2}[DCC]+\tfrac{F}{2}[S^{1,1}-\antitube]+C[S^{1,0}-\antitube]
\]
and occurs
if and only if $F+2C\leq r$.

The element of positive $Q$-sign is 
\[ T_{\frac{r-2C}{2}}^{\spit}[F]+C[S^{1,0}-\antitube]
\]
and occurs if and only if $F+2C\equiv r+2$ mod $4$.
\end{enumerate}
\end{thm}

\begin{proof}
Let $X\in N_r[F,C:(C,0)]$.  We know $r$ and $F$ are even and $F+2C\leq
r+2$ by Theorem~\ref{th:invariants}.  Since $F>0$ and $C>0$ we know
$r\neq 0$, so $r\geq 2$.   
Choose an oval in $X$.  It cannot be separating because it is not the
entire fixed set (since there are isolated fixed points), 
so  $S^{1,0}$-surgery shows $X\iso
Y+[S^{1,0}-\antitube]$ where $Y$ has taxonomy $[F,C-1:(C-1,0)]$ and
$\beta(Y)=r-2$.  

The rest of the proof is by induction on $C$.  Assume first that
$C=1$.  Then $Y$ has taxonomy $[F,0:(0,0)]$.  There are two cases,
depending on whether $Y$ is orientable or not.  If not, we know by
Theorem~\ref{th:F,0} that $F\leq r-2$ and $Y\iso
S^2_a+\frac{r-F-2}{2}[DCC]+\frac{F}{2}[S^{1,1}-\antitube]$, so we are
done.  If, on the contrary, $Y$ is orientable then since it has
isolated fixed points we have $Y\iso T_{\frac{r-2}{2}}^{\spit}[F]$ by
Theorem~\ref{th:T_g-action}.  Note that this can only happen when
$F\equiv 2+2(\frac{r-2}{2})$ modulo $4$.  This leads to the second
possibility for $X$, and  completes the $C=1$ analysis.

Suppose $C\geq 2$.  Then since $F>0$ and $C-1>0$, $Y$ has both ovals and
isolated fixed points---so it cannot be orientable.  Hence $Y\in
N_{r-2}[F,C-1:(C-1,0)]$ and we are done by induction.
\end{proof}

Next we turn to the case where $C_->0$:

\begin{thm}
\label{th:F>0b}
Suppose that $F\geq 0$ and $C_->0$.  
\begin{enumerate}[(a)]
\item  $N_r[F,C:(C_+,C_-)]$ is empty
unless $r\equiv F\equiv C_-$ (mod $2$) and $F+2C\leq r+2$.  It
contains at most two elements: at most one having negative $Q$-sign,
and at most one having positive $Q$-sign.  
\item Suppose $r\equiv F\equiv C_-$ (mod $2$) and $F+2C\leq r+2$.
The element of negative $Q$-sign is
\[ S^2_a + \tfrac{r-F-2C}{2}[DCC] + \tfrac{F+C_-}{2}[S^{1,1}-\antitube]+
C_+[S^{1,0}-\antitube]+C_-[FM]
\]
and occurs if and only if $F+2C\leq r$.  
The element of positive $Q$-sign is
\[ T_{\frac{r-C_--2C_+}{2}}^{\spit}[F+C_-]+C_+[S^{1,0}-\antitube]+C_-[FM] 
\]
and occurs if and only if $F+2C\equiv r+2$ (mod $4$).  
\end{enumerate}
\end{thm}

\begin{proof}
Let $X\in N_r[F,C:(C_+,C_-)]$.  Theorem~\ref{th:invariants} shows
that $r\equiv F\equiv C_-$ (mod $2$) and $F+2C\leq r+2$. 
Doing
$MF$-surgery around each one-sided oval in $X$ shows that $X\iso
Y+C_-[FM]$ where $Y$ has taxonomy $[F+C_-,C_+:(C_+,0)]$ and
$\beta(Y)=r-C_-$.  Note that the $Q$-signs of $X$ and $Y$ must
coincide, by Proposition~\ref{pr:surg-invariance}(ii).

If $C_+>0$ then $Y$ has both isolated fixed points and ovals, so $Y$
is non-orientable.  
Theorem~\ref{th:F>0a} implies that
there are at most two possibilities for $Y$, namely
\[ S^2_a +
\bigl (\tfrac{r-C_--(F+C_-)-2C_+}{2}\bigr)[DCC]+
\tfrac{F+C_-}{2}[S^{1,1}-\antitube]+C_+[S^{1,0}-\antitube] \]
and
\[ T^{\spit}_{\frac{r-C_--2C_+}{2}}[F+C_-] + C_+[S^{1,0}-\antitube].
\]
The first space occurs if and only if $(F+C_-)+2C_+\leq r-C_-$
(equivalently, $F+2C\leq r$), and the second occurs if and only if
$(F+C_-)+2C_+\equiv r-C_-+2$ mod $4$ (or equivalently, 
$F+2C\equiv r+2$ mod $4$).
The desired result is immediate.

Next suppose that $C_+=0$.  If $Y$ is orientable then we know by
Theorem~\ref{th:T_g-action} that
there is an isomorphism $Y\iso T_{\frac{r-C_-}{2}}^{\spit}[F+C_-]$, and that
$F+C_-\equiv 2+r-C_-$ (mod $4$).  If $Y$ is non-orientable then by
Theorem~\ref{th:F,0} 
we know $Y\iso S^2_a +
\frac{(r-C_-)-(F+C_-)}{2}[DCC]+\frac{F+C_-}{2}[S^{1,1}-\antitube]$ and
$F+C_-\leq r-C_-$.  
Again, the desired result is immediate.
\end{proof}

\subsection{Classification in the case \mdfn{$F=C_-=0$}}
This is the only remaining case.  

\begin{thm} 
\label{th:0C0}
Let $C>0$ and $r\geq 2$.  
\begin{enumerate}[(a)]
\item The set $N_r[0,C:(C,0)]$ is empty unless $r$ is even and $C\leq
\frac{r}{2}$.
\item When $r$ is even the set
contains at most four elements: at most three of
negative $Q$-sign, and at most one of positive $Q$-sign.  
\item Assume $r$ is even and $C\leq \frac{r}{2}$.  
The three elements of negative $Q$-sign are:
\begin{tabular}{ll}
$S^{2,1}+(\frac{r}{2}-C+1)[DCC]+(C-1)[S^{1,0}-\at]$ & \text{(always
  occurs)} \\
$S^2_a+(\frac{r}{2}-C)[DCC] + C[S^{1,0}-\at]$ & \text{(occurs if
  and only if $C<\frac{r}{2}$)} \\
$T_1^a + (\frac{r}{2}-C-1)[DCC]+C[S^{1,0}-\at]$ & \text{(occurs iff
   $C<\frac{r}{2}-1$)}.
\end{tabular}

\noindent
These spaces have distinct isomorphism types.
\item Still assume that $r$ is even and $C\leq \frac{r}{2}$.  The
element of positive $Q$-sign is
\[ T_{\frac{r}{2}-C}^{\rot} + C[S^{1,0}-\antitube] \]
and occurs if and only if  $2C\equiv r+2$ (mod $4$).  In
particular, note that $C<\frac{r}{2}$ here.  
\end{enumerate}
\end{thm}

\begin{proof}
If $X\in N_r[0,C:(C,0)]$ then $r\equiv 0$ (mod $2$) and $2C\leq r+2$
by Theorem~\ref{th:invariants}.  Assume that $C=1$.  If the unique
oval is separating then $X\iso \Doub(N_{\frac{r}{2}},1)\iso
S^{2,1}+\frac{r}{2}[DCC]$.  
If the oval
is non-separating then we can do surgery to see that $X\iso
Y+[S^{1,0}-\antitube]$ where the action on $Y$ is free.  
If $Y$ is orientable then it is either $T_{\frac{r-2}{2}}^{\anti}$ or
$T_{\frac{r-2}{2}}^{\rot}$, the latter only possible when
$\frac{r-2}{2}$ is odd.  The first option is not possible since it
would
imply that  $X$ is orientable.  Note that since $C=1$ the criterion that
$\frac{r-2}{2}$ is odd is equivalent to saying $2C\equiv r+2$ (mod
$4$).  

If $Y$ is
non-orientable then there at at most two possibilities for what it
could be, and they are as listed in 
Theorem~\ref{th:N-free}.  A little thought checks that these are
precisely the second two possibilities listed in (c) above.
This completes the proof when $C=1$.

Next proceed by induction.  If $C>1$ then choose an oval in $X$ and
try to do surgery.  The oval cannot be separating because $C>1$, and
so surgery shows $X\iso Y+[S^{1,0}-\antitube]$ where $Y$ has taxonomy
$[0,C-1:(C-1,0)]$.  
If $Y$ is orientable then since $C>1$ the involution on $Y$ is
orientation-reversing.  But then $X$ is orientable, which is a
contradiction.  
So $Y$ is non-orientable, and hence by induction we
know the possibilities for $Y$.  One readily checks that these then
yield the desired options and conditions for $X$.  

There is one thing left to justify, namely the claim that the three
spaces listed in (c) have distinct isomorphism types.  The first space
is separating whereas the latter two are not, so that takes care of
the first space.  If the second two spaces are isomorphic then by
Proposition~\ref{pr:1,0-cancel} we would conclude that
$S^2_a+(\frac{r}{2}-C)[DCC]\iso T_1^a+(\frac{r}{2}-C-1)[DCC]$.
However, the actions here are free and we already proved they were not
isomorphic by computing their characteristic classes; see
Proposition~\ref{pr:classify-base}(b).  (For another proof that the
second two spaces in (c) are not isomorphic, without reducing to the
free case, see Corollary~\ref{co:DD-sep}).  
\end{proof}

At this point we have finished the proofs for
Proposition~\ref{pr:taxonomy} and Theorem~\ref{th:P2}.  Those results
are merely concatenations of Theorems~\ref{th:N-free}, \ref{th:F,0},
and \ref{th:0C0}.

\begin{remark}
The classification proofs in this section are fairly unpleasant.  The
basic idea is very simple, though: given spaces of a fixed taxonomy, latch
onto a certain piece of structure and do surgery to reduce to a
smaller $\beta$-genus.  The trouble is that the surgery leads to
various cases: it might have produced a space with two components, it
might have produced a connected orientable space, and it might have
produced a connected non-orientable space.  Further cases arise
because sometimes in low $\beta$-genus the classification results are
slightly different (e.g. Theorem~\ref{th:N-free}).  It is really just this
constant presence of cases that makes the bookkeeping unpleasant.  
\end{remark}


\section{Counting $C_2$-actions on non-orientable surfaces}

To explain how to count the number of involutions on $N_r$, it will be
convenient to examine the tables in Appendix~\ref{se:tables}.  These
tables are obtained by first listing all taxonomies $[F,C:(C_+,C_-)]$
where $F+2C\leq r+2$ and $F\equiv r\equiv C_-$ (mod $2$).  Note that
$C=C_++C_-$ here, as always.  We then go down our list and find the
taxonomies with $F+2C\equiv r+2$ (mod $4$) and $C>0$: 
each of these yields one action of positive $Q$-sign.  Next, we find
the taxonomies in our list satisfying $F+2C\leq r$ and either $F>0$ or
$C_->0$: each of these yields one action of negative $Q$-sign.
Finally, when $r$ is even 
we look at the taxonomies $[0,C:(C,0)]$.  When $C=0$ there are
two actions of negative $Q$-sign (except when $r=2$, where there is
only one); when $1\leq C\leq \frac{r}{2}-2$
there are three actions of negative $Q$-sign; when $r>2$ and
$C=\frac{r}{2}-1$ there are two actions of negative $Q$-sign; and when
$C=\frac{r}{2}$ there is one action of  negative $Q$-sign. 
This accounts for all the possible actions.

Let $A(r)$ denote the denote the number of tuples $[F,C:(C_+,C_-)]$
of nonnegative integers (having $C=C_++C_-$) satisfying $F+2C\leq r$
and $F\equiv C_-\equiv r$ (mod $2$).  Similarly, let $B(r)$ denote the
number of tuples $[F,C:(C_+,C_-)]$ satisfying $F+2C\leq r+2$, $F\equiv
C_-\equiv r$ (mod $2$), and $F+2C
\equiv r+2$ (mod $4$).  

Let $\Phi(r)$ denote the number of nontrivial involutions on $N_r$ (up to
isomorphism).  Let $\Phi_+(r)$ and $\Phi_-(r)$ be the number of
nontrivial involutions with positive and negative $Q$-sign, respectively.
These quantities are topological, in contrast to the quantities $A(r)$
and $B(r)$ which are purely algebraic.  
The following table shows these numbers in a few cases (the cases
where $r$ is even and odd are separated, for reasons that will become
clear).

\vspace{0.2in}

\begin{tabular}{c||c|c|c|c|c|c|c||c||c|c|c|c|c|c|c|c|}
r & 2 & 4 & 6 & 8 & 10 & 12 & 14 &&  1 & 3 & 5 & 7 & 9 & 11 & 13 & 15\\
\hline
\hline
A(r) & 3 & 7 & 13 & 22 & 34 & 50 & 70 && 0 & 1 & 3 & 7 & 13 & 22 & 34 & 50\\
\hline
B(r) & 5 & 8 & 14 & 20 & 30 & 40 & 55&& 1 & 2 & 5& 8 & 14 & 20 & 30 & 40 \\
\hline
\hline
$\Phi_-(r)$ & 3 & 9 & 17 & 28 & 42 & 60 & 82 && 0 & 1 & 3 & 7 & 13 & 22 & 34 & 50\\
\hline
$\Phi_+(r)$ & 2 & 5 & 10 & 16 & 25 & 35 & 49 && 1 & 2 & 5& 8 & 14 & 20 & 30 & 40 \\
\hline
$\Phi(r)$ & 5 & 14 & 27 & 44 & 67 & 95 & 131 && 1 & 3 & 8 & 15 & 27 &
42 & 64 & 90
\end{tabular}

\vspace{0.2in}

\begin{prop}
\label{pr:Phi}
If $r$ is odd then $\Phi_-(r)=A(r)$ and $\Phi_+(r)=B(r)$.  
If $r$ is even then 
\[ \Phi_-(r)=A(r)+r-2\qquad\text{and}\qquad
\Phi_+(r)=B(r)-1-
\begin{cases} \frac{r+4}{4} & \text{if
  $r\equiv 0$ mod $4$,} \\[0.1in]
\frac{r+6}{4} & \text{if $r\equiv 2$ mod $4$.}
\end{cases}
\]
\end{prop}

\begin{proof}
The idea is that $A(r)$ and $B(r)$ are close to the number of $C_2$-actions
on $N_r$ having negative and positive $Q$-sign, respectively; but when
$r$ is even we
must make slight corrections.  To understand these, it will be
convenient to refer the reader again to the tables in
Appendix~\ref{se:tables}.  We first treat the case when $r$ is even,
despite the fact that this is the more complex case.  

The number $A(r)$ is an undercount for the actions with negative
$Q$-sign, in that it misses the ``extra'' actions of type
$[0,C:(C,0)]$; i.e., it counts these taxonomies as having only one
action, when in fact there are two or three.  To be precise, first
assume $r>2$.  Then when
$C=0$ we have missed one action; when $1\leq C\leq \frac{r}{2}-2$ we
have missed two actions; and when $C=\frac{r}{2}-1$ we have missed one
action.  So in total we have missed $1+2(\frac{r}{2}-1)+1=r-2$
actions.  Thus, $A(r)+r-2$ is the total number of actions on $N_r$ with
negative $Q$-sign.  When $r=2$ we have actually not missed any actions,
and so $A(r)+r-2$ is still the correct count.  

The number $B(r)$ is an overcount for the actions on $N_r$ with
positive $Q$-sign.  The taxonomy $[0,\frac{r+2}{2}:(\frac{r+2}{2},0)]$
is counted in $B(r)$, but does not actually correspond to an action.  
Likewise, the taxonomies $[F,0:(0,0)]$  for $F\leq r+2$ and $F\equiv r+2$
(mod $4$) are counted in $B(r)$ but do not correspond to actions.  
The number of such $F$ is given by $1+\frac{r+2}{4}$ when $4|r+2$ and
$1+\frac{r}{4}$ when $4|r$.  This yields the desired formula for
$\Phi_+(r)$.

Note that the exceptional cases that appeared in the last two
paragraphs all had $C_-=0$; these cannot appear when $r$ is odd, since
by Theorem~\ref{th:invariants} any action satisfies $C_-\equiv r$ mod $2$.
\end{proof}

\begin{lemma}
\label{le:AB}
The sequences $A$ and $B$ satisfy the recursion relations
\begin{align*}
 A(r+2) 
&= \begin{cases}
A(r)+\frac{1}{16}(r+4)(r+8) & \text{if $r\equiv 0$ mod $4$},\\
A(r)+\frac{1}{16}(r+6)^2 & \text{if $r\equiv 2$ mod $4$},\\
A(r)+\frac{1}{16}(r+1)(r+5) & \text{if $r\equiv 3$ mod $4$},\\
A(r)+\frac{1}{16}(r+3)^2 & \text{if $r\equiv 1$ mod $4$}\\
\end{cases}
\end{align*}
and
\begin{align*}
 B(r+4) 
&= \begin{cases}
B(r)+\frac{1}{16}(r+8)(r+12) & \text{if $r\equiv 0$ mod $4$},\\
B(r)+\frac{1}{16}(r+10)^2 & \text{if $r\equiv 2$ mod $4$,}\\
B(r)+\frac{1}{16}(r+5)(r+9) & \text{if $r\equiv 3$ mod $4$},\\
B(r)+\frac{1}{16}(r+7)^2 & \text{if $r\equiv 1$ mod $4$.}\\
\end{cases}
\end{align*}
\end{lemma}

\begin{proof}
We will actually prove that
\begin{align*} A(r+2) &=\begin{cases}
A(r)+2[1+2+\cdots+\tfrac{r+4}{4}] & \text{if $r\equiv 0$ mod $4$},
\\
A(r)+2[1+2+\cdots+\tfrac{r+6}{4}]-\tfrac{r+6}{4} & \text{if $r\equiv
  2$ mod $4$},\\
A(r)+2[1+2+\cdots+\tfrac{r+1}{4}] & \text{if $r\equiv 3$ mod $4$},\\
A(r)+2[1+2+\cdots+\tfrac{r+3}{4}]-\tfrac{r+3}{4} & \text{if $r\equiv 1$ mod $4$},\\
\end{cases} 
\end{align*}
and
\begin{align*}
B(r+4) &=
\begin{cases}
B(r)+2[1+2+\cdots+\tfrac{r+8}{4}] & \text{if $r\equiv 0$ mod $4$},
\\
B(r)+2[1+2+\cdots+\tfrac{r+10}{4}]-\tfrac{r+10}{4} & \text{if $r\equiv
  2$ mod $4$},\\
B(r)+2[1+2+\cdots+\tfrac{r+5}{4}] & \text{if $r\equiv 3$ mod $4$},
\\
B(r)+2[1+2+\cdots+\tfrac{r+7}{4}]-\tfrac{r+7}{4} & \text{if $r\equiv
  1$ mod $4$}.
\end{cases}
\end{align*}
These formulas readily yield the ones in the statement of the lemma.

The difference $A(r+2)-A(r)$ counts solutions to $F+2C=r+2$ where
$F\equiv C_-\equiv r$ (mod $2$).  Assume that $r$ is even.
Then the possibilities for the pair
$(F,C)$ are 
\[ (r+2,0),\ (r,1),\ (r-2,2),\ \ldots, (2,\tfrac{r}{2}),\
(0,\tfrac{r+2}{2}).
\]
For a given value of $C$, let $f(C)$ denote the number of pairs
$(C_+,C_-)$ such that $C=C_++C_-$ and $C_-$ is even.  Then we have
\[ A(r+2)-A(r)=f(0)+f(1)+f(2)+\cdots+f(\tfrac{r}{2})+f(\tfrac{r+2}{2}).
\]
But it is easy to see that $f(C)=\lfloor \frac{C}{2}\rfloor+1=
\lfloor \frac{C+2}{2}\rfloor$ for all $C$.  
So 
\[ A(r+2)-A(r)=1+1+2+2+\cdots+ \lfloor \tfrac{r+4}{4}\rfloor+
\lfloor \tfrac{r+6}{4}\rfloor.
\]
Looking at the cases $r\equiv 0$ (mod $4$) and $r\equiv 2$ (mod $4$)
separately, one readily obtains the desired formulas.  

We are not yet done with all the formulas for the $A$ function, but
let us put that on hold and consider $B$.  The difference
$B(r+4)-B(r)$ counts solutions to $F+2C=r+6$ with $F\equiv C_-\equiv
r$ (mod $2$).  When $r$ is even the possibilities for $(F,C)$ are
\[ (r+6,0),\ (r+4,1),\ (r+2,2),\ \ldots,\ \bigl (0,\tfrac{r+6}{2}\bigr
).
\]
For each value of $C$ in the above list, the number of possibilities
for $(C_+,C_-)$ is $\lfloor \frac{C+2}{2}\rfloor$ just as before.   So
\[ B(r+4)-B(r)=\sum_{C=0}^{\frac{r+6}{2}} \lfloor
\tfrac{C+2}{2}\rfloor
=1+1+2+2+3+3+\cdots+\lfloor \tfrac{r+10}{4}\rfloor
\]
and the desired formulas follow immediately.

The formulas when $r$ is odd follow by very similar arguments.
The main difference here is the function $f(C)$, which now counts
solutions to $C=C_++C_-$ where $C_-$ is odd.  It is easy to check that
this is given by $f(C)=\lfloor \frac{C+1}{2}\rfloor$.  For
$A(r+2)-A(r)$ one sums this function from $C=0$ to $C=\frac{r+1}{2}$,
and the rest of the argument is similar to above.  

Finally, when $r$ is odd one finds by the same arguments that
\[ B(r+4)-B(r)=\sum_{C=0}^{\frac{r+5}{2}} \lfloor
\tfrac{C+1}{2}\rfloor = 1+1+2+2+\cdots+\lfloor\tfrac{r+5}{4}\rfloor +
\lfloor\tfrac{r+7}{4}\rfloor
\]
and the desired formulas follow readily.
\end{proof}

\begin{prop}
\label{pr:A,B}
The sequences $A$ and $B$ are given by the formulas
\[ A(r)=\begin{cases}
\frac{1}{96}(r+3)(r+4)(r+8) & \text{if $r\equiv 0$ mod $4$},\\
\frac{1}{96}(r+2)(r+6)(r+7) & \text{if $r\equiv 2$ mod $4$},\\
\frac{1}{96}(r-1)(r+3)(r+4) & \text{if $r\equiv 1$ mod $4$},\\
\frac{1}{96}r(r+1)(r+5) & \text{if $r\equiv 3$ mod $4$},\\
\end{cases}
\]
and
\[
 B(r)=\begin{cases}
\frac{1}{192}(r+4)(r+8)(r+12) & \text{if $r\equiv 0$ mod $4$},\\
\frac{1}{192}(r+6)(r+8)(r+10) & \text{if $r\equiv 2$ mod $4$},\\
\frac{1}{192}(r+3)(r+5)(r+7) & \text{if $r\equiv 1$ mod $4$},\\
\frac{1}{192}(r+1)(r+5)(r+9) & \text{if $r\equiv 3$ mod $4$}.\\
\end{cases}
\]
\end{prop}

\begin{proof}
This follows from Lemma~\ref{le:AB} by a routine induction.
\end{proof}

\begin{remark}
\label{re:factorizations}
The factorizations in the formulas from Proposition~\ref{pr:A,B} are
somewhat amazing, and of course the induction proof doesn't give a
satisfying explanation.  For $A$ there is also a cyclic pattern to the roots
in the various cases (an $x\ra x+3$ pattern modulo 9).  This pattern is present
in $B$ to some extent, but not as consistently.  In any case, it would
be interesting to see a more conceptual understanding of these
formulas.  
\end{remark}

\begin{cor} 
\label{co:A+B}
For all values of $r$ one has
\[ A(r)+B(r)=\begin{cases}
\frac{1}{64}(r+4)(r+6)(r+8) & \text{if $r\equiv 0$ mod $4$},\\
\frac{1}{64}(r+3)^3 & \text{if $r\equiv 1$ mod $4$},\\
\frac{1}{64}(r+6)^3 & \text{if $r\equiv 2$ mod $4$},\\
\frac{1}{64}(r+1)(r+3)(r+5) & \text{if $r\equiv 3$ mod $4$.}
\end{cases}
\]
\end{cor}

\begin{proof}
Left to the reader.
\end{proof}

\begin{thm}
For $r\geq 1$ the number of nontrivial $C_2$-actions on $N_r$ is given by the
following formulas:

\[\begin{cases}
\frac{(r+3)^3}{64}=\frac{1}{64}\bigl ( r^3+9r^2 +27r+27 \bigr ) &
\text{if $r\equiv 1$\ mod $4$}\\[0.2in]
\frac{(r+1)(r+3)(r+5)}{64}=\frac{1}{64}\bigl ( r^3+9r^2+23r+15\bigr )
& \text{if $r\equiv 3$\ mod $4$}\\[0.2in]
\frac{1}{64}\bigl ( r^3+18r^2+152r-64 \bigr )
&
\text{if $r\equiv 0$\ mod $4$}\\[0.2in]
\frac{1}{64}\bigl ( r^3+18r^2+156r -72 \bigr )
&
\text{if $r\equiv 2$\ mod $4$.}\\[0.2in]
\end{cases}
\]
\end{thm}

\begin{proof}
Proposition~\ref{pr:Phi} gives $\Phi(r)=A(r)+B(r)$ when $r$ is odd,
and 
\[ \Phi(r)=A(r)+B(r)+r-3-\begin{cases} \tfrac{r+4}{4} & \text{if
  $r\equiv 0$ mod $4$},\\
\tfrac{r+6}{4} & \text{if $r\equiv 2$ mod $4$}.
\end{cases}
\]
Now use
Corollary~\ref{co:A+B}.
\end{proof}


\section{The $DD$-invariant and problem P3}
\label{se:DD}

The $DD$-invariant (short for {\it double Dickson\/} invariant) 
is a fairly simple, homological invariant for $C_2$-actions.  It 
is not a very powerful invariant, but it turns out
to detect a subtle difference between $C_2$-spaces that is not easily
seen via other means.  This will play a key role in the
last part of our
story.

To explain the basic idea behind the $DD$-invariant, note that an involution
$\sigma$ on a surface $X$ induces an involution $\sigma^*$ of
$H^1(X;\Z/2)$.  The cup product equips this vector space with a nondegenerate
symmetric bilinear form, and $\sigma^*$ is an isometry.  Write
$\Iso(H^1(X;\Z/2))$ for the group of isometries.  If $\theta$ is
another involution on $X$ and the $C_2$-spaces $(X,\sigma)$ and
$(X,\theta)$ are isomorphic, then $\sigma^*$ is conjugate to
$\theta^*$ inside $\Iso(H^1(X;\Z/2))$.  The $DD$-invariant of $\sigma$
is a $4$-tuple in $\N\times \Z/2\times \N\times \Z/2$ that completely
classifies the conjugacy classes of involutions in
$\Iso(H^1(X;\Z/2))$.  This invariant is purely algebraic, and
was introduced in \cite{D}.  

In this section we briefly introduce the technology needed to
understand the $DD$-invariant, and we perform the key calculation that
will be needed for the last stage of our classification results,
whose proof we also finish off here.
Certainly there is
potential for more work on how the $DD$-invariant interacts with the
$C_2$-equivariant surgery techniques discussed in this paper; 
we do not pursue this here.  

\medskip

\subsection{The algebraic \mdfn{$DD$}-invariant}
Let $(V,b)$ be a finite-dimensional vector space over $\F_2$ equipped
with a nondegenerate, symmetric bilinear form.  We say that $V$ is
\dfn{symplectic} if $b(v,v)=0$ for all $v\in V$, and otherwise we say
$V$ is \dfn{orthogonal}.  Every symplectic space is even-dimensional
and has a symplectic  basis, whereas every orthogonal space has an
orthonormal basis (see \cite[Proposition 2.1]{D}). 

Write $\Iso(V)$ for the group of isometries of $(V,b)$.  The behavior
of this group splits into three cases: $V$ is symplectic (SYMP), $V$
is orthogonal and odd-dimensional (ODDO), and $V$ is orthogonal and
even-dimensional (EVO).  See \cite{D} for a complete discussion.  
Our goal is to explain how to classify involutions in $\Iso(V)$, and
the answer is slightly different in the three cases.    
The first invariant is quite simple:

\begin{defn}
If $\sigma\in \Iso(V)$ is an involution then the \mdfn{$D$-invariant}
$D(\sigma)$ is defined to be the rank of $\sigma+\Id$.
\end{defn}

This is related to the classically-defined Dickson invariant.
It is clear that it is indeed an invariant of the conjugacy class of
$\sigma$, and
it is easy to prove that $0\leq D(\sigma)\leq \tfrac{\dim V}{2}$
always (see \cite[Proposition 3.1]{D}).  

There is a unique vector $\Omega\in V$ having
the property that $b(v,\Omega)=b(v,v)$ for all $v\in V$ \cite[Section 2]{D}.
The space $V$ is symplectic if and only if $\Omega=0$.  Any isometry
of $V$ must preserve $\Omega$, and hence also $\langle
\Omega\rangle^{\perp}$.   Note that when $\dim V$ is odd one has
$b(\Omega,\Omega)\neq 0$ and so
$V=\langle \Omega\rangle \oplus \langle \Omega\rangle^{\perp}$.  

Here is the next invariant:

\begin{defn}
If $\sigma\in \Iso(V)$ is an involution then the map $F_\sigma\colon V\ra \F_2$ given
by $v\mapsto b(v,\sigma v)$ is linear.  Define $\alpha(\sigma)$ by
\[ \alpha(\sigma)=
\begin{cases}
\rank F_\sigma & \text{if $V$ is even-dimensional} \\
\rank F_\sigma|_{\langle \Omega\rangle^{\perp}} & \text{if $V$ is
  odd-dimensional.}
\end{cases}
\]
Note that $\alpha(\sigma)\in \{0,1\}$ always.
\end{defn}

The reason for the cases in the above definition is that when $V$ is
odd-dimensional one always has $b(\Omega,\sigma
\Omega)=b(\Omega,\Omega)=1$, and so $\rank F_\sigma=1$ no matter what
$\sigma$ is.  

As explained in \cite{D}, in the SYMP and ODDO cases the two
invariants $D$ and $\alpha$ completely classify the conjugacy classes
of involutions.  In the EVO case we need to work a bit harder.  

Suppose now that $V$ is EVO, and that $\sigma\in \Iso(V)$ is an
involution.  Define $m\sigma\colon V\ra V$ by
\[ (m\sigma)(v)=v+b(v,v)\Omega.
\]
We call $m\sigma$ the \dfn{mirror} of $\sigma$.  
One easily checks that $m\sigma$ is still an isometry, is also an
involution, and $m(m\sigma)=\sigma$ (see \cite{D} for complete details). 
When $V$ is SYMP or ODDO we simply define $m\sigma=\sigma$.  

\begin{example}
Suppose that $V$ is EVO, with orthonormal basis $e_1,\ldots,e_n$.
Then $\Omega=\sum_i e_i$.  If $A$ is the matrix of $\sigma$ with
respect to this basis, then the matrix for $m\sigma$ is obtained from
$A$ by changing each $0$ entry to a $1$, and each $1$ entry to a $0$.
This is where the term ``mirror'' comes from.
\end{example}

\begin{defn}
Let $\sigma\in \Iso(V)$ be an involution.  Define $DD(\sigma)\in
\N\times \Z/2\times \N\times \Z/2$ to be the $4$-tuple
$DD(\sigma)=[D(\sigma),\alpha(\sigma),D(m\sigma),\alpha(m\sigma)]$.
We also write $\tD(\sigma)=D(m\sigma)$,
$\talpha(\sigma)=\alpha(m\sigma)$, and 
\[ DD(\sigma)=
[D(\sigma),\alpha(\sigma),\tD(\sigma),\talpha(\sigma)].
\]  
\end{defn}

Notice that when $V$ is SYMP or ODDO, the last two coordinates of the
$DD$-invariant are simply repetitions of the first two coordinates.
We have set things up this way only because it allows us to treat the three
cases for $V$ simultaneously.  For example,
some of the main results of \cite{D} can be stated as follows:

\begin{thm}
Two involutions $\sigma,\theta\in \Iso(V)$ are conjugate if and only
if $DD(\sigma)=DD(\theta)$.  
\end{thm}

\begin{remark}
For future reference, note that for the identity involution one has
\[ DD(\Id)=\begin{cases}
[0,0,0,0] & \text{if $V$ is SYMP,}\\
[0,1,0,1] & \text{if $V$ is ODDO,}\\
[0,1,1,0] & \text{if $V$ is EVO.}
\end{cases}
\]
\end{remark}

We will need the following simple calculation:

\begin{prop}
\label{pr:DD-sum-alg}
Let $V$ and $W$ be finite-dimensional vector spaces over $\F_2$, and
suppose that $V$ is equipped with  a
nondegenerate symplectic bilinear form and $W$ with a nondegenerate
orthogonal bilinear form.  Let $\sigma\in \Iso(V)$
be an involution.  Assume $W$ is even-dimensional with orthonormal
basis $x_1, y_1, x_2, y_2,\ldots,x_r,y_r$, and let $\theta\in \Iso(W)$
be the  involution satisfying $\theta(x_i)=y_i$ for all $i$.
Then $\sigma\oplus\theta$ is an involution on $V\oplus W$ and
\[ DD(\sigma\oplus\theta)=
\begin{cases}
[D(\sigma)+r,\alpha(\sigma),D(\sigma)+r-1,1] & \text{if $r$ is odd,}
\\
[D(\sigma)+r,\alpha(\sigma),D(\sigma)+r,1] & \text{if $r$ is even.}
\end{cases}
\]
\end{prop}

\begin{proof}
The matrix for $\theta$ with respect to the given basis on $W$ is the
block diagonal matrix with $r$ copies of $\begin{bsmallmatrix} 0 & 1
\\ 1 & 0\end{bsmallmatrix}$ along the diagonal.  It is immediate that
$\alpha(\theta)=0$ and $D(\theta)=r$.  It only takes a moment more to
compute the mirror of this matrix and calculate that
$\tD(\theta)=r$ if $r$ is even, 
$\tD(\theta)=r-1$ if $r$ is odd, and $\talpha(\theta)=1$ always.  

One can now either compute $DD(\sigma\oplus\theta)$ by brute force, or
else consult \cite[Theorem 5.4]{D} to say that
\[
DD(\sigma\oplus\theta)=[D(\sigma)+D(\theta),\max\{\alpha(\sigma),\alpha(\theta)\},
\tD(\sigma)+\tD(\theta),\max\{\talpha(\sigma),\talpha(\theta)\}].
\]
Now simply recall that $\tD(\sigma)=D(\sigma)$ and
$\talpha(\sigma)=\alpha(\sigma)$, since $V$ is symplectic.
\end{proof}

\subsection{Topological \mdfn{$DD$}-invariants}

Now let us return to the topological setting, where $(X,\sigma)$ is a
surface with $C_2$-action.  

\begin{defn}
The
\mdfn{$DD$-invariant} $DD(X)$ is defined to equal the (algebraic)
$DD$-invariant of the map $\sigma^*\in \Iso(H^1(X;\Z/2))$.  We
likewise write $D(X)$, $\alpha(X)$, $\tD(X)$, and $\talpha(X)$ 
 for the components of
the $DD$-invariant.  
\end{defn}

\begin{example}
\label{ex:Dinv-tori}
Of course $DD(S^2_a)=DD(S^{2,0})=DD(S^{2,1})=DD(S^{2,2})=[0,0,0,0]$.  For the
unique nontrivial action on $\RP^2$, the map 
$\sigma_*$ is the identity: so $DD(\RP^2)=[0,1,1,0]$.

There are six possible actions on $T_1$: 
\[ T_1^{\triv}, \ T_1^{\anti},
 \ 
T_1^{\rot},\  T_1^{\spit}[4],\  T_1^{\refl}[2],\ \text{and}\ 
S^2_a+[S^{1,0}-\antitube].
\]  
We leave the reader to
check that $\sigma_*$ equals the identity for the first five of these,
so that these all have $DD=[0,0,0,0]$.  For the sixth case, this space is shown
in the following picture:

\begin{picture}(300,110)
\put(100,0){\includegraphics[scale=0.45]{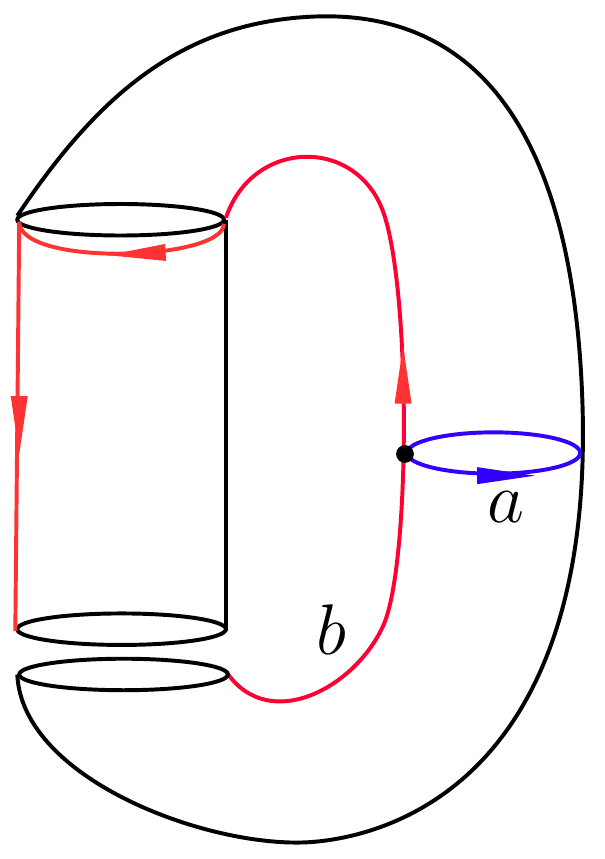}}
\end{picture}

\noindent
Note that the $S^2_a$ is drawn as a cylinder with antipodal action,
the $S^{1,0}-\antitube$ is attached to the top and
bottom, and the bottom end of the tube must be rotated 180 degrees before
attachment.  If $a$ and $b$
are the basis for $H^1$ shown in the picture then $\sigma_*(a)=a$ and
$\sigma_*(b)=a+b$.  The space $H^1$ is symplectic, and 
one readily computes that $DD\bigl (S^2_a+[S^{1,0}-\antitube]\bigr )=[1,1,1,1]$.

Similar calculations can be done for the six $C_2$-actions on the
Klein bottle (see the first table in Appendix~\ref{se:tables}).  We leave it as an
exercise for the reader to check the following computations:

\vspace{0.1in}

\begin{tabular}{c|c|c|c|c|c}
 $\scriptstyle{K^{\triv}}$ & $\scriptstyle{S^2_a+[DCC]}$ & 
$\scriptstyle{S^{2,1}+[DCC]}$ & $\scriptstyle{S^2_a+[S^{1,1}-AT]}$
& $\scriptstyle{S^{2,2}+[S^{1,0}-AT]}$ & $\scriptstyle{S^{2,2}+2[FM]}$
\\
\hline
 $[0,1,1,0]$ & $[1,0,0,1]$ & $[1,0,0,1]$ & $[1,0,0,1]$ &
$[0,1,1,0]$ 
& $[0,1,1,0]$
\end{tabular}

\vspace{0.1in}
\noindent
Observe that these examples show that the $DD$-invariant is relatively
weak in terms of its ability to differentiate equivariant spaces.  (Of
course $O(2,\Z/2)\iso \Z/2$ and so it is not surprising that there are
only two possible $DD$-invariants here).  
\end{example}

Here is a simple result that will be useful:

\begin{prop}
\label{pr:DD-DCC}
Let $X$ be an orientable $2$-manifold with involution $\sigma$.  Then
\[ DD(X+r[DCC])=
\begin{cases}
[D(X)+r,\alpha(X),D(X)+r-1,1 ] & \text{if $r$ is odd,} \\
[D(X)+r,\alpha(X),D(X)+r,1]  & \text{if $r$ is even.}
\end{cases}
\]
\end{prop}

\begin{proof}
Let $Y=X+r[DCC]$.  Then $H^1(Y;\Z/2)\iso H^1(X;\Z/2)\oplus
(\Z/2)^{2r}$.  The involution on $H^1(Y;\Z/2)$ is readily checked to
be the direct sum of the involution on $H^1(X;\Z/2)$ and the
involution on $(\Z/2)^{2r}$ that pairwise swaps basis elements.  Now
apply Proposition~\ref{pr:DD-sum-alg}.
\end{proof}

The following result gives two calculations that will be important in
our applications:

\begin{prop} 
\label{pr:DD-computation}
For $C\geq 1$ one has
$DD(S^2_a+C[S^{1,0}-\antitube])=[1,1,1,1]$ and $DD(T^a_1+C[S^{1,0}-\antitube])=[2,1,2,1]$.  
\end{prop}

\begin{proof}
In both cases, the underlying space is orientable and so the
bilinear form on $H^1$ is symplectic.  So we just need to compute the
$D$- and $\alpha$-invariants.

The space $X=S^2_a+3[S^{1,0}-\antitube]$ has the following model:

\begin{picture}(300,120)
\put(50,0){\includegraphics[scale=0.5]{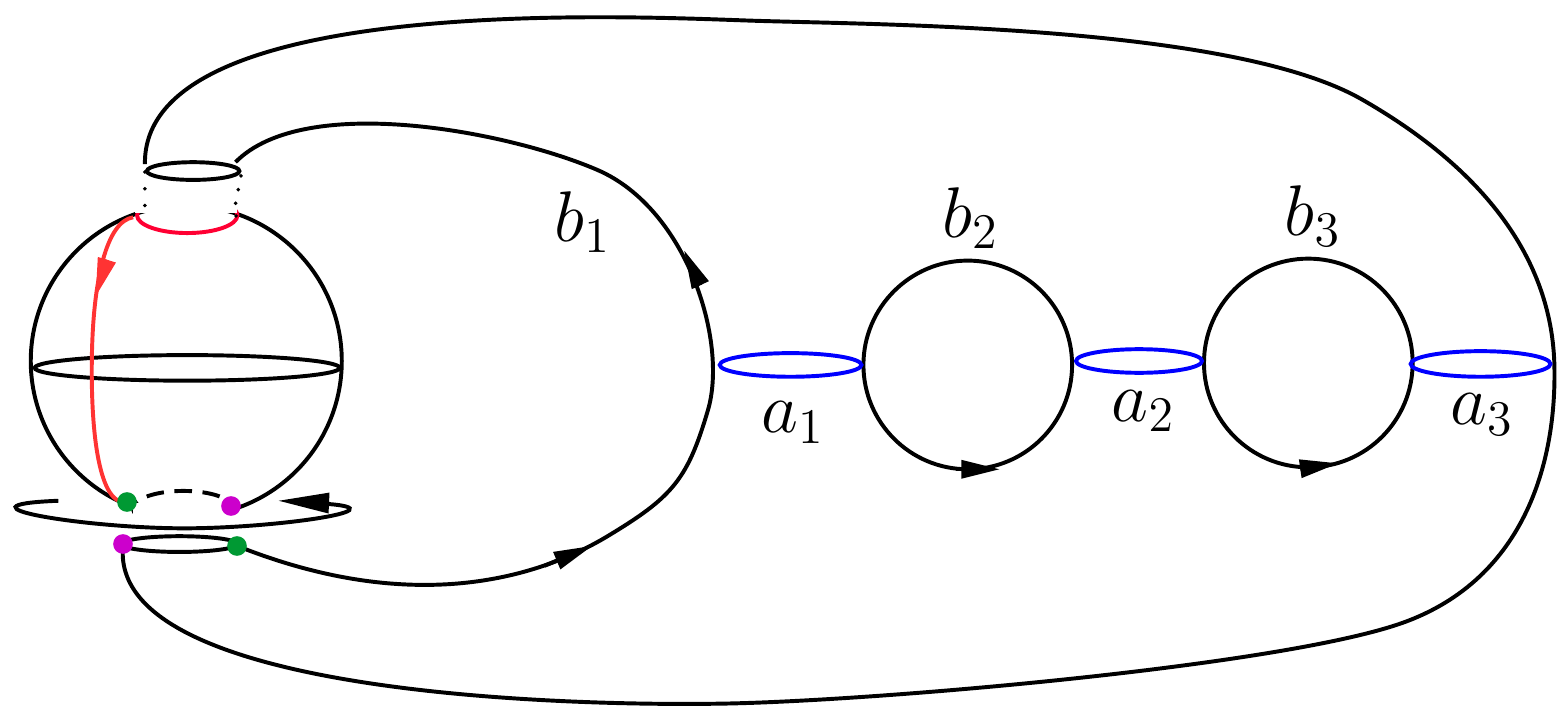}}
\end{picture}

\noindent
Here the sphere has the antipodal action, whereas on the toral
component the action is reflection across the equatorial plane.  The
toral component is glued to the sphere in the evident manner at the
top, but at the bottom it is glued after a $180$-degree rotation.  

The indicated cycles $a_i,b_i$ for $1\leq i\leq 3$ give a 
basis for the symplectic space $H_1(X;\Z/2)$ (but not a symplectic
basis).  
Note that the portion of $b_1$ on the sphere runs along the front part of
the top hole, then down the front part of the sphere.  
The cycles $a_i$ are all fixed by $\sigma_*$, and $b_i$ is
fixed for $i\geq 2$.  Finally, one checks that
$\sigma_*(b_1)=b_1+a_1+a_2+a_3$.  To see this, let $c$ denote the
portion of $b_1$ that runs along the sphere (shown in red).  Then
$\sigma_*(b_1)+b_1=c+\sigma_*(c)$.    But $\sigma_*(c)$ is the path
that runs along the back side of the bottom hole and then up the back
side of the sphere, so it is easy to see that $c+\sigma_*(c)$ is
homologous to the $1$-cycle that runs along the top hole.  This $1$-cycle 
is then clearly homologous to $a_1+a_2+a_3$.

We can now compute that $b_1\cdot \sigma_*(b_1)=b_1\cdot a_1=1$, so
$\alpha(X)=1$.  Moreover, 
$\im(\sigma_*+\Id)$ is spanned by $a_1+a_2+a_3$, thus $D(X)=1$.
The generalization to arbitrary $C\geq 1$ is clear.

The space $Y=T_1^a+3[S^{1,0}-\antitube]$ has the following model:

\begin{picture}(300,120)
\put(50,0){\includegraphics[scale=0.5]{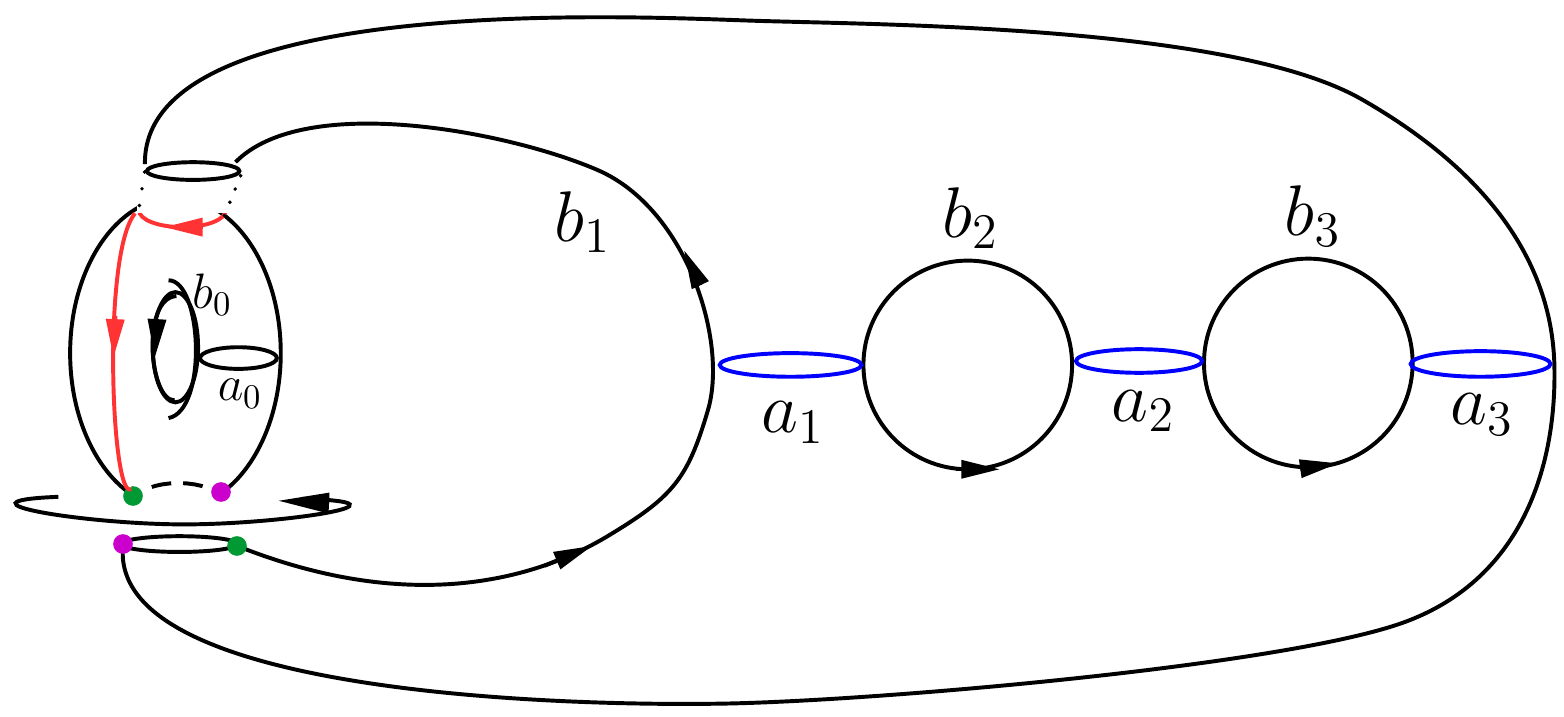}}
\end{picture}

\noindent
The conventions in this diagram are similar to those in the one drawn
for
 $X$.  Here the cycles
$a_i,b_i$ for $0\leq i\leq 3$ are a basis for $H_1(X;\Z/2)$, with the part
of $b_1$ running along the torus shown in red.  The cycles
$a_1,a_2,a_3,b_0,b_2,b_3$ are all fixed by $\sigma_*$, so it remains
to compute $\sigma_*(a_0)$ and $\sigma_*(b_1)$.  

The cycle $\sigma(a_0)$ is the parallel loop to $a_0$ on the opposite
branch of the torus, so clearly $\sigma(a_0)+a_0$ is cohomologous to
the loop $c$ running along the top hole of the torus---which is in turn
cohomologous to $a_1+a_2+a_3$.  So $\sigma_*(a_0)=a_0+a_1+a_2+a_3$.  
Likewise, $\sigma(b_1)+b_1$ is the path shown in red followed by its
antipode (the path that runs along the back side of the bottom hole,
then up the back side of the torus).  One readily checks that
$\sigma(b_1)+b_1$ is then cohomologous to $b_0+c$ (where $c$ is as
above), and therefore $\sigma_*(b_1)=b_1+b_0+a_1+a_2+a_3$.

From these formulas one readily computes that $(\sigma_* b_1)\cdot
b_1=1$, so $\alpha(Y)=1$.  Also $\im(\sigma_*+\Id)=\langle
a_1+a_2+a_3,b_0\rangle$, so $D(Y)=2$.  Again, the generalization to
arbitrary $C\geq 1$ is clear.
\end{proof}

\begin{cor} 
\label{co:DD-sep}
Let $C\geq 1$ and $r>2C$. \par\noindent Then
$DD\Bigl (S^2_a+(\tfrac{r}{2}-C)[DCC] + C[S^{1,0}-\at]\Bigr )=
[\tfrac{r}{2}-C+1,1,u,1]$
where
\[ u=  
\begin{cases}
\tfrac{r}{2}-C & \text{if $\frac{r}{2}-C$ is odd}, \\
\tfrac{r}{2}-C+1  & \text{if $\frac{r}{2}-C$ is even.} 
\end{cases}
\]
If $r>2C+2$ then
$
DD\Bigl (T_1^a + (\tfrac{r}{2}-C-1)[DCC]+C[S^{1,0}-\at]\Bigr )=
[\tfrac{r}{2}-C+1,1,u,1]
$
where 
\[u= \begin{cases}
\tfrac{r}{2}-C+1 & \text{if $\frac{r}{2}-C$ is odd}, \\
\tfrac{r}{2}-C  & \text{if $\frac{r}{2}-C$ is even.} 
\end{cases}
\]
Consequently, the $C_2$-space $S^2_a+(\tfrac{r}{2}-C)[DCC] +
C[S^{1,0}-\at]$ is not isomorphic to
$T_1^a + (\frac{r}{2}-C-1)[DCC]+C[S^{1,0}-\at]$.
\end{cor}

\begin{proof}
This follows immediately from Proposition~\ref{pr:DD-computation} and
Proposition~\ref{pr:DD-DCC}.
\end{proof}

\begin{remark}
Note that in Corollary~\ref{co:DD-sep} one really needs the $\tD$ invariant
to distinguish the spaces; the $D$, $\alpha$, and $\talpha$ invariants
fail to do the job.  
\end{remark}

\subsection{A complete set of invariants}

We now complete our classification of $C_2$-actions on $2$-manifolds,
by solving problem (P3) from the introduction.

\begin{thm}
Suppose given $C_2$-actions on closed $2$-manifolds $X$ and $Y$, where
$X$ and $Y$ have the same non-equivariant topological type.  Then
\begin{enumerate}[(a)]
\item If the invariants $F$, $C_+$,  $C_-$, $Q$, $\epsilon$, and $DD$
are the same for $X$ and $Y$, then $X\iso Y$ as $C_2$-spaces.
\item If $F+C_- >0$ or if $Q$ is positive, 
and 
the invariants $F$, $C_+$, $C_-$, $Q$ are the
same for $X$ and $Y$, then $X\iso Y$ as $C_2$-spaces.  
\item If $X$ is orientable and
the invariants $F$, $C_+$, $C_-$, $Q$ are the
same for $X$ and $Y$, then $X\iso Y$ as $C_2$-spaces.  
\end{enumerate}
\end{thm}

\begin{proof}
We have already proven (b) and (c): (b) synthesizes results from
Theorems~\ref{th:F,0}, \ref{th:F>0a}, and \ref{th:F>0b}, whereas (c)
is just Corollary~\ref{co:P3-orient}.   To prove (a) we need to analyze
the case where $X$ is non-orientable, $F=C_-=0$, and $Q$ is negative.
By Theorem~\ref{th:invariants} this occurs only when $X\iso N_r$ where $r$ is
even, and Theorem~\ref{th:0C0} says that
in this case there is exactly one element of positive
$Q$-sign and at most three elements of negative $Q$-sign.  The latter
three elements are
\begin{align*}
& S^{2,1}+(\tfrac{r}{2}-C+1)[DCC]+(C-1)[S^{1,0}-\at] \\
& S^2_a+(\tfrac{r}{2}-C)[DCC] + C[S^{1,0}-\at] \\
& T_1^a + (\tfrac{r}{2}-C-1)[DCC]+C[S^{1,0}-\at]
\end{align*}
The first of these is separating, whereas the latter two are not: so
$\epsilon$ distinguishes the first from the latter two.  By
Corollary~\ref{co:DD-sep} the $DD$-invariant distinguishes the
second from the third.
\end{proof}


\section{Connections with the mapping class group}
\label{se:MCG}

The problem we have pursued in this paper, of describing isomorphism
classes of $C_2$-actions on a $2$-manifold $X$, has some relation to
the
problem of finding elements of order at most $2$ in the mapping class group
$\cM(X)$.  Certainly  a $C_2$-action on $X$ yields such an element in the
mapping class group.  They are different problems,
though, and in this section we will give some examples showing the
differences.  

\medskip

If $G$ is a group, let $G_{\btwo}$ denote the elements of $G$ having
order at most $2$.  Then $G$ acts on $G_{\btwo}$ by conjugation; write
$G_{\btwo}/\ssim$ for the set of orbits.
Let $\Homeo(X)$ denote the group of self-homeomorphisms of $X$, and
set $\Invol(X)=\Homeo(X)_{\btwo}/\ssim$.  Then $\Invol(X)$ coincides with the set
of isomorphism classes of $C_2$-actions on $X$.  

The projection $\Homeo(X)\ra \cM(X)$ induces a map $\Gamma_X\colon \Invol(X)\ra
\cM(X)_{\btwo}/\ssim$.  We give a few examples investigating this map.

\begin{example}[The $2$-sphere]
Here one has $\cM(S^2)=\Z/2$; the mapping class of an automorphism
simply measures  whether it is orientation-preserving or reversing.  So
$\cM(S^2)_{\btwo}/\ssim$ has two elements.  We know, however, that
$\Invol(S^2)$ has four elements: two of them ($S^{2,0}$ and $S^{2,2}$)
map to the identity mapping class, and the other two ($S^2_a$ and
$S^{2,1}$) map to the non-identity element.
\end{example}

\begin{example}[The torus]
\label{ex:MCG-torus}
Here we have $\cM(T_1)\iso \GL_2(\Z)$, with the isomorphism given by
the action of mapping classes on $H_1(T_1)$.  That is, the natural map
$\cM(T_1)\ra \Aut(H_1(T_1;\Z))$ is an isomorphism.

Some algebraic work (see Section~\ref{se:GL2} below) 
reveals that $\GL_2(\Z)_{\btwo}/\ssim$ has four
elements, represented by $I$, $-I$, $\begin{bsmallmatrix} 1 & 0 \\ 0 &
-1\end{bsmallmatrix}$, and $\begin{bsmallmatrix} 0 & 1\\ 1&
0\end{bsmallmatrix}$.  
Each of $I$ and $-I$ are the only elements in their orbit.  All other
elements of $\GL_2(\Z)_{\btwo}$ have the form 
$\begin{bsmallmatrix} a & b\\ c & -a\end{bsmallmatrix}$ where
$bc=1-a^2$, and such a matrix
 is in the orbit of $\begin{bsmallmatrix} 1&0 \\ 0 &
-1\end{bsmallmatrix}$ if and only if $a$ is odd and both $b$ and $c$
are even; otherwise it is in the orbit of $\begin{bsmallmatrix} 0 &
1\\ 1 & 0\end{bsmallmatrix}$.  These facts take a little work to
check, but they are not hard: see Section~\ref{se:GL2} below.

The following table lists the six $C_2$-actions on $T_1$ together with
their image in $\GL_2(\Z)_{\btwo}/\ssim$.  Here we just computed the
action of each involution on $H_1(T_1)$ and used the algebraic rules
from the preceeding paragraph.

\vspace{0.2in}

\bgroup
\def\arraystretch{1.5}
\begin{tabular}{c|c|c|c|c|c}
$T_1^{\text{triv}}$ & $T_1^{\anti}$ & $T_1^{\rot}$ & $T_1^{\spit}[4]$ &
$T_1^{\refl}[2]$ & $S^2_a+[S^{1,0}-\antitube]$ \\
\hline
$I$ & $\begin{bsmallmatrix} 1 & 0 \\ 0 & -1\end{bsmallmatrix}$ & $I$ &
$-I$ &
$\begin{bsmallmatrix} 1 & 0 \\ 0 & -1\end{bsmallmatrix}$
& $\begin{bsmallmatrix} 0 & 1 \\ 1 & 0\end{bsmallmatrix}$
\end{tabular}
\egroup

\end{example}

\vspace{0.1in}

\begin{example}[The Klein bottle]
Here we have $\cM(K)\iso \Z/2\times \Z/2$ by \cite[Lemma 5]{L}.  
In fact, a careful look at \cite{L}  reveals that the evident map 
\[ \Psi\colon \cM(K)\ra \Aut(H_1(K;\Q))_{\btwo}
\times O(H_1(K;\Z/2))
\]
is an isomorphism, where $O$ denotes the orthogonal group with respect
to the intersection form (this orthogonal group  is
readily checked to be $\Z/2$).  Recall $H_1(K;\Q)\iso \Q$,
$\Aut(\Q)\iso \Q^*$, and $(\Q^*)_{[2]}=\{1,-1\}$.  
Because $\cM(K)$ is abelian, and all
elements have order at most $2$, 
 $\cM(K)_{\btwo}/\ssim \,=\cM(K)$.  Since both $\Aut(H_1(K;\Q))_{\btwo}$
and $O(H_1(K;\Z/2))$ are isomorphic to $\Z/2$, we will in both cases
use $1$ and $-1$ to represent the identity and the unique non-trivial
element, respectively.
    
The following table lists the six $C_2$-actions on the Klein bottle
together with their images in $\cM(K)$ (or more precisely, their
images under $\Psi$).

\vspace{0.1in}

\bgroup
\def\arraystretch{1.5}
\begin{tabular}{c|c|c|c|c|c}
$\scriptstyle{S^2_a+[DCC]}$ & $\scriptstyle{S^{2,1}+[DCC]}$ & $\scriptstyle{S^2_a+[S^{1,1}-AT]}$ & $\scriptstyle{S^{2,2}+[S^{1,0}-AT]}$ &
 $\scriptstyle{S^{2,2}+2[FM]}$ & $\scriptstyle{K^{\triv}}$ \\
\hline
$(-1,-1)$ & $(1,-1)$
& $(-1,-1)$ & $(-1,1)$
  & $(1,1)$
& $(1,1)$
\end{tabular}
\egroup
\end{example}

\noindent
This case is a little harder than the torus, so we give some hints to
these calculations.  The picture below shows a Klein bottle
represented as a sphere
with two crosscaps, where the boundary of the disk in our picture
should be squashed to a point.
The loops $\alpha_1$ and $\alpha_2$ 
are an orthogonal basis for $H_1(K;\Z/2)$, and either one by itself
constitutes
a basis for $H_1(K;\Q)$.  Note that $2(\alpha_1+\alpha_2)=0$ in
$H_1(K;\Z)$, and so $\alpha_1=-\alpha_2$ in $H_1(K;\Q)$.  

\begin{picture}(300,100)
\put(100,5){\includegraphics[scale=0.5]{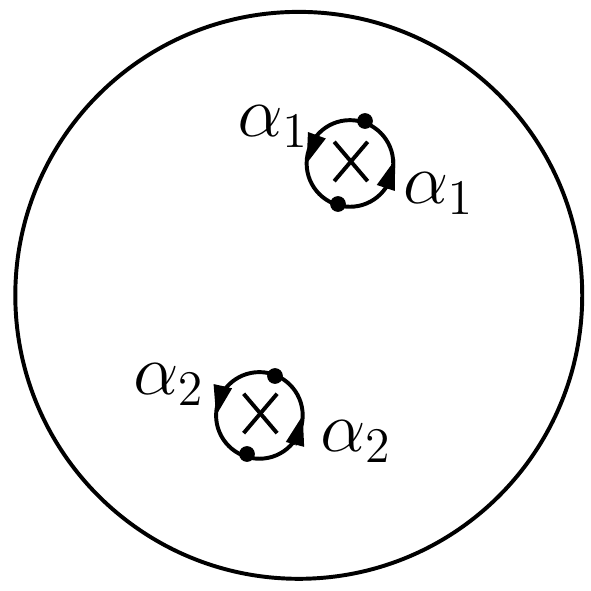}}
\end{picture}

\noindent
This model allows one to readily do the calculations for
$S^2_a+[DCC]$, $S^{2,1}+[DCC]$, and $S^{2,2}+2[FM]$.

For the remaining two cases, the spheres with attached antitubes, it
is perhaps easiest to use other models.  The picture below shows two
Klein bottles (where the arrows denote gluing, not the $C_2$-action):

\begin{picture}(300,100)
\put(0,10){\includegraphics[scale=0.6]{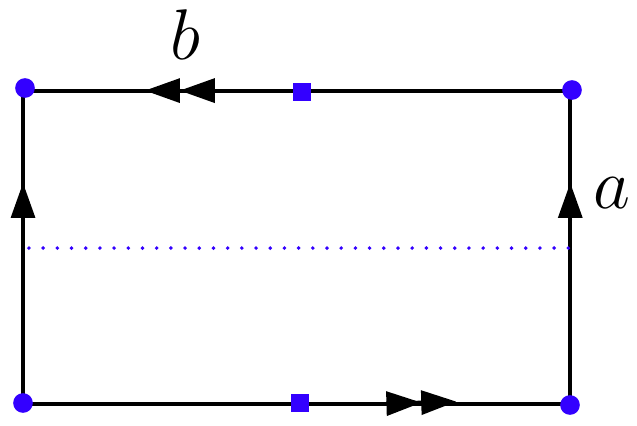}}
\put(200,10){\includegraphics[scale=0.6]{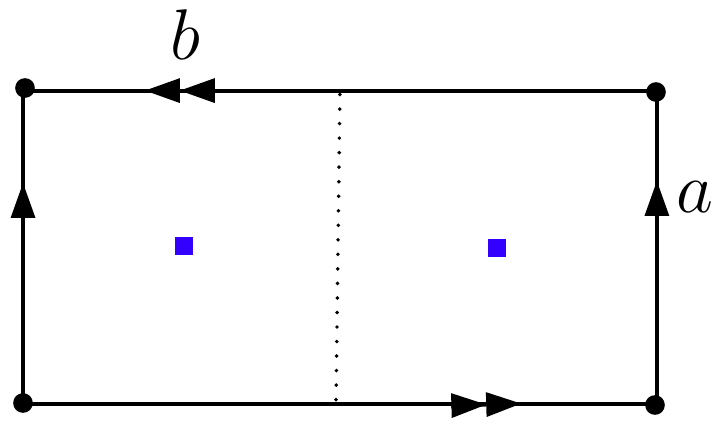}}
\end{picture}

\noindent
In the first, the involution is reflection across the dotted line.
The fixed set is a circle together with two points, and so this is a
model for $S^{2,2}+[S^{1,0}-\antitube]$ (this follows from our
classification; see the table for $N_2$ in Appendix B). 
In the second picture, the involution rotates each of
the two squares 180-degrees about their center.  Here the fixed set
consists of exactly two points, so this is a model for
$S^2_a+[S^{1,1}-\antitube]$.  
In both cases loop $a$ is a generator for
$H_1(K;\Q)$, and the pair $\{a,a+b\}$ is an orthonormal basis for
$H_1(K;\Z/2)$.   Using these models, it is easy to compute the
remaining entries in the above table.

\vspace{0.1in}

In the general case,
it seems possible that $\Gamma_X\colon
\Invol(X)\ra \cM(X)_{\btwo}/\ssim$ is always
surjective.  We do not know how to prove this, though.  
If $X$ is orientable then $\Homeo^+(X)_{\btwo}\ra \cM^+(X)_{\btwo}$ is
surjective by \cite[Theorem 7.1]{FM}, but this does not seem to
immediately imply that $\Homeo(X)_{\btwo}\ra \cM(X)_{\btwo}$ is
surjective.   Moreover, the above examples show that even when $\Gamma_X$
is surjective the cardinalities of the fibers can
differ.  Given an element of $\cM(X)_{\btwo}$ it is
unclear how to predict the size of the fiber over this element.

\subsection{Conjugacy classes of order two elements of
  \mdfn{$\GL_2(\Z)$}}
\label{se:GL2}
We close this section by giving the algebraic analysis needed for
Example~\ref{ex:MCG-torus}.  
One readily checks that $\GL_2(\Z)_{\btwo}$ has exactly two elements of
determinant one, namely $I$ and $-I$.  Furthermore, the elements of
determinant $-1$ all have the form $\begin{bsmallmatrix} a & b \\ c &
-a \end{bsmallmatrix}$ where $a^2+bc=1$.  

The matrices $I$ and $-I$
are central in $\GL_2(\Z)$ and so are the only elements in their
conjugacy classes.  It remains to determine the conjugacy classes for
the above matrices of determinant $-1$.  
We start with four identities:
\[ \begin{bmatrix} -1 & 0 \\ 0 & 1\end{bmatrix}^{-1}\cdot
\begin{bmatrix} x & y \\ z & -x \end{bmatrix} \cdot 
\begin{bmatrix} -1 & 0 \\ 0 & 1\end{bmatrix} =
\begin{bmatrix} x & -y \\ -z & -x \end{bmatrix} \tag{1}
\]

\[ \begin{bmatrix} 1 & 0 \\ \lambda & 1\end{bmatrix}^{-1}\cdot
\begin{bmatrix} x & y \\ z & -x \end{bmatrix} \cdot 
\begin{bmatrix} 1 & 0 \\ \lambda & 1\end{bmatrix} =
\begin{bmatrix} x+\lambda y  & y \\ -2\lambda x-\lambda^2 y +z  &
-x-\lambda y \end{bmatrix}\tag{2}
\]
\[ 
\begin{bmatrix} 1 & \lambda \\ 0 & 1\end{bmatrix}^{-1}\cdot
\begin{bmatrix} x & y \\ z & -x \end{bmatrix} \cdot 
\begin{bmatrix} 1 & \lambda \\ 0 & 1\end{bmatrix} =
\begin{bmatrix} x-\lambda z  & 2\lambda x-\lambda^2z+y \\ 
z   &
-x+\lambda z \end{bmatrix},\tag{3}
\]
\[
\begin{bmatrix} 0 & 1 \\ 1 & 0\end{bmatrix}^{-1}\cdot
\begin{bmatrix} x & y \\ z & -x \end{bmatrix} \cdot 
\begin{bmatrix} 0 & 1 \\ 1 & 0 \end{bmatrix} =
\begin{bmatrix} -x & z \\ y & x
\end{bmatrix}.\tag{4}
\]
Since $\GL_2(\Z)$ is generated by elementary matrices, relations
(1)--(3)  generate all conjugacy relations.  So (4) is actually a
consequence of these, but we list it anyway because of its usefulness.

Observe
that if $A\in \GL_2(\Z)_{\btwo}$ is such that $a_{12}$ and $a_{21}$
are even, then this same property holds for the conjugates of $A$
obtained from (1)--(4), and therefore for {\it all\/} conjugates of
$A$.  
In particular, the matrices
\[ S=\begin{bmatrix} 1 & 0 \\ 0 & -1\end{bmatrix}
\qquad\text{and}\qquad
T=\begin{bmatrix} 0 & 1 \\1 & 0\end{bmatrix}
\]
are in different conjugacy classes.

\begin{prop}
All matrices of determinant $-1$ in $\GL_2(\Z)_{\btwo}$ are conjugate
to either $S$ or $T$.
\end{prop}

\begin{proof}
For matrices $A\in \GL_2(\Z)_{\btwo}$ we proceed by induction on $|a_{11}|$.
When $a_{11}=0$ there are only two such matrices, namely $T$ and
$-T$.  These are conjugate by relation (1).

When $a_{11}=1$ we get the matrices $\begin{bsmallmatrix}
1 & b\\ 0 & -1\end{bsmallmatrix}$ and $\begin{bsmallmatrix}
1 & 0\\ b & -1\end{bsmallmatrix}$, for any $b\in \Z$.  But relations
(2) and (3) show that
\[ 
\begin{bmatrix} 1 & 0\\ b & -1 \end{bmatrix} 
\sim \begin{bmatrix} 1 & 0 \\ b-2\lambda & -1\end{bmatrix},
\qquad
\begin{bmatrix} 1 & b\\ 0 & -1 \end{bmatrix} 
\sim \begin{bmatrix} 1 & b+2\lambda \\ 0 & -1\end{bmatrix}
\]
for any $\lambda$ in $\Z$.  So only the parity of $b$ matters, and
this leaves us with the three elements
\[ \begin{bmatrix} 1 & 0 \\ 0 & -1\end{bmatrix}, \quad  
 \begin{bmatrix} 1 & 1 \\ 0 & -1\end{bmatrix}, \quad  
 \begin{bmatrix} 1 & 0 \\ 1 & -1\end{bmatrix}.
\]
The second two are readily checked to be conjugate to $T$; for
example, using (2) we get

\[ \begin{bmatrix} 1 & 0 \\ 1 & 1\end{bmatrix}^{-1} \cdot 
 \begin{bmatrix} 0 & 1 \\ 1 & 0\end{bmatrix} \cdot 
 \begin{bmatrix} 1 & 0 \\ 1 & 1\end{bmatrix} =
 \begin{bmatrix} 1 & 1 \\ 0 & -1\end{bmatrix}.
\]

A similar analysis applies to $a_{11}=-1$: all such matrices are conjugate
to one of 
\[ \begin{bmatrix} -1 & 0\\ 0 & 1\end{bmatrix}, \quad
\begin{bmatrix} -1 & 0\\ 1 & 1\end{bmatrix}, \quad
\begin{bmatrix} -1 & 1\\ 0 & 1\end{bmatrix}.
\]
The first is conjugate to $S$ using relation (4), and the second two
are conjugate to $T$ using relation (4) and what we have already
shown.    

Now assume that $\begin{bsmallmatrix} a & b \\ c & -a\end{bsmallmatrix}\in
\GL_2(\Z)$ has determinant $-1$ and $|a|>1$.  Then
$bc=1-a^2=(1-a)(1+a)$, and in particular $b\neq 0$ and $c\neq 0$.  
If $|b|\geq |a|+1$
and $|c|\geq |a|+1$ then $|bc|\geq a^2+2|a|+1 > a^2-1$, and this is a
contradiction.  So either $0<|b|\leq |a|$ or $0<|c|\leq |a|$.  In the
former case we use relation (2) with $\lambda\in \{1,-1\}$ to reduce
the magnitude of the upper left entry of the matrix.  By induction
this new matrix is conjugate to $S$ or $T$, so we are done.  The case
$|c|\leq |a|$ is similar, this time using relation (3).
\end{proof}

\begin{prop}
A matrix $A\in \GL_2(\Z)_{\btwo}$ of determinant $-1$ is conjugate to
$S$ if and only if $a_{12}$ and $a_{21}$ are both even.  
\end{prop}

\begin{proof}
The ``only if'' part has already been proven, since the property of
$a_{12}$ and $a_{21}$ being even is preserved by relations (1)--(4).  
Now assume that $\begin{bsmallmatrix} a & b \\ c &
-a\end{bsmallmatrix}$ has determinant $-1$ and is such that $b$ and
$c$ are even.  Then $a$ is odd, so write $a=2n+1$, $b=2b'$, and
$c=2c'$.  The relation $a^2+bc=1$ becomes $n(n+1)+b'c'=0$.  

Write each of $b'$ and $c'$ as a product of positive prime factors
and possibly a $-1$.  Because $b'c'=-n(n+1)$, we can pull out enough
factors from $b'$ and $c'$ so that their product is $n+1$.  We can
represent this by the following picture, where each box contains some
subset of the terms in the factorizations:

\begin{picture}(300,70)
\put(0,0){\includegraphics[scale=0.5]{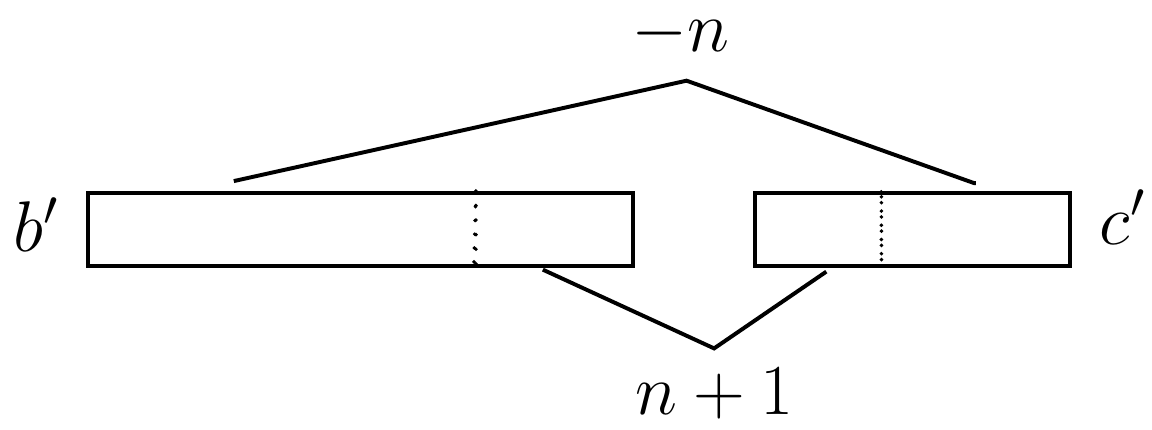}}
\end{picture}

\noindent
Define $x$, $y$, $-z$, and $w$ to be the products of the terms in each
of the boxes, according to the picture

\begin{picture}(300,70)
\put(0,5){\includegraphics[scale=0.5]{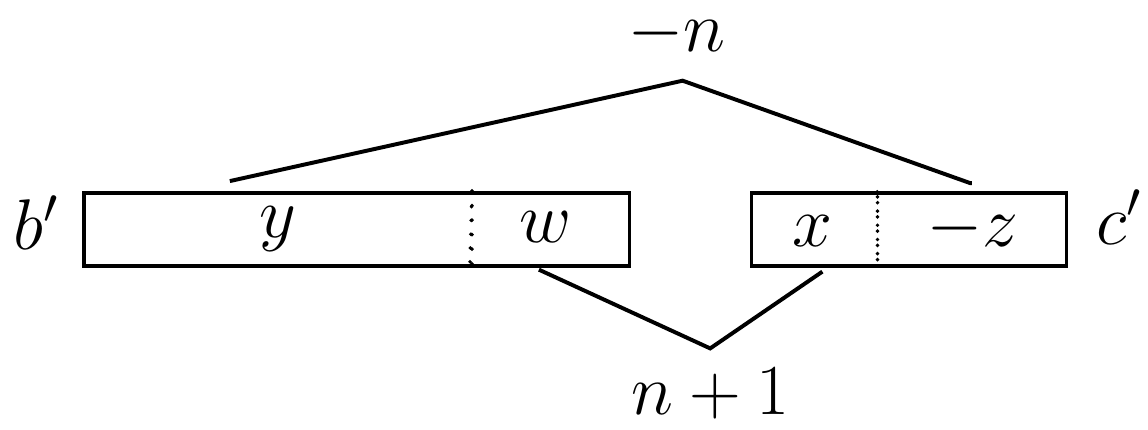}}
\end{picture}

\noindent
So we have $wx=n+1$, $yz=n$, $wy=b'$, and $xz=-c'$.  From this one
readily checks that $xw-yz=1$ and
\[ 
\begin{bmatrix}
x & y \\ z & w \end{bmatrix}^{-1} \cdot
\begin{bmatrix} 1 & 0 \\ 0 & -1 \end{bmatrix} \cdot
\begin{bmatrix}
x & y \\ z & w \end{bmatrix} =
\begin{bmatrix}
2n+1 & 2b' \\ 2c' & -(2n+1)
\end{bmatrix}=
\begin{bmatrix}a & b \\ c & -a \end{bmatrix}.
\]
which completes the proof.
\end{proof}

\appendix

\section{Surgery invariance}

A critical theorem in surgery theory guarantees that the end result of a
surgery doesn't depend on {\it where\/} it was performed.  As an
example in our
equivariant context, this says that
  the result of 
sewing an $S^{1,0}$-antitube into a $C_2$-equivariant space 
does not depend (up to isomorphism) on where the antitube was sewn in.  
We use this in several key places too numerous to mention---but see,
as one example, the appearance in the proof of
Proposition~\ref{pr:surg-subtract}.  In this appendix we give a proof
of this fundamental result.  

\medskip

We begin with a ``moving lemma'':

\begin{prop}
\label{pr:ML}
Let $X$ be a connected, closed $2$-manifold with an involution.
Let $a$ and $b$ be points in a common path component of $X-X^{C_2}$, 
and assume that $a\neq b$ and $a\neq \sigma b$.
Then there is a simple path $\alpha$ from $a$ to $b$ in $X-X^{C_2}$ that does
not intersect its conjugate  $\sigma\alpha$: that is, $\alpha(s)\neq
\sigma\alpha(t)$ for all values of $s$ and $t$.  
\end{prop}

\begin{proof}
First note that if $a$ and $\sigma a$ are in different path components
of $X-X^{C_2}$ then there is nothing to prove, as any simple path from $a$ to
$b$ will do the job.  So assume that $a$ and $\sigma a$ are in the
same path component of $X-X^{C_2}$.  

Pick any nice (smooth, or even PL) simple path $\alpha$ from $a$ to $b$ in
$X-X^{C_2}$.  
Assume that $\alpha$ and $\sigma\alpha$ cross each other. For each
point $z$ where they cross, the points $z$ and $\sigma z$ are distinct
and so have disjoint Euclidean neighborhoods.  By altering $\alpha$
in a small neighborhood of $z$, we can assume that
$z$ is an isolated point of the intersection---even more, we can
assume the intersection of $\alpha$ and $\sigma\alpha$ at $z$ is
transverse.  
The same therefore
holds at $\sigma z$.  Proceeding in this way, we can assume that
every point of intersection is transverse---and so there are only
finitely many such points.  Call them $p_1,\ldots,p_k$, and
assume them to be ordered so that they are encountered in
succession as one moves along $\alpha$.  Then it must be that
$\sigma(p_1)=p_k$, $\sigma(p_2)=p_{k-1}$, and so forth.  Therefore $k$
cannot be odd, otherwise the middle $p_i$ would be a fixed point and
this contradicts the way $\alpha$ was chosen.  

Let $x=p_{\frac{k}{2}}$, so that $\sigma x=p_{\frac{k}{2}+1}$.
Locally around $x$ the paths look something like the following (the
two diagrams show neighborhoods of $x$ and $\sigma x$):

\begin{picture}(300,110)
\put(30,10){\includegraphics[scale=0.4]{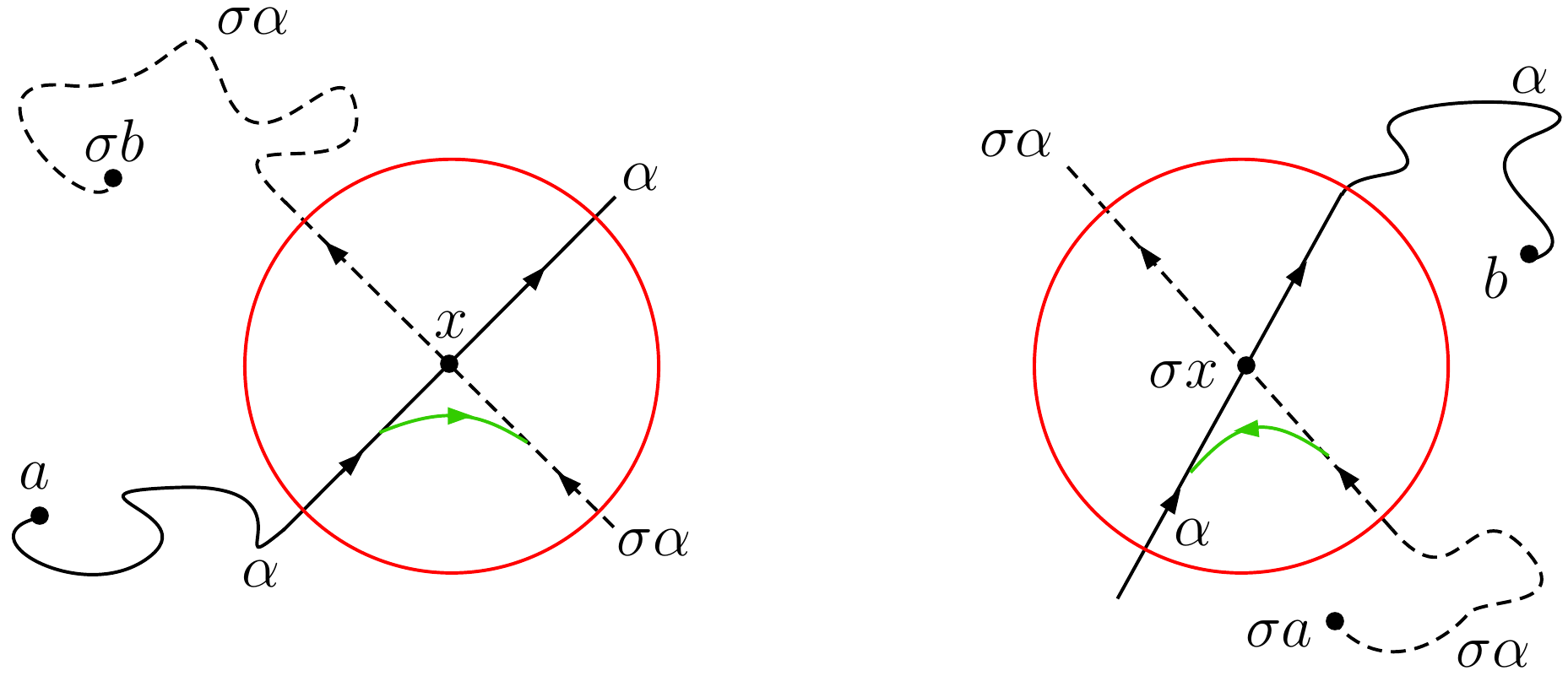}}
\end{picture}

\noindent
Note that the relative orientations of $\alpha$ and $\sigma \alpha$
might coincide with what is shown above, or else they might be
reversed (either around one of $x$ and $\sigma x$, or around both);
these orientations will not matter for the argument.  

Choose a small ``sidepath'' that passes from $\alpha$ to $\sigma \alpha$,
avoiding $x$, as shown in green in the first picture.  The conjugate of this
small path is also shown near $\sigma x$.  Define a new path $\alpha'$
as follows:

\begin{enumerate}[(1)]
\item Start at $a$ and follow $\alpha$ until just before getting to
$x$.
\item Follow the chosen sidepath to avoid $x$, ending up on $\sigma
\alpha$.
\item Follow $\sigma\alpha$ backwards until reaching $\sigma x$.
\item Now follow $\alpha$ again until reaching $b$.  
\end{enumerate}
Here is the modified picture showing $\alpha'$:

\begin{picture}(300,100)
\put(30,0){\includegraphics[scale=0.4]{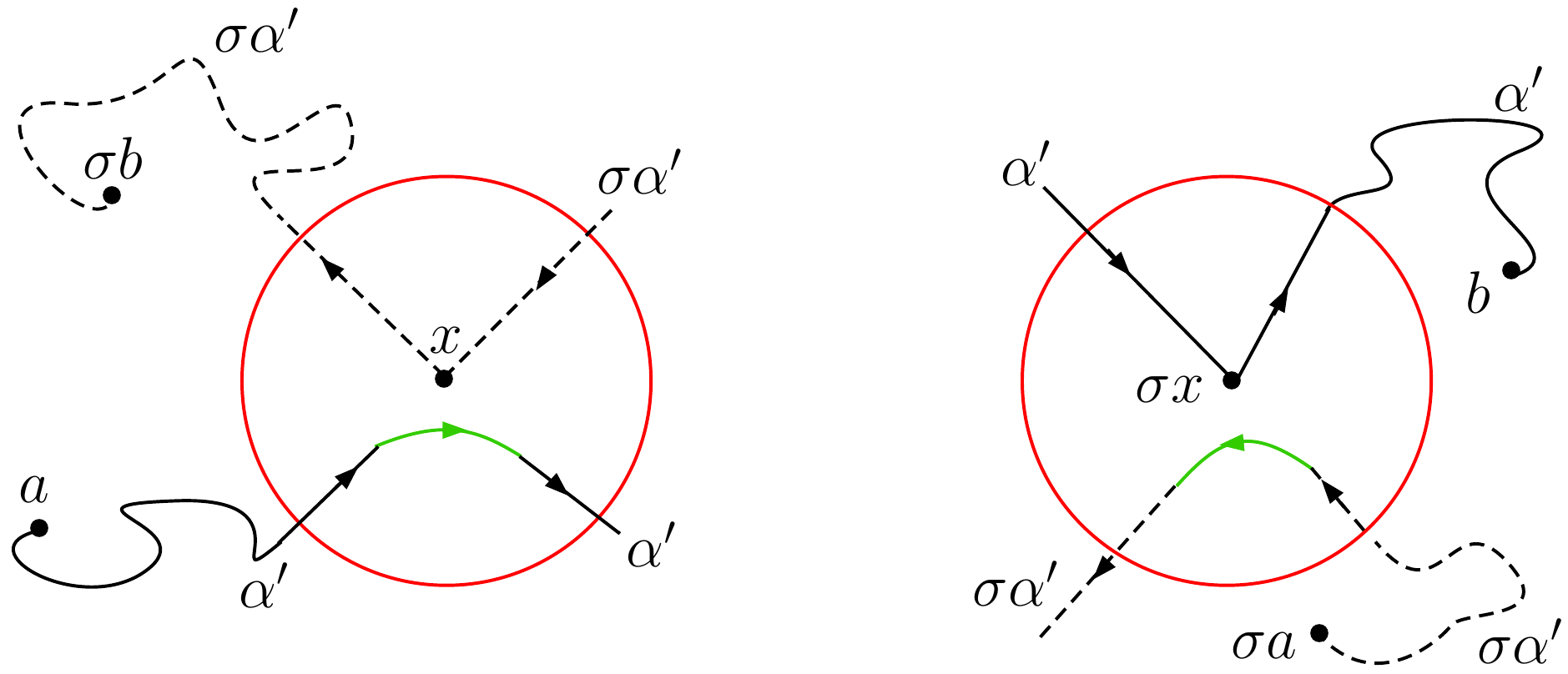}}
\end{picture}

Because $\alpha$ and $\sigma \alpha$ did not intersect each other between $x$ and
$\sigma x$, the path $\alpha'$ intersects $\sigma \alpha'$ in exactly two
fewer points than $\alpha$ and $\sigma\alpha$ did.

Proceeding inductively, one gradually modifies $\alpha$, reducing the
number of intersection points with its conjugate by two each time,  until there
are no intersection points at all.
\end{proof}

\begin{cor}
\label{co:nicepath}
Let $X$ be a connected, closed $2$-manifold with nontrivial involution, and let
$a,b\in X^{C_2}$ be distinct.  Then there is a simple path $\alpha$ from $a$
to $b$ that does not intersect its conjugate other than in the two endpoints.
\end{cor}

\begin{proof}
Start by choosing Euclidean neighborhoods $U_a$ and $U_b$, of $a$ and
$b$ respectively, that are disjoint, stable under conjugation, and
such that $\sigma|_{\bd U_a}$ and $\sigma|_{\bd U_b}$ are not constant.  Let
$x$ be any point on the boundary of $U_a$ that is not fixed.  
If $b$ is not an
isolated fixed point then by Proposition~\ref{pr:conn} the oval passing through $b$
touches both path components; so one can find a point $y$
on the boundary of $U_b$, also not fixed, that is
in the same path component as $x$ in $X-X^{C_2}$.  
If $b$ is isolated then by Corollary~\ref{co:isfp=>ns} the space $X-X^{C_2}$
is connected, so one can again choose a $y\in \bd U_b$ satisfying the
same properties.

By
Proposition~\ref{pr:ML} there exists a simple path $\alpha$ from $x$
to $y$ in $X-X^{C_2}$ having the property that $\alpha$ does not
intersect its conjugate.
Choose any simple path $u$ from $a$ to $x$ in $U_a$ that avoids its
conjugate, 
and any simple path $v$ from $y$ to $b$ in $U_b$ that avoids its
conjugate.  Then the concatenation $v\alpha u$ has the desired property.
\end{proof}

\begin{cor}[Surgery invariance]
\label{co:surg-inv}
Let $X$ be a path-connected, closed $2$-manifold with involution.  
\begin{enumerate}[(a)]
\item Let
$Y_1$ be obtained from $X$ by removing disjoint, conjugate
disks embedded in $X-X^{C_2}$ 
and sewing in an equivariant antitube of a certain type.
Let $Y_2$ be similarly obtained from $X$, sewing in the same type of
antitube, but using a different pair of conjugate embedded disks.  Then
there is an equivariant isomorphism $Y_1\iso Y_2$.

\item
Likewise, if $M$ is a connected $2$-manifold then the equivariant
isomorphism type of $X\#_2 M$ is independent of the conjugate disks
used in the construction.

\item
Finally, if $X$ has isolated fixed points then the space $X+[FM]$ is
independent (up to equivariant isomorphism) of the choice of fixed
point and surrounding disk used in the construction.

\end{enumerate}
\end{cor}

\begin{proof}
First consider (a).
Let the disks used to make $Y_i$ be called $D_i$ and $\sigma D_i$.
Clearly we can shrink $D_i$ and $\sigma D_i$ as much as we want
without changing the equivariant homeomorphism type of $Y_i$.
In particular, we can assume that $D_1$ does not meet $D_2\cup \sigma
D_2$.  

Let $a_1$ be the center of $D_1$, and $a_2$ be the center of $D_2$.
Without loss of generality we can assume that $a_1$ and $a_2$ are in
the same path component of $X-X^{C_2}$ (if not, reverse the names of
$D_2$ and $\sigma D_2$).  By Proposition~\ref{pr:ML}, there exists a
simple path $\alpha$ from $a_1$ to $a_2$ that does not intersect its
conjugate.  

Let $U$ denote a neighborhood of the path $\alpha$ that is small
enough to be Euclidean, to contain $D_1$ and $D_2$ (recall
that we can shrink these disks as much as we want), and to have the
property that $U\cap \sigma U=\emptyset$.  Then there exists
a self-homeomorphism of $U$ that fixes its boundary and sends $D_1$ to
$D_2$.  Extend this to a map $h\colon X\ra X$ by letting $h$ be the
identity outside of $U\cup \sigma U$, the chosen automorphism inside
of $U$, and the conjugate of this chosen automorphism inside of
$\sigma U$.  So $h$ is an equivariant map, and clearly 
$h$ extends to give an equivariant isomorphism between $Y_1$ and
$Y_2$.

The proof for (b) is identical to that of (a).

For (c), the independence of the choice of disk is clear enough.
Suppose that $a$ and $b$ are two isolated fixed points in $X$.  As in
Remark~\ref{re:S11-fp} there is an $S^{1,1}$-antitube inside of $X$
that passes through the points $a$ and $b$.  The Dehn twist on this
antitube can be modeled by a $C_2$-equivariant map that interchanges
$a$ and $b$: in terms of the picture below, this is the map that in
the fiber $y=t$ rotates the $xz$-plane 
counterclockwise about the
center of the tube, through $2\pi t$ radians.

\begin{picture}(300,100)
\put(60,0){\includegraphics[scale=0.5]{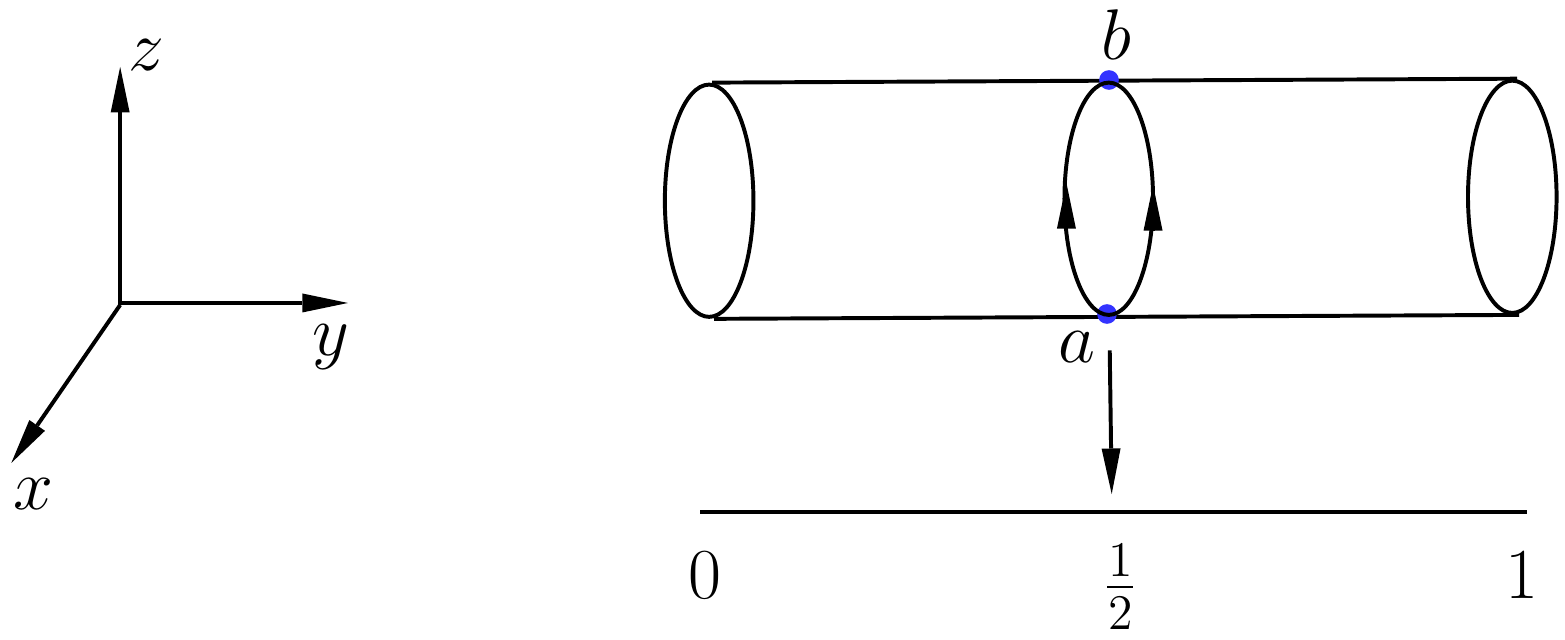}}
\end{picture}

\vspace{0.1in}

\noindent
This equivariant model for the Dehn twist extends to give an equivariant isomorphism
$X+_a[FM]\iso X+_b[FM]$ between $FM$-surgeries around $a$ and $b$.
\end{proof}


\newpage

\section{Tables of involutions on non-orientable surfaces}
\label{se:tables}

Recall from Section~\ref{se:non-orient} that 
$N_r[F,C\colon (C_+,C_-),Q]$ denotes the set of isomorphism classes of
involutions on $N_r$ having taxonomy equal to $[F,C:(C_+,C_-),Q]$.  
The tables below give information about these sets, listing both the
number of elements as well as explicit
names for all the elements.  Of course listing the
{\it number\/} is then redundant information, but we include this
because it makes certain patterns more evident.  

The tables
are organized as follows.  The unsigned taxonomies
$[F,C:(C_+,C_-)]$ index the rows, and the $Q$-signs index the
columns.  For the rows we list all tuples where $F+2C\leq r+2$ and
$F\equiv C_-\equiv r$ (mod $4$), which are the restrictions imposed by
Theorem~\ref{th:invariants}.
The first set of columns gives the number of elements, and the
second sets gives the names of the elements.  For example, the first
table depicts the $5$ nontrivial involutions on $N^2$, three with
negative $Q$-sign and two with positive $Q$-sign.  Tables are included
for $N_k$ where $2\leq k\leq 7$.

Note that ``antitube'' is abbreviated to ``AT'' in the tables, for
space considerations.

\vspace{0.2in}

\begin{tabular}{c|c|c|c|c}
$N_2$ & - & + &  - & + \\
\hline
4,0:(0,0) && & \\
2,0:(0,0) &1 & & $\scriptstyle{S^2_a+[S^{1,1}-AT]}$\\
2,1:(1,0) & & 1 & & $\scriptstyle{S^{2,2}+[S^{1,0}-AT]}$ \\
0,0:(0,0) &1 & & $\scriptstyle{S^2_a+[DCC]}$ \\
0,1:(1,0) & 1 & & $\scriptstyle{S^{2,1}+[DCC]}$ \\
0,2:(2,0) && & \\
0,2:(0,2) & & 1 & & $\scriptstyle{S^{2,2}+2[FM]}$\\
\end{tabular}

\vspace{0.2in}

\begin{tabular}{c|c|c|c|c}
$N_4$ & - & + &  - & + \\
\hline
6,0:(0,0) & &&  \\
4,0:(0,0) & 1 & & $\scriptstyle{S^2_a+2[S^{1,1}-AT]}$\\
4,1:(1,0) &  & 1 &&  $\scriptstyle{T_1^{\spit}[4]+[S^{1,0}-AT]}$\\
2,0:(0,0) & 1 & & $\scriptstyle{S^2_a+[DCC]+[S^{1,1}-AT]}$\\
2,1:(1,0) & 1 & & $\scriptstyle{S^2_a+[S^{1,1}-AT]+[S^{1,0}-AT]}$\\
2,2:(0,2) & & 1 & & $\scriptstyle{T_1^{\spit}[4]+2[FM]}$\\
2,2:(2,0) & & 1  & & $\scriptstyle{S^{2,2}+2[S^{1,0}-AT]}$ \\
0,0:(0,0) & 2 & & $\scriptstyle{S^2_a+2[DCC],\quad T_1^a+[DCC]}$\\
0,1:(1,0) & 2 & 1 & $\scriptstyle{S^{2,1}+2[DCC],\ 
  S^2_a+[DCC]+[S^{1,0}-AT]}$ 
& $\scriptstyle{T_1^{\rot}+[S^{1,0}-AT]}$\\
0,2:(2,0) & 1 & & $\scriptstyle{S^{2,1}+[DCC]+[S^{1,0}-AT]}$ \\
0,2:(0,2) & 1 & & $\scriptstyle{S^2_a+[S^{1,1}-AT]+2[FM]}$\\
0,3:(3,0) & & & \\
0,3:(1,2) & & 1 & & $\scriptstyle{S^{2,2}+[S^{1,0}-AT]+2[FM]}$\\
\end{tabular}

\vspace{0.2in}

\begin{tabular}{c|c|c|c|c}
$N_6$ & - & + & - & + \\
\hline
8,0:(0,0) && && \\
6,0:(0,0) &1 && $\scriptstyle{S^2_a+3[S^{1,1}-AT]}$\\
6,1:(1,0) &&1 && $\scriptstyle{T_2^{\spit}[6]+[S^{1,0}-AT]}$\\
4,0:(0,0) &1 & & $\scriptstyle{S^2_a+[DCC]+2[S^{1,1}-AT]}$  \\
4,1:(1,0) &1  & & $\scriptstyle{S^2_a+2[S^{1,1}-AT]+2[S^{1,0}-AT]}$   \\
4,2:(2,0) & & 1 &&  $\scriptstyle{T_1^{\spit}[4]+2[S^{1,0}-AT]}$\\
4,2:(0,2) &  &1 &&  $\scriptstyle{T_2^{\spit}[6]+2[FM]}$\\
2,0:(0,0) &1  & & $\scriptstyle{S^2_a+2[DCC]+[S^{1,1}-AT]}$   \\
2,1:(1,0) &1 & 1 &
$\scriptstyle{S^2_a+[DCC]+[S^{1,1}-AT]+[S^{1,0}-AT]}$
&    $\scriptstyle{T_2^{\spit}[2]+[S^{1,0}-AT]}$\\
2,2:(2,0) &1  &  & $\scriptstyle{S^2_a+[S^{1,1}-AT]+2[S^{1,0}-AT]}$   \\
2,2:(0,2) &1  & & $\scriptstyle{S^2_a+2[S^{1,1}-AT]+2[FM]}$   \\
2,3:(3,0) && 1  & & $\scriptstyle{S^{2,2}+3[S^{1,0}-AT]}$   \\
2,3:(1,2) & &1  & & $\scriptstyle{T_1^{\spit}[4]+1[S^{1,0}-AT]+[FM]}$   \\
0,0:(0,0) &2& & $\scriptstyle{S^2_a+3[DCC],\quad T_1^a+2[DCC]}$\\
0,1:(1,0) &3& & $\scriptstyle{S^{2,1}+3[DCC],\
  S^2_a+2[DCC]+[S^{1,0}-AT],}$\\
&&& $\scriptstyle{T_1^a+[DCC]+[S^{1,0}-AT]}$\\
0,2:(2,0) &2& 1
& $\scriptstyle{S^{2,1}+2[DCC]+[S^{1,0}-AT],}$
 & $\scriptstyle{T_1^{\rot}+2[S^{1,0}-AT]}$ \\
&&& $\scriptstyle{S^2_a+[DCC]+2[S^{1,0}-AT]}$ \\
0,2:(0,2) &1&1 & $\scriptstyle{S^2_a+[DCC]+[S^{1,1}-AT]+2[FM]}$
& $\scriptstyle{T_2^{\spit}[2]+2[FM]}$ \\
0,3:(3,0) &1& & $\scriptstyle{S^{2,1}+[DCC]+2[S^{1,0}-AT]}$\\
0,3:(1,2) &1& & $\scriptstyle{S^2_a+[S^{1,1}-AT]+[S^{1,0}-AT]+2[FM]}$\\
0,4:(4,0) && && \\
0,4:(2,2) &&1 && $\scriptstyle{S^{2,2}+2[S^{1,0}-AT]+2[FM]}$\\
0,4:(0,4) &&1 && $\scriptstyle{T_1^{\spit}[4]+4[FM]}$\\
\end{tabular}

\vspace{0.2in}

\begin{tabular}{c|c|c|c|c}
$N_3$ & - & + & - & + \\
\hline
3,1:(0,1) && 1 && $\scriptstyle{T^{\spit}_1[4]+[FM]}$ \\
1,1:(0,1) &1 & & $\scriptstyle{S^2_a+[S^{1,1}-AT]+[FM]}$ \\
1,2:(1,1) & &1 && $\scriptstyle{S^{2,2}+[S^{1,0}-AT]+[FM]}$ \\
\end{tabular}

\vspace{0.2in}

\bgroup

\begin{tabular}{c|c|c|c|c}
$N_5$ & - & + &  - & + \\
\hline
5,1:(0,1) & & 1 & & $\scriptstyle{T^{\spit}_2[6]+[FM]}$\\
3,1:(0,1) & 1 & & $\scriptstyle{S^2_a+2[S^{1,1}-AT]+[FM]}$\\
3,2:(1,1) &  & 1 & & $\scriptstyle{T_1^{\spit}[4]+[S^{1,0}-AT]+[FM]}$\\
1,1:(0,1) & 1 & 1 & $\scriptstyle{S^2_a+[DCC]+[S^{1,1}-AT]+[FM]}$ & 
$\scriptstyle{T_2^{\spit}[2]+[FM]}$\\
1,2:(1,1) & 1 & &  $\scriptstyle{S^{2}_a+[S^{1,1}-AT]+[S^{1,0}-AT]+[FM]}$ \\
1,3:(2,1) & & 1 & & $\scriptstyle{S^{2,2}+2[S^{1,0}-AT]+[FM]}$ \\
1,3:(0,3) &  & 1 & & $\scriptstyle{T_1^{\spit}[4]+3[FM]}$\\
\end{tabular}

\egroup

\vspace{0.2in}

\begin{tabular}{c|c|c|c|c}
$N_7$ & - & + &  - & + \\
\hline
7,1:(0,1) &&1 & & $\scriptstyle{T_3^{\spit}[8]+[FM]}$ \\
5,1:(0,1) &1 & & $\scriptstyle{S^2_a+3[S^{1,1}-AT]+[FM]}$ \\ 
5,2:(1,1) & &1 & & $\scriptstyle{T_2^{\spit}[6]+[S^{1,0}-AT]+[FM]}$\\
3,1:(0,1) &1  &1 & $\scriptstyle{S^2_a+[DCC]+2[S^{1,1}-AT]+[FM]}$
& $\scriptstyle{T_3^{\spit}[4]+[FM]}$ \\
3,2:(1,1) &1  & & $\scriptstyle{S^2_a+2[S^{1,1}-AT]+[S^{1,0}-AT]+[FM]}$ \\
3,3:(2,1) & &1  && $\scriptstyle{T_1^{\spit}[4] + 2[S^{1,0}-AT]}$ \\
3,3:(0,3) &  &1 && $\scriptstyle{T_2^{\spit}[6]+ 3[FM]}$\\
1,1:(0,1) &1  & & $\scriptstyle{S^2_a+2[DCC]+[S^{1,1}-AT]+[FM]}$ \\
1,2:(1,1) &1 & 1 & $\scriptstyle{S^2_a+[DCC]+[S^{1,1}-AT]+[S^{1,0}-AT]+[FM]}$
& $\scriptstyle{T_2^{\spit}[2] + [S^{1,0}-AT]+[FM]}$ \\
1,3:(2,1) &1  & & $\scriptstyle{S^2_a+[S^{1,1}-AT]+2[S^{1,0}-AT]+[FM]}$ \\
1,3:(0,3) &1  & & $\scriptstyle{S^2_a+2[S^{1,1}-AT]+[FM]}$ \\
1,4:(3,1) & &1 && $\scriptstyle{S^{2,2}+3[S^{1,0}-AT]+[FM]}$ \\
1,4:(1,3) & &1 && $\scriptstyle{T_1^{\spit}[4]+[S^{1,0}-AT]+3[FM]}$\\
\end{tabular}


\bibliographystyle{amsalpha}

\end{document}